\documentclass[
a4paper,
11pt, 
DIV=14,
toc=bibliography, 
headsepline, 
parskip=half-,
abstract=true,
]{scrartcl}

\usepackage[english]{babel}
\usepackage[utf8]{inputenc}
\usepackage[T1]{fontenc}
\usepackage[english,cleanlook]{isodate}
\usepackage{lmodern}
\usepackage[babel=true]{microtype}
\usepackage{amsmath}
\usepackage{mathtools}
\usepackage{amssymb}
\usepackage{amsthm}
\usepackage[markcase=noupper]{scrlayer-scrpage} 
\usepackage{enumitem}
\usepackage{tabularx}
\usepackage{multirow}
\usepackage{tikz} 
\usepackage{tikz-3dplot} 
\usepackage{mathrsfs} 

\setkomafont{pageheadfoot}{\footnotesize}
\automark[section]{section}
\ohead{A.\,Carlotto, M.\,B.\,Schulz, D.\,Wiygul}
\ihead{\headmark}

\providecommand{\R}{\mathbb{R}}
\providecommand{\C}{\mathbb{C}}

\providecommand{\Z}{\mathbb{Z}}
\providecommand{\Sp}{\mathbb{S}}
\providecommand{\B}{\mathbb{B}}
\providecommand{\K}{\mathbb{K}}

\providecommand{\apr}{\mathbb{A}}
\providecommand{\pri}{\mathbb{P}}
\providecommand{\pyr}{\mathbb{Y}}

\providecommand{\refl}{\underline{\mathsf{R}}}
\providecommand{\rot}{\mathsf{R}}
\newcommand{\Wedge}[2]{W_{#1}^{#2}}
\providecommand{\st}{\, :\ }
\providecommand{\lipdom}{\Omega}
\providecommand{\potential}{q}
\providecommand{\robinpotential}{r}
\providecommand{\closure}[1]{\overline{#1}}  
\providecommand{\bdy}[1]{\partial_{#1}}
\providecommand{\dir}{\mathrm{D}}
\providecommand{\neum}{\mathrm{N}}
\providecommand{\dbdy}{\bdy{\dir}}
\providecommand{\nbdy}{\bdy{\neum}}
\providecommand{\rbdy}{\bdy{\mathrm R}}
\providecommand{\ltwosob}[1]{H^{#1}}
\providecommand{\sob}{\ltwosob{1}}
\providecommand{\sobd}[1]{\sob_{#1}}
\providecommand{\conormal}[2]{\eta^{#1}_{#2}}
\providecommand{\twoff}[1]{A^{#1}}
\providecommand{\met}[1]{g^{#1}}
\providecommand{\geuc}{\met{\R^3}}
\providecommand{\gsph}{\met{\Sp^2}}

\providecommand{\indexform}[1]{Q^{#1}}
\providecommand{\hausmeas}[2]{\mathscr{H}^{#1}(#2)}
\providecommand{\hausint}[2]{d\hausmeas{#1}{#2}}
\providecommand{\weakbf}{T}
\providecommand{\eigenvalc}[3]{\lambda_{#2}^{#3}\left(#1\right)}
\providecommand{\eigenval}[2]{\eigenvalc{#1}{#2}{}}
\providecommand{\eigenvalsym}[4]{\eigenvalc{#1}{#2}{#3,#4}}
\providecommand{\eigenspc}[3]{E_{#3}^{#2}(#1)}
\providecommand{\eigensp}[2]{\eigenspc{#1}{#2}{}}
\providecommand{\eigenspsym}[4]{\eigenspc{#1}{#2}{#3,#4}}
\providecommand{\proj}{\pi}
\providecommand{\eigenproj}[2]{\proj_{#1}^{#2}}
\providecommand{\equivind}[1]{\ind_{#1}}
\providecommand{\equivnul}[1]{\nul_{#1}}
\providecommand{\symind}[2]{\ind_{#1}^{#2}}
\providecommand{\symnul}[2]{\nul_{#1}^{#2}}
\providecommand{\grp}{G}
\providecommand{\stabgrp}{H}
\providecommand{\twsthom}{\sigma}
\providecommand{\twstact}[1]{\widehat{#1}}
\providecommand{\invproj}[2]{\proj_{#1,#2}}
\providecommand{\nrmlsgn}[1]{\sgn_{#1}}
\providecommand{\inttext}{\mathrm{int}}
\providecommand{\exttext}{\mathrm{ext}}
\providecommand{\intbdy}{\partial_\inttext}
\providecommand{\extbdy}{\partial_{\exttext}}
\providecommand{\dint}[1]{#1^{\mathrm{D}_\inttext}}
\providecommand{\nint}[1]{#1^{\mathrm{N}_\inttext}}
\providecommand{\conffact}{\rho}

\providecommand{\cycgrp}[1]{\boldsymbol{#1}}

\providecommand{\zerogen}{\Xi^{-\K_0 \cup \K_0}}
\providecommand{\threebc}{\Sigma^{-\K_0 \cup \B^2 \cup \K_0}}
\providecommand{\initzerogen}{\widehat{\Xi}^{-\K_0 \cup \K_0}}
\providecommand{\initthreebc}{\widehat{\Sigma}^{-\K_0 \cup \B^2 \cup \K_0}}
\providecommand{\tow}{\mathbb{M}}
\providecommand{\towdom}{\tow_{\mathrm{fb}}}
\providecommand{\liptowdom}{\widehat{\tow}_\mathrm{fb}}
\providecommand{\two}[1]{#1^\Xi}
\providecommand{\three}[1]{#1^\Sigma}
\providecommand{\catr}{K}
\providecommand{\discr}{B}
\providecommand{\towr}{M}
\providecommand{\towrdom}[1]{M_{\mathrm{fb},#1}}
\providecommand{\liptowrdom}[1]{\widehat{M}_{\mathrm{fb},#1}}
\providecommand{\transr}{\Lambda}
\providecommand{\loc}{\mathrm{loc}}
\providecommand{\resolvent}{R}
\providecommand{\subspace}{\underset{\clap{\scriptsize subspace}}{\subset}}

\DeclarePairedDelimiter\abs{\lvert}{\rvert}
\DeclarePairedDelimiter\nm{\lVert}{\rVert}
\DeclarePairedDelimiter\sk{\langle}{\rangle}
\DeclarePairedDelimiter\interval{]}{[}
\DeclarePairedDelimiter\Interval{[}{[}

\DeclarePairedDelimiter\IntervaL{[}{]}

\DeclareMathOperator{\Ogroup}{O}
\DeclareMathOperator{\Ugroup}{U}
\DeclareMathOperator{\SOgroup}{SO}
\DeclareMathOperator{\sgn}{sgn}

\DeclareMathOperator{\morseindex}{index}
\DeclareMathOperator{\dive}{div}
\DeclareMathOperator{\ind}{ind} 
\DeclareMathOperator{\nul}{nul}
\DeclareMathOperator{\Closure}{Closure}
\DeclareMathOperator{\Span}{Span}

\usepackage[initials]{amsrefs}

\definecolor{Dirichlet}{cmyk}{.5, 1, 0, 0}
\definecolor{Neumann}  {cmyk}{ 0,.5, 1, 0}
\definecolor{Robin}    {cmyk}{ 1, 0,.5, 0}
\usetikzlibrary{calc,plotmarks}
\makeatletter
\tikzset{
    scale plot marks/.is choice,
    scale plot marks/true/.style={},	
    scale plot marks/false/.code={
        \def\pgfuseplotmark##1{\pgftransformresetnontranslations\csname pgf@plot@mark@##1\endcsname}
    },
every mark/.append style={scale plot marks=false},
plus/.style={mark=+,mark size=2.25pt},
vdash/.style={mark=|,mark size=2.25pt},
hdash/.style={mark=-,mark size=2.25pt},
bullet/.style={mark=*,mark size=1.125pt},
declare function={acosh(\x)=ln(\x+sqrt(\x*\x-1));},
}
\makeatother 
\pgfmathsetmacro{\fbmsscale}{0.46}
\pgfmathsetmacro{\unitscale}{\fbmsscale*\textwidth/2cm}

\usepackage[
pdftitle={Spectral estimates for free boundary minimal surfaces via Montiel–Ros partitioning methods},
pdfauthor={Alessandro Carlotto, Mario B. Schulz and David Wiygul}, 
pdfborder={0 0 0},
]{hyperref}

\theoremstyle{plain}
\newtheorem{theorem}{Theorem}[section]
\newtheorem{lemma}[theorem]{Lemma}
\newtheorem{corollary}[theorem]{Corollary} 
\newtheorem{proposition}[theorem]{Proposition}

\theoremstyle{definition}

\newtheorem{example}[theorem]{Example}
 
\theoremstyle{remark}
\newtheorem{remark}[theorem]{Remark}

\numberwithin{equation}{section}

\title{Spectral estimates for free boundary minimal surfaces via Montiel–Ros partitioning methods}
\author{Alessandro Carlotto, Mario B. Schulz, David Wiygul}
\date{
\vspace*{-3ex}
} 

\newcommand\printaddress{{
\setlength{\parindent}{17pt}
\bigskip
\par
\vbox{
{\scshape Alessandro Carlotto}
\newline Universit\`a di Trento, 
Dipartimento di Matematica,
via Sommarive 14, 
38123 Povo di Trento, 
Italy
\newline
\textit{E-mail address:} 
\texttt{alessandro.carlotto@unitn.it}
\par\medskip
{\scshape Mario B. Schulz}
\newline 
University of M\"unster, 
Mathematisches Institut, 
Einsteinstrasse 62,
48149 M\"unster,
Germany
\newline
\textit{E-mail address:} 
\texttt{mario.schulz@uni-muenster.de}
\par\medskip
{\scshape David Wiygul}
\newline Universit\`a di Trento, 
Dipartimento di Matematica,
via Sommarive 14, 
38123 Povo di Trento, 
Italy
\newline
\textit{E-mail address:} 
\texttt{davidjames.wiygul@unitn.it}
\par
}
}}

\begin{document}

\maketitle

\begin{abstract}
We adapt and extend the Montiel–Ros methodology to compact manifolds with boundary, allowing for mixed 
(including oblique) boundary conditions and also accounting for the action of a finite group $\grp$ together with an additional twisting homomorphism $\sigma\colon \grp\to\Ogroup(1)$. 
We then apply this machinery in order to obtain quantitative lower and upper bounds on the growth rate of the Morse index of free boundary minimal surfaces with respect to the topological data (i.\,e. the genus and the number of boundary components) of the surfaces in question. 
In particular, we compute the exact values of the equivariant Morse index and nullity for two infinite families of examples, with respect to their maximal symmetry groups, and thereby derive explicit two-sided linear  bounds when the equivariance constraint is lifted.
\end{abstract}

\section{Introduction}

Despite a profusion
of constructions of
free boundary
minimal surfaces in the Euclidean unit ball $\B^3$
over the course of the past decade
(\cites{FraSch11, FraSch16, KarKokPol14, GirLag21}
via optimization of the first Steklov eigenvalue,
\cites{Ketover16equiv, Ketover16FB, CarFraSch20}
via min-max methods for the area functional,
and
\cites{
FolPacZol17,
KapouleasLiDiscCCdesing, KapouleasMcGrathGenLinDblI,
KapouleasWiygul17,
KapouleasZouCloseToBdy,
CSWnonuniqueness
}
via gluing methods),
many basic questions about the space
of such surfaces remain open.
The reader is referred to
\cites{FraserLectureNotes,
LiSurvey, FranzThesis}
for recent overviews of the field.
In particular, 
so far it is only
for the rotationally symmetric examples,
planar discs through the origin
and critical catenoids,
that the exact value of the Morse index is actually known
(see \cites{Devyver2019, SmithZhou2019, Tran2020}). The present manuscript is the first in a series of works aimed at shedding new light on this fundamental invariant, which (also due to its variational content, and thus to its natural connection with min-max theory, cf.  \cite{MarquesNeves2016, MarquesNeves2017Sur, MarquesNeves2020} and references therein) has acquired great importance within geometric analysis.

Partly motivated by the corresponding conjectures concerning closed minimal hypersurfaces in manifolds of positive Ricci curvature (cf. \cite{AmbCarSha18-SMNconj, Neves2014}), five years ago the first-named author proved with Ambrozio and Sharp a universal lower bound for the index of any free boundary minimal surface in any mean-convex subdomain $\Omega$ of $\R^3$ in terms of the topological data of the surface under consideration. Specifically, it was shown in \cite{AmbCarSha18-Index} that the following estimate holds:
\begin{equation}\label{eq:ACSestimate}
\morseindex(\Sigma)\geq \frac{1}{3}(2g+b-1)
\end{equation}
where $\Sigma$ is any free boundary minimal surface in $\Omega$, and $g$, $b$ denote respectively its genus and the number of its boundary components. This result was then partly complemented by the one of Lima (see \cite[Theorem 4]{Lim17}), that is an affine upper bound with a very large, yet in principle computable numerical constant.
In this article we shall develop a general methodology, building upon  the fundamental work by Montiel and Ros -- as first presented in \cite{MontielRos} -- which allows, among other things, to significantly refine such universal estimates bringing the geometry and symmetry group of the surfaces under consideration into play. This approach, while motivated by our goal to better understand the behaviour of certain infinite families of free boundary minimal surfaces in $\B^3$ (aiming for two-sided bounds in terms of explicit, affine functions of the topological data), turns out to be of independent interest and much wider applicability. 

In more abstract terms, we shall be concerned here with proving effective estimates for (part of) the spectrum of Schr\"odinger-type operators on bounded Lipschitz domains of Riemannian manifolds, combined with mixed boundary conditions, that will be -- on disjoint portions of the boundary in question -- of Dirichlet or Robin (oblique) type. 
Summarizing and oversimplifying things to the extreme, the number of eigenvalues of any such operator \emph{below a given threshold} can be estimated by suitably partitioning the domain into finitely many subdomains, provided one adjoins Dirichlet boundary conditions in the interior boundaries when aiming for lower bounds, and Neumann boundary conditions in the interior boundaries for upper bounds instead. 
We refer the reader to Section \ref{sec:nota} for the setup of our problem together with our standing assumptions, and to the first part of Section \ref{sec:theory} (specifically to Proposition \ref{mr}, and Corollary \ref{cor:MontielRosHuman}) for precise statements.

Often times (yet not always) the partitions mentioned above naturally relate to the underlying symmetries of the problem in question, which is in particular the case for some of the classes of free boundary minimal surfaces in $\B^3$ that have so far been constructed.
With this remark in mind, a peculiar (and, a posteriori, fundamental) feature of our work is the development of the Montiel--Ros methodology in the presence of the action of a group $\grp$ together with an additional twisting homomorphism $\sigma\colon\grp\to\Ogroup(1)$, in the terms explained in Section \ref{subs:GroupAction}. 
This allows, for instance, to explicitly and transparently study how the Morse index of a given free boundary minimal surface depends on the symmetries one imposes, namely to look at the ``functor''
$(\grp,\twsthom)\to\symind{\grp}{\twsthom}(\weakbf)$, where $T$ denotes the index (Jacobi) form of the surface in question. As apparent even from the simplest examples we shall discuss, this perspective turns out to be very natural and effective in tackling the geometric problems we are interested in.

With this approach, \emph{lower} bounds are sometimes relatively cheap to obtain. One way they can derived is from ambient Killing vector fields, once it is shown that the associated (scalar-valued) Jacobi field on the surface under consideration vanishes along the (interior) boundary of any domain of the chosen partition,
which in practice amounts to suitably \emph{designing the partition and picking the Killing field} given the geometry of the problem.
We present one simple yet paradigmatic such result in Proposition \ref{indBelowPyr}, which concerns free boundary minimal surfaces with pyramidal or prismatic symmetry in $\B^3$. Instead, \emph{upper} bounds are often a lot harder to obtain
and shall typically rely on finer information than the sole symmetries of the scene one deals with. 
Said otherwise, one needs to know \emph{how} (i.\,e. by which method) the surface under study has been obtained. 

We will develop here a detailed analysis of the Morse index of the two families of free boundary minimal surfaces we constructed in our recent, previous work \cite{CSWnonuniqueness}. Very briefly, using gluing methods of essentially PDE-theoretic character, we obtained there a sequence $\threebc_{m}$ of surfaces having genus $m$, three boundary components and antiprismatic symmetry group $\apr_{m+1}$, and a sequence $\zerogen_n$ of surfaces having genus zero, $n+2$ boundary components and prismatic symmetry group $\pri_n$.
As we described at length in Section 7 therein, with data (cf. Table 2 and Table 3) and heuristics, numerical simulations for the Morse index of the surfaces in the former sequence display a seemingly ``erratic'' behaviour, as such values do not align on the graph of any affine function, nor seem to exhibit any obvious periodic pattern. 
This is a rather unexpected behaviour (by comparison e.\,g. with other families of examples, say in the round three-dimensional sphere, see \cite{KapouleasWiygulIndex}), which obviously calls for a careful study that we carry through
in Section \ref{sec:CSW} of the present article. In particular, we establish the following statement:

\begin{theorem}[Index estimates for $\threebc_m$ and $\zerogen_n$]
\label{cswIndexBounds}
There exist $m_0, n_0>0$ such that for all integers
$m>m_0$ and $n>n_0$ the
Morse index and nullity of the free boundary minimal surfaces
$\threebc_m, \zerogen_n \subset \B^3$
satisfy the bounds
\begin{align*}
2m+1&\leq\ind(\threebc_m), 
&
\ind(\threebc_m)+\nul(\threebc_m)&\leq 12m+12,
\\
2n+2&\leq\ind(\zerogen_n), 
&
\ind(\zerogen_n)+\nul(\zerogen_n)&\leq 8n.
\end{align*}
\end{theorem}

In fact, the upper bound in this ``absolute estimate'' follows quite easily by combining the ``relative estimate'', associated to the equivariant Morse index of these surfaces (with respect to their respective \emph{maximal} symmetry groups) with the aforementioned Proposition \ref{mr}. The next statement thus pertains to such equivariant bounds, for which we do obtain equality, thus settling part of Conjecture~7.7 (iv) and Conjecture~7.9 (iv) of \cite{CSWnonuniqueness}. We stress that neither family is constructed variationally, 
and thus there is actually no cheap index bound one can extract from the design methodology itself; on the contrary, this statement indicates \emph{a posteriori} that the families of surfaces in question may in principle be constructed (even in a non-asymptotic regime) by means of min-max schemes generated by $2$-parameter sweepouts, modulo the well-known problem of fully controlling the topology in the process (cf. \cite{CarFraSch20}).

\begin{theorem}[Equivariant index and nullity of $\threebc_m$ and $\zerogen_n$]
\label{cswEquivariantIndexAndNullity}
There exist $m_0, n_0>0$ such that for all integers
$m>m_0$ and $n>n_0$ the equivariant 
Morse index and nullity of the free boundary minimal surfaces
$\threebc_{m}, \zerogen_n \subset \B^3$
satisfy
\begin{align*}
\equivind{\apr_{m+1}}(\threebc_{m})&=2, & \equivnul{\apr_{m+1}}(\threebc_{m})&=0,
\\
\equivind{\pri_n}(\zerogen_n)&=2, & \equivnul{\pri_n}(\zerogen_n)&=0.
\end{align*}
\end{theorem}

The main idea behind the proof of these results, or -- more precisely -- for the upper bounds can only be explained by recalling, in a few words, how the surfaces in question have been constructed. 
Following the general methodology of
\cite{KapouleasEuclideanDesing}, one first considers a singular configuration, that is a formal union of minimal surfaces in $\B^3$ (not necessarily free boundary), then its regularization -- which needs the use of (wrapped) periodic minimal surfaces in $\R^3$,  to desingularize near the divisors, and controlled interpolation processes between the building blocks in play -- and, thirdly and finally, the perturbation of such configurations to exact minimality (at least for \emph{some} values of the parameters), while also ensuring proper embeddedness and accommodating the free boundary condition. 
Here we first get a complete understanding of the index and nullities of the building blocks, for the concrete cases under consideration in Section \ref{sec:CSW}. 
In somewhat more detail,
the analysis of the Karcher--Scherk towers (the periodic building blocks employed in either construction) exploits, in a substantial fashion, the use of the Gauss map, which allows one
to rephrase the initial geometric question into
one for the spectrum of simple elliptic operators of the form $\Delta_{\gsph}+2$ on suitable (typically singular, i.\,e. spherical triangles, wedges or lunes) subdomains of round $\Sp^2$, with mixed boundary conditions, and possibly subject to additional symmetry requirements. 
The analysis of the other building blocks -- disks and asymmetric catenoidal annuli -- is more direct, although, in the latter case, trickier than it may first look (see e.\,g. Lemma \ref{K0IndexAndNullity}).

Once that preliminary analysis is done, we then prove that, corresponding to the (local) geometric convergence results (that are implied by the very gluing methodology) there are robust spectral convergence results that serve our scopes. However, a general challenge in the process is that gluing constructions typically have \emph{transition regions} where different scales interact with each another: in our constructions of the sequences $\threebc_m$ and $\zerogen_n$ such regions occur between the catenoidal annuli $\K_0$ (as well as the disk $\B^2$ in the former case) and the wrapped Karcher--Scherk towers, roughly at distances between $m^{-1}$ and $m^{-1/2}$ (respectively $n^{-1}$ and $n^{-1/2}$) from the equatorial $\Sp^1$. 
As a result,
we need to deal with delicate scale-picking arguments, an \emph{ad hoc} study of the geometry of such regions (cf. Lemma \ref{transitionEstimates}) and -- most importantly -- prove the corresponding uniform bounds for eigenvalues and eigenfunctions  (collected in Lemma \ref{UniformBoundsOnTower}), which allow to rule out pathologic concentration phenomena, thereby leading to the desired conclusions.

\paragraph{Acknowledgements.} 
The authors wish to express their sincere gratitude to Giada Franz for a number of conversations on themes related to those object of the present manuscript.
This project has received funding from the European Research Council (ERC) under the European Union’s Horizon 2020 research and innovation programme (grant agreement No. 947923). 
The research of M.\,S. was funded by the Deutsche Forschungsgemeinschaft (DFG, German Research Foundation) under Germany's Excellence Strategy EXC 2044 -- 390685587, Mathematics M\"unster: Dynamics--Geometry--Structure, and the Collaborative Research Centre CRC 1442, Geometry: Deformations and Rigidity. 
Part of this article was finalized while A.\,C. was visiting the ETH-FIM, whose support and excellent working conditions are gratefully acknowledged.

\section{Notation and standing assumptions}
\label{sec:nota}

\subsection{Boundary value problems for Schr\"{o}dinger operators on Lipschitz domains}
Let $\lipdom$ be a Lipschitz domain
of a smooth, compact $d$-dimensional manifold $M$
with (possibly empty) boundary $\partial M$,
by which we mean here a non-empty, open subset of $M$
whose boundary is everywhere locally representable
as the graph of a Lipschitz function.
We do not require -- at least in general -- $\Omega$ to be connected, and we admit the case $\closure{\lipdom}=M$ (where $\closure{\lipdom}$ denotes the closure
of $\lipdom$ in $M$), when of course
$\partial\lipdom=\partial M$, the boundary of the ambient manifold in question. Throughout this article we will in fact assume $d \geq 2$.

We are going to study the spectrum of a given Schr\"{o}dinger operator on $\lipdom$ subject to boundary conditions and, sometimes, symmetry constraints. 
Such symmetry constraints will be encoded in terms of equivariance with respect to a certain group action, which we shall specify at due place.

The Schr\"{o}dinger operator
\begin{equation*}
\Delta_g+\potential
\end{equation*}
is determined
by the data of a given smooth Riemannian metric $g$ on
$\closure{\lipdom}$
and a given smooth (i.\,e. $C^\infty$) function $\potential \colon \closure{\lipdom} \to \R$. 
To avoid ambiguities, we remark here
that
a function (or tensor field) on $\closure{\lipdom}$
smooth if it is the restriction of a smooth tensor field on $M$ or -- equivalently -- on a relatively open set containing $\closure{\Omega}$.

The boundary conditions are specified by another smooth function
$\robinpotential \colon \closure{\lipdom} \to \R$
and a decomposition
\begin{equation}
\label{bdydecomp}
\partial \lipdom
=
\closure{\dbdy \lipdom}
  \cup 
  \closure{\nbdy \lipdom}
  \cup
  \closure{\rbdy \lipdom}
\end{equation}
where the sets on the right-hand side are
the closures of pairwise disjoint open subsets
$\dbdy\lipdom$, $\nbdy\lipdom$, and $\rbdy\lipdom$
of $\partial \lipdom$. 

Somewhat more specifically,
we will consider the spectrum of the operator $\Delta_g + q$ subject to the Dirichlet, Neumann, and Robin conditions
\begin{align}\label{eq:BoundaryConditions}
\left\{
\begin{aligned}
u&=0 
&&\text{ on }
  \dbdy \lipdom,
\\
du(\conormal{\lipdom}{g})&=0
  &&\text{ on }
  \nbdy \lipdom,
\\
du(\conormal{\lipdom}{g})&=\robinpotential u
  &&\text{ on }
  \rbdy \lipdom,
\end{aligned}
\right.
\end{align}
where $\conormal{\lipdom}{g}$
is the almost-everywhere defined
outward unit normal induced by $g$ on $\partial \lipdom$.

It is obviously the case that the Neumann boundary conditions can be regarded as a special case of their inhomogenous counterpart, however it is convenient -- somewhat artificially -- to distinguish them in view of the later applications we have in mind, to the study of the Morse index of free boundary minimal surfaces.

\subsection{Sobolev spaces and traces}
To pose the problem precisely
we introduce the Sobolev space
$\sob(\lipdom,g)$
consisting of all real-valued functions in $L^2(\lipdom,g)$
which have a weak $g$-gradient whose pointwise $g$-norm
is also in $L^2(\lipdom,g)$;
then $\sob(\lipdom,g)$ is a Hilbert space
equipped with the inner product
\begin{equation*}
\sk{u,v}_{\sob(\lipdom,g)}
\vcentcolon=
\int_\lipdom
  \bigl(
    uv
    +
    g(\nabla_g u, \nabla_g v)
  \bigr)
  \,
  \hausint{d}{g},
\end{equation*}
integrating with respect to the $d$-dimensional
Hausdorff measure induced by $g$.
(We say a function $u \in L^1_{\loc}(\lipdom,g)$
has a weak $g$-gradient $\nabla_g u$
if
$\nabla_g u$ is a measurable vector field
on $\lipdom$ with pointwise $g$ norm
in $L^1_{\loc}(\lipdom,g)$
and
$
 \int_\lipdom g(X, \nabla_g u)
   \, \hausint{d}{g}
 =
 -\int_\lipdom u \, \dive_g X
   \, \hausint{d}{g}
$
for every smooth vector field $X$
on $\lipdom$ of relatively compact support,
where $\dive_g X$ is the $g$ divergence of $X$;
$\nabla_g u$ is uniquely defined whenever it exists,
modulo vector fields vanishing almost everywhere.)

Under our assumptions on $\partial \lipdom$
we have a bounded trace map
$\sob(\lipdom,g) \to L^2(\partial \lipdom, g)$,
extending the restriction map
$
 C^1(\closure{\lipdom})
 \to
 C^0(\partial\lipdom)
$.
(The Hilbert space
$L^2(\partial \lipdom, g)$
is defined using either
the $(d-1)$-dimensional
Hausdorff measure $\hausmeas{d-1}{g}$
induced by $g$
or, equivalently,
the almost-everywhere
defined volume density
induced by $g$ on $\partial \lipdom$.)
In fact, we have not only boundedness of this map
but also the stronger inequality
\begin{equation}
\label{interpolatingTraceInequality}
\nm{u|_{\partial \lipdom}}_{L^2(\partial \lipdom,g)}
\leq
C(\lipdom,g)
  \Bigl(
    \epsilon\nm{u}_{\sob(\lipdom,g)}
    +C(\epsilon)\nm{u}_{L^2(\lipdom,g)}
  \Bigr)
\end{equation}
for all $u \in \sob(\lipdom,g)$, 
all $\epsilon>0$,
some $C(\lipdom,g)$
independent of $u$ and $\epsilon$,
and some $C(\epsilon)$
independent of $u$
and $(\lipdom,g)$.
(This can be deduced, for example,
by inspecting the proof
of Theorem 4.6
in \cite{EvansGariepy}:
specifically,
we can apply
the Cauchy--Schwarz inequality
(weighting with $\epsilon$, as standard)
to the inequality immediately
above the line labeled
($\star\star\star$)
on page 158 of the preceding reference,
whose treatment of Lipschitz domains
in Euclidean space
is readily adapted to our setting.)

For each $\mathrm{C}\in \{\mathrm{D},\mathrm{N},\mathrm{R}\}$,
indicating one of the boundary conditions we wish to impose, 
by composing the preceding trace map
with the restriction
$
 L^2(\partial\lipdom,g)
 \to
 L^2(\bdy{\mathrm{C}}\lipdom,g)
$, since $\bdy{\mathrm{C}} \lipdom$ is open 
in $\partial \lipdom$,
we also get a trace map
$
 \cdot|_{\bdy{\mathrm{C}}}
 \colon
 \sob(\lipdom,g)
 \to
 L^2(\bdy{\mathrm{C}} \lipdom,g)
$.
In practice we will consider traces
on just $\dbdy\lipdom$
and $\rbdy\lipdom$.
Considering the condition on $\dbdy\lipdom$ we will then define
\begin{equation*}
\sobd{\dbdy \lipdom}(\lipdom,g)
\vcentcolon=
\{
  u \in \sob(\lipdom,g)
  \st
  u|_{\dbdy \lipdom}=0
\},
\end{equation*}
that is obviously to be understood in the sense of traces, in the terms we just described,
and we remark that
\eqref{interpolatingTraceInequality}
also clearly holds
with $\partial\lipdom$ on the left-hand side
replaced by $\rbdy\lipdom$
(or by $\dbdy\lipdom$ or $\nbdy\lipdom$,
but we have no need of the inequality in these cases).

\subsection{Bilinear forms and their eigenvalues and eigenspaces}\label{subsec:BasicSpectral}
Corresponding to the above data we define the bilinear form $\weakbf=\weakbf[\lipdom,g,\potential,\robinpotential,\dbdy\lipdom,\nbdy\lipdom,\rbdy\lipdom]$
by 
\begin{equation}
\label{bilinear_form_def}
\begin{aligned}
\weakbf\colon\sobd{\dbdy \lipdom}(\lipdom,g)
\times\sobd{\dbdy \lipdom}(\lipdom,g)
&\to\R
\\
(u,v)&\mapsto
  \int_\lipdom
    \Bigl(
      g(\nabla_g u, \nabla_g v)
      -\potential u v
    \Bigr)
    \, \hausint{d}{g}
  -\int_{\rbdy \lipdom}
    \robinpotential u v
    \, \hausint{d-1}{g}.
\end{aligned}
\end{equation}
Then $\weakbf$ is symmetric, bounded, and coercive as encoded in the following three equations respectively:
\begin{alignat}{2}\notag
\forall u,v &\in \sobd{\dbdy\lipdom}(\lipdom,g)\quad &
\weakbf(u,v)&=\weakbf(v,u),
\\[1ex]\label{eq:bounded}
\forall u &\in \sobd{\dbdy\lipdom}(\lipdom,g) &
\weakbf(u,u)&\leq\bigl(1+C(\lipdom,g,\potential, \robinpotential)\bigr)\nm{u}_{\sob(\lipdom,g)}^2,
\\\label{eq:coercive}
\forall u &\in \sobd{\dbdy\lipdom}(\lipdom,g) &
\weakbf(u,u)&\geq \frac{1}{2}\nm{u}_{\sob(\lipdom,g)}^2
  -C(\lipdom,g,\potential, \robinpotential)  \nm{u}_{L^2(\lipdom,g)}^2,
\end{alignat}
where, for \eqref{eq:bounded} and \eqref{eq:coercive}, one can take
$C(\lipdom,g,\potential,\robinpotential)
  =
  \nm{\potential}_{C^0(\closure{\lipdom})}
  +C(\lipdom,g)
    \nm{\robinpotential}_{C^0(\closure{\rbdy\lipdom})}$,
thanks to the trace inequality \eqref{interpolatingTraceInequality}.
From these three properties
and the Riesz representation theorem
for Hilbert spaces
it follows that
for some constant
$
 \Lambda
 =
 \Lambda(\lipdom,g,\potential,\robinpotential)>0
$
there exists a linear map
$\resolvent{}
 \colon
 L^2(\lipdom,g)
 \to
 \sobd{\dbdy\lipdom}(\lipdom,g)
$
such that
$
 \weakbf(\resolvent f,v)
 +\Lambda \sk{\iota \resolvent f,\iota v}_{L^2(\lipdom,g)}
 =
 \sk{f,\iota v}_{L^2(\lipdom,g)}
$
for all functions
$f \in L^2(\lipdom,g)$
and
$v \in \sobd{\dbdy\lipdom}(\lipdom,g)$, where we have introduced the inclusion map
$
 \iota
 \colon
 \sobd{\dbdy \lipdom}(\lipdom,g)
 \to
 L^2(\lipdom,g).
$

(Of course, if $f$ is smooth
then standard elliptic \emph{interior} regularity results
ensures that
$u$ is as well smooth on $\Omega$
and there satisfies the equation
$-(\Delta_g + \potential -\Lambda)u=f$ in a classical pointwise sense.)
Since the inclusion
$
 \sob(\lipdom,g)
 \hookrightarrow
 L^2(\lipdom,g)
$
is compact
(see for example
Section 7 of Chapter~4 of \cite{TaylorPDE1})
and of course the inclusion of the closed subspace
$
 \sobd{\dbdy \lipdom}(\lipdom,g)
 \hookrightarrow
 \sob(\lipdom,g)
$
is bounded, the aforementioned maps
$\iota\colon\sobd{\dbdy \lipdom}(\lipdom,g)\to L^2(\lipdom,g)$
and the composite $\iota \resolvent\colon L^2(\lipdom,g)\to L^2(\lipdom,g)$ are also both compact operators.
Furthermore, 
to confirm that $\iota\resolvent$ is symmetric we simply note that (by appealing to the equation defining the operator $R$, with $Rf_1$ and $Rf_2$ in place of $v$)
\begin{align*}
\sk{f_2,\iota\resolvent f_1}_{L^2(\lipdom,g)}
&=\weakbf(\resolvent f_2, \resolvent f_1)+\Lambda \sk{\iota\resolvent f_2,\iota\resolvent f_1 }_{L^2(\lipdom,g)}
\\
&=\weakbf(\resolvent f_1, \resolvent f_2)+\Lambda \sk{\iota\resolvent f_1,\iota\resolvent f_2}_{L^2(\lipdom,g)}
=\sk{f_1,\iota\resolvent f_2}_{L^2(\lipdom,g)}
\end{align*}
for all $f_1, f_2 \in L^2(\lipdom,g)$. 
That being clarified, to improve readability we will from now on refrain from explicitly indicating the inclusion map $\iota$ in our equations.

With slight abuse of language, in the setting above we call $\lambda\in\R$ an \emph{eigenvalue} of
$\weakbf$
if there exists a non-zero
$u \in \sobd{\dbdy \lipdom}(\lipdom,g)$
such that
\begin{equation}
\label{eigendef}
\forall v \in \sobd{\dbdy \lipdom}(\lipdom,g)
\quad
\weakbf(u,v)
=
\lambda
  \sk{u,v}_{L^2(\lipdom,g)},
\end{equation}
and we call any such $u$ an \emph{eigenfunction} of $\weakbf$ with eigenvalue $\lambda$.
(We caution that the notions
of eigenfunctions
and eigenvalues depend
not only on $\weakbf$
but also on the underlying
metric $g$;
for the sake of convenience
we choose to suppress
the latter dependence
from our notation.)

Hence, as a consequence of the key facts we presented before this definition, one can prove by well-known arguments the existence of a discrete spectrum for the ``shifted'' elliptic operator $(\Delta_g+q)-\Lambda$ subject to the very same boundary conditions \eqref{eq:BoundaryConditions}. As a straightforward corollary, by accounting for the shift, we obtain
the following conclusions for $T$:
\begin{itemize} 
\item
the set of eigenvalues of $\weakbf$
is discrete in $\R$ and bounded below,
\item
for each eigenvalue of $\weakbf$
the corresponding eigenspace
has finite dimension,
\item
there exists an Hilbertian basis
$\{e_j\}_{j=1}^\infty$
for $L^2(\lipdom,g)$
consisting of eigenfunctions of $\weakbf$,
\item
and $\{e_j\}_{j=1}^\infty$
has dense span in
$\sobd{\dbdy\lipdom}(\lipdom,g)$.
\end{itemize}
(To avoid ambiguities, we remark that the phrase \emph{Hilbertian basis} refers to a countable, complete orthonormal system for the Hilbert space in question.)
For each integer $i \geq 1$
we write $\eigenval{\weakbf}{i}$
for the
$i\textsuperscript{th}$
eigenvalue of $\weakbf$
(listed with repetitions
in non-decreasing order, in the usual fashion).
There holds
the usual min-max characterization
\begin{equation}\label{eq:MinMaxEigenv}
\eigenval{\weakbf}{i}
=
\min
  \left\{
    \max
      \biggl\{
        \frac{\weakbf(w,w)}
        {\nm{w}_{L^2(\lipdom,g)}^2}
        \st  
        0 \neq w \in W
    \biggr\}
    \st
    W
    \subspace
    \sobd{\dbdy \lipdom}(\lipdom,g),
    ~
    \dim W = i
  \right\}.
\end{equation}
Next, for any $t \in \R$ we let
$\eigensp{\weakbf}{=t}$
denote the (possibly trivial) linear span,
in $\sobd{\dbdy\lipdom}(\lipdom,g)$,
of the eigenfunctions of $\weakbf$
with eigenvalue $t$,
and, more generally, for any $t \in \R$
and any binary relation $\sim$ on $\R$
(in practice $<$, $\leq$,
$>$, $\geq$, or $=$)
we set
\begin{equation*}
\eigensp{\weakbf}{\sim t}
\vcentcolon=
\Closure_{L^2(\lipdom,g)}
  \left(
  \Span
    \biggl(
      \bigcup_{s \sim t}
      \eigensp{\weakbf}{=s}      
    \biggr)
  \right)
\end{equation*}
and we denote the corresponding orthogonal projection by
\begin{equation*}
\eigenproj{\weakbf}{\sim t}
\colon
L^2(\lipdom,g)
\to
\eigensp{\weakbf}{\sim t}.
\end{equation*}
That is, the space
$\eigensp{\weakbf}{\sim t}$
has been defined
to be the closure in
$L^2(\lipdom,g)$
of the span of all eigenfunctions
of $\weakbf$ having eigenvalue
$\lambda$ such that $\lambda \sim t$.
Of course
$\eigensp{\weakbf}{\sim t}$
is a subspace of
$\sobd{\dbdy\lipdom}(\lipdom,g)$
-- in particular -- whenever the former has finite dimension.
Taking $\sim$ to be equality
clearly reproduces
the originally defined space
$\eigensp{\weakbf}{=t}$.

For future use
observe that the above spectral theorem
for $\weakbf$ implies
\begin{equation}
\label{basic_spectral_facts}
\begin{aligned}
(\eigensp{\weakbf}{\sim t})^{\perp_{L^2(\lipdom,g)}}
  = 
  \eigensp{\weakbf}{\not \sim t},
&\qquad
\eigensp{\weakbf}{< t}
  \subspace
  \eigensp{\weakbf}{\leq t}
  \subspace
  \sobd{\dbdy\lipdom}(\lipdom,g),
\\
\forall u \in
  \eigensp{\weakbf}{\sim t}
    \cap
    \sobd{\dbdy\lipdom}(\lipdom,g)
  \quad
  \weakbf(u,u)
    &\sim t
    \nm{u}_{L^2(\lipdom,g)}^2
  \quad
  \mbox{for $\sim$ any one of 
        $<,\leq,>,\geq$},
\end{aligned}
\end{equation}
and
\[
\forall u \in
  \sobd{\dbdy\lipdom}(\lipdom,g)
    \cap
    \left(
      \eigensp{\weakbf}{\leq t}
      \cup
      \eigensp{\weakbf}{\geq t}
    \right) 
\quad
   \weakbf(u,u)=t\nm{u}_{L^2(\lipdom,g)}^2 
~\Rightarrow~
  u \in \eigensp{\weakbf}{= t},
\]
throughout which $t$ is any real number
(not necessarily an eigenvalue of $\weakbf$)
and where in the first equality of \eqref{basic_spectral_facts} 
$\sim$ is any relation on $\R$
and $\not \sim$ its negation
(so that $\{s \not \sim t\}=\R \setminus \{s \sim t\}$
for any $t \in \R$).

\paragraph{Index and nullity.} 
In the setting above, and under the corresponding standing assumption, we shall define the non-negative integers 
\begin{equation*}
\ind(\weakbf)
  \vcentcolon=
  \dim \eigensp{\weakbf}{<0}
\quad \text{ and } \quad
\nul(\weakbf)
  \vcentcolon=
  \dim \eigensp{\weakbf}{=0},
\end{equation*}
called, respectively,
the \emph{index} and \emph{nullity} of $\weakbf$. Such invariants will be of primary interest in our applications.

\subsection{Group actions}\label{subs:GroupAction}
Let $\grp$ be a finite group
of smooth diffeomorphisms
of $M$, each restricting to an isometry
of $(\closure{\lipdom},g)$.
Then, as for any group of diffeomorphism of $\lipdom$,
we have the standard (left) action of $\grp$
on functions on $\lipdom$ via pullback:
\begin{equation*}
(\phi,u)
\mapsto
u \circ \phi^{-1}
=
\phi^{-1*}u
\quad
\text{ for all }
\phi \in \grp, ~
u \colon \lipdom \to \R.
\end{equation*}
We say that a function $u$
is $\grp$-invariant if
it is invariant under this action:
equivalently
$u \circ \phi = u$
for all $\phi \in \grp$.

We can also twist this action by orthogonal transformations
on the fiber $\R$:
given in addition to $\grp$
a group homomorphism
$
 \twsthom
 \colon
 \grp
 \to
 \Ogroup(1)=\{-1,1\}
$,
we define the action
\begin{equation*}
(\phi,u)
\mapsto
\twsthom(\phi)(u \circ \phi^{-1})
=
\twsthom(\phi)\phi^{-1*}u
\quad
\mbox{for all }
\phi \in \grp, ~
u \colon \lipdom \to \R,
\end{equation*}
and we call a function
$(\grp,\twsthom)$-invariant
if it is invariant under this action.
Obviously the above standard action 
$(\phi,u) \mapsto u \circ \phi^{-1}$
is recovered by taking
the trivial homomorphism
$\twsthom \equiv 1$.
We also comment that one could of course
replace $\R$ by $\C$
and correspondingly
$\Ogroup(1)$ by $\Ugroup(1)$
(and in the preceding sections instead
work with Sobolev spaces over $\C$)
though we restrict attention
to real-valued functions in this article.

Since, by virtue of our initial requirement, $\grp$ is a group of isometries
of $(\closure{\lipdom},g)$,
the above twisted action yields a unitary representation of $G$ in $L^2(\lipdom,g)$, i.\,e. a group homomorphism
\begin{equation}
\label{twistedrep}
\begin{aligned}
\twstact{\twsthom}
  \colon
  \grp
  &\to
  \Ogroup\bigl(L^2(\lipdom,g)\bigr) 
\\
  \phi
  &\mapsto
  \twsthom(\phi)\phi^{-1*}
\end{aligned}
\end{equation}
whose target are the global isometries
of $L^2(\lipdom,g)$; we note that the same conclusions hold true with $\sob(\lipdom,g)$ in place of $L^2(\lipdom,g)$.
The corresponding subspaces of $(\grp,\twsthom)$-invariant functions, in $L^2(\lipdom,g)$ or $\sob(\lipdom,g)$, are readily checked to be closed, and thus Hilbert spaces themselves. That said, we define
the orthogonal projection
\begin{equation}
\label{invariantprojection}
\begin{aligned}
\invproj{\grp}{\twsthom}
  \colon
  L^2(\lipdom,g)
  &\to
  L^2(\lipdom,g)
\\
u
  &\mapsto
  \frac{1}{\abs{\grp}}
  \sum_{\phi \in \grp}
  \twstact{\twsthom}(\phi)u.
\end{aligned}
\end{equation}
Here $\abs{\grp}$ is the order of $\grp$, which -- we recall -- is
assumed throughout to be finite. The image of $L^2(\lipdom,g)$
under $\invproj{\grp}{\twsthom}$
thus consists of $(\grp,\twsthom)$-invariant
functions.

\begin{remark}
One could lift the finiteness assumption,
say by allowing $\grp$ to be a compact Lie group,
requiring $\twsthom$ to be continuous,
and replacing the finite average in
\eqref{invariantprojection}
with the average over $\grp$
with respect to its Haar measure
(which reduces to the former
for finite $\grp$). However, with a view towards our later applications, in this article we will content ourselves with the finiteness assumption, which allows for a lighter exposition.
\end{remark}

Henceforth we make the additional assumptions that
$\grp$ \emph{globally} (i.\,e.  as sets) preserves each of
$\dbdy\lipdom$,
$\nbdy\lipdom$,
and
$\rbdy\lipdom$,
and that
$\potential$
and
$\robinpotential$
are both $\grp$-invariant.
Each element of
$\twstact{\twsthom}(\grp)$
then preserves also
$\sobd{\dbdy\lipdom}(\lipdom,g)$
and the bilinear form $\weakbf$,
and the projection
$\invproj{\grp}{\twsthom}$
commutes with the projection
$\eigenproj{\weakbf}{\sim t}$,
for any $t \in \R$
and binary relation $\sim$ on $\R$ (as above).
In particular
$\invproj{\grp}{\twsthom}$
preserves each eigenspace
$\eigensp{\weakbf}{=t}$
of $\weakbf$,
and more generally the space
\begin{equation}
\label{inveigensp}
\eigenspsym{\weakbf}{\sim t}{\grp}{\twsthom}
\vcentcolon=
\invproj{\grp}{\twsthom}
  (\eigensp{\weakbf}{\sim t})
\end{equation}
is a subspace of $\eigensp{\weakbf}{\sim t}$.

For each integer $i \geq 1$ we can then define
$\eigenvalsym{\weakbf}{i}{\grp}{\twsthom}$,
the $i$\textsuperscript{th}
$(\grp,\twsthom)$-eigenvalue of $\weakbf$,
to be the $i$\textsuperscript{th}
eigenvalue of $\weakbf$ having a $(\grp,\twsthom)$-invariant
eigenfunction (by definition non-zero),
counting with multiplicity as before; equivalently one can work with spaces of $(G,\sigma)$-invariant functions and derive the analogous conclusions as in Subsection \ref{subsec:BasicSpectral} directly in that setting.

\begin{remark}\label{rem:SpectralPathology}
We explicitly note, for the sake of completeness, that under no additional assumptions on the group $G$ and the homomorphism $\sigma$ it is possible that the space of $(G,\sigma)$-invariant functions be finite dimensional (possibly even of dimension zero). 
This type of phenomenon happens, for instance, when every point of the manifold
$M$ is a fixed point of an isometry on which $\sigma$ takes the value $-1$.
In this case, all conclusions listed above still hold true, but need to be understood with a bit of care: the corresponding sequence of eigenvalues
$\eigenvalsym{\weakbf}{1}{\grp}{\twsthom}\leq \eigenvalsym{\weakbf}{2}{\grp}{\twsthom}\leq \ldots$
will in fact just be a finite sequence, consisting say of $I(\grp,\twsthom)$ elements, counted with multiplicity as usual; we shall formally convene that $\eigenvalsym{\weakbf}{i}{\grp}{\twsthom}=+\infty$ for $i>I(\grp,\twsthom)$. 
That being said, we also remark that this phenomenon patently does not occur for the Jacobi form of the two sequences of free boundary minimal surfaces we examine in Sections \ref{sec:FBMSgen} and \ref{sec:CSW}.
\end{remark}

In this equivariant framework we still have the corresponding min-max characterization
\begin{equation}
\label{symminmax}
\eigenvalsym{\weakbf}{i}{\grp}{\twsthom}
=
\min
  \left\{
    \max
      \biggl\{
        \frac{\weakbf(w,w)}
        {\nm{w}_{L^2(\lipdom,g)}^2}
\st
0\neq w \in W
    \biggr\}
\st
    W
    \subspace
    \invproj{\grp}{\twsthom}\bigl(\sobd{\dbdy \lipdom}(\lipdom,g)\bigr),
    ~
    \dim W = i
\right\}.
\end{equation}

We also define the $(\grp,\twsthom)$-index and $(\grp,\twsthom)$-nullity
\begin{align*}
\symind{\grp}{\twsthom}(\weakbf)
  &\vcentcolon=
  \dim \eigenspsym{\weakbf}{<0}{\grp}{\twsthom}
\quad\text{ and }\quad
\symnul{\grp}{\twsthom}(\weakbf)
 \vcentcolon=
  \dim \eigenspsym{\weakbf}{=0}{\grp}{\twsthom}
\end{align*}
of $\weakbf$.
Obviously we can recover
$\eigensp{\weakbf}{\sim t}$,
$\eigenval{\weakbf}{i}$,
and the standard index and nullity
by taking $\grp$ to be the trivial group. As mentioned in the introduction, we reiterate
that it is one of the goals of the present article
to study, \emph{for fixed $g$ and $T$}, how these numbers (index $\symind{\grp}{\twsthom}(\weakbf)$ and nullity $\symnul{\grp}{\twsthom}(\weakbf)$) depend on $G$ and $\sigma$. 

\paragraph{Terminology.}
For the sake of brevity, we shall employ the phrase \emph{admissible data}
to denote any tuple
\(
 (
 \lipdom,g,\potential,\robinpotential,
 \dbdy\lipdom,\nbdy\lipdom,\rbdy\lipdom,
 \grp,\twsthom
 )
\)
satisfying all the standing assumptions presented up to now.
We digress briefly to highlight two important special cases, which warrant additional notation.

\begin{example}[Actions of order-$2$ groups] \label{ex:Group2El}
When $\abs{G}=2$,
there are precisely two homomorphisms
$\grp \to \Ogroup(1)$. Considering such homomorphisms, and the corresponding $(G,\sigma)$-invariant functions,
we may define
$\grp$-even or $\grp$-odd functions. Hence, we may call
$\symind{\grp}{+}$ and $\symind{\grp}{-}$
the $\grp$-even and $\grp$-odd index,
and likewise for the nullity.
Clearly, we always have
\begin{equation}
\label{evenOddDecomp}
\begin{cases}
\ind(\weakbf)
  =
  \symind{\grp}{+}(\weakbf)
  +
  \symind{\grp}{-}(\weakbf),
\\[1ex]
\nul(\weakbf)
  =
  \symnul{\grp}{+}(\weakbf)
  +
  \symnul{\grp}{-}(\weakbf).
\end{cases}
\end{equation}
\end{example}

\begin{example}[Actions of self-congruences of two-sided hypersurfaces] \label{ex:SelfCongr}
Suppose, momentarily, that $(M,g)$ is isometrically embedded (as a codimension-one submanifold) in a Riemannian manifold $(N,h)$, that the set $\Omega$ be connected
and assume further that the normal bundle of $M$ over $\lipdom$ is trivial.
Then we can pick a unit normal
$\nu$ on $\lipdom$
and thereby identify -- as usual -- sections of the normal bundle of $M_{|\lipdom}$
with functions on $\lipdom$.
With this interpretation of functions on $\lipdom$ in mind
and $\grp$ now a finite group of
diffeomorphisms of $N$
that map $\lipdom$ onto itself (as a set),
and everywhere on $\lipdom$
preserve the ambient metric $h$ meaning that $\phi^\ast h=h$ for any $\phi\in G$,
we have a natural action given by
\begin{align*}
(\phi,u)
&\mapsto
\nrmlsgn{\nu}(\phi)
  (u \circ \phi^{-1})
\quad
\text{ for all }
\phi \in \grp,~
u \colon \lipdom \to \R,
\end{align*}
where $\nrmlsgn{\nu}(\phi)\vcentcolon=h(\phi_*\nu,\nu)$ is a constant in $\Ogroup(1)=\{1,-1\}$.
We shall further assume
that the action of $\grp$ on $\lipdom$
is faithful, meaning that only the identity
element fixes $\lipdom$ pointly; this assumption is always satisfied in our applications.

In this context we continue to say that a function
$u \colon \lipdom \to \R$
is $\grp$-invariant
if $u=u \circ \phi$
for all $\phi \in \grp$,
and we say rather that
$u$ is $\grp$-equivariant
if $u=\nrmlsgn{\nu}(\phi) u \circ \phi$
for all $\phi \in \grp$
(that is, noting the identity
$
 \nrmlsgn{\nu}(\phi)
 =
 \nrmlsgn{\nu}(\phi^{-1})
$,
provided $u$ is invariant under
the $\nrmlsgn{\nu}$-twisted $\grp$ action).

Similarly, in this context, we set
\begin{equation}\label{eq:MeaningEquivariance}
\equivind{\grp}(\weakbf)
  \vcentcolon=
  \symind{\grp}{\nrmlsgn{\nu}}(\weakbf)
\quad \mbox{and} \quad
\equivnul{\grp}{\weakbf}
  \vcentcolon=
  \symnul{\grp}{\nrmlsgn{\nu}}(\weakbf),
\end{equation}
which we may refer to as simply
the $\grp$-equivariant index
and $\grp$-equivariant nullity
of $\weakbf$. 
We point out that we are abusing notation
in the above definitions
in that,
on the right-hand side of each,
in place of $\grp$
we mean really the group,
isomorphic to $\grp$
by virtue of the faithfulness assumption,
obtained by restricting
each element of $\grp$ to $\lipdom$,
and in place of $\nrmlsgn{\nu}$
we mean really the
corresponding homomorphism,
well-defined by the faithfulness assumption,
on this last group of isometries of $\lipdom$.
\end{example}

We now return to the more general assumptions on $G$ preceding this paragraph.

\subsection{Subdomains}\label{subsec:subdomains}

Suppose that $\lipdom_1 \subset \lipdom$ is another Lipschitz domain of $M$ (cf. Figure \ref{fig:subdomain}).
We shall define 
\begin{equation}
\label{subdomainbdydecomps}
\begin{aligned}
\intbdy\lipdom_1
  &\vcentcolon=
  \partial\lipdom_1 \cap \lipdom,
\qquad
&\extbdy\lipdom_1
  &\vcentcolon=
  \partial \lipdom_1 
    \setminus
    \closure{\intbdy\lipdom_1},
\\[1ex]
\dint{\dbdy}\lipdom_1
  &\vcentcolon=
  (\extbdy\lipdom_1 \cap \dbdy\lipdom)
    \cup \intbdy\lipdom_1,
\qquad
&\nint{\dbdy}\lipdom_1
  &\vcentcolon=
  \extbdy\lipdom_1 \cap \dbdy\lipdom,
\\[1ex]
\dint{\nbdy}\lipdom_1
  &\vcentcolon=
  \extbdy\lipdom_1 \cap \nbdy\lipdom,
\qquad
&\nint{\nbdy}\lipdom_1
  &\vcentcolon=
  (\extbdy\lipdom_1 \cap \nbdy\lipdom)
    \cup \intbdy\lipdom_1,
\\[1ex]
\dint{\rbdy}\lipdom_1
  &\vcentcolon=
  \extbdy\lipdom_1 \cap \rbdy\lipdom,
\qquad
&\nint{\rbdy}\lipdom_1
  &\vcentcolon=
  \extbdy\lipdom_1 \cap \rbdy\lipdom.
\end{aligned}
\end{equation}
In this way we prepare to pose two different
sets of boundary conditions on $\lipdom_1$,
whereby, roughly speaking,
in both cases
$\partial\lipdom_1$ inherits
whatever boundary condition is in effect
on $\partial\lipdom$ wherever the two meet
(corresponding to $\extbdy\lipdom_1$)
and the two sets of conditions are distinguished
by placing either
the Dirichlet or the Neumann condition
on the remainder
of the boundary
(corresponding to $\intbdy{\lipdom_1}$).
Naturally associated to these two sets of conditions
are the bilinear forms
\begin{equation}
\label{DirAndNeumInternalizationsOfBilinearForm}
\begin{aligned}
\dint{\weakbf}_{\lipdom_1}
  &\vcentcolon=
  \weakbf[
    \lipdom_1,g,\potential,\robinpotential,
    \dint{\dbdy}\lipdom_1,
    \dint{\nbdy}\lipdom_1,
    \dint{\rbdy}\lipdom_1
  ],
\\[1ex]
\nint{\weakbf}_{\lipdom_1}
  &\vcentcolon=
  \weakbf[
    \lipdom_1,g,\potential,\robinpotential,
    \nint{\dbdy}\lipdom_1,
    \nint{\nbdy}\lipdom_1,
    \nint{\rbdy}\lipdom_1
  ],
\end{aligned}
\end{equation}
defined, respectively, on the Sobolev spaces
$\sobd{\dint{\dbdy}\lipdom_1}(\lipdom_1,g)$
and $\sobd{\nint{\dbdy}\lipdom_1}(\lipdom_1,g)$.

\begin{figure}[b]
\centering
\begin{tikzpicture}[scale=3.1]
\coordinate(A)at(0,0.75);
\coordinate(B)at(-2,0);
\coordinate(C)at(-1.66,-1);
\coordinate(D)at(0,-0.66);
\coordinate(E)at(1.66,-1);
\coordinate(F)at(2,0);
\begin{scope}[line width=4pt]
\draw[Robin](C)sin(D)cos(E);
\draw[Robin](D)node[below]{$\rbdy\lipdom$};
\draw[Dirichlet]
(B)to node[midway,sloped,below]{$\dbdy\lipdom$}(C)
(E)to node[midway,sloped,below]{$\dbdy\lipdom$}(F);
\draw[Neumann](F)to node[midway,sloped,above]{$\nbdy\lipdom$}(A)to node[midway,sloped,above]{$\nbdy\lipdom$}(B);
\fill[black!10](A)--(B)--(C)sin(D)cos(E)--(F)--cycle;
\end{scope}
\begin{scope}
\clip(A)--(B)--(C)sin(D)--cycle;
\filldraw[black!25,draw=black!75,line width=4pt]
(A)--(B)--(C)sin(D)--cycle;
\end{scope}
\path(A)--(C)node[midway,circle,inner sep=0pt]{$\lipdom_1$};
\begin{scope}[black!75]
\path(A)to node[midway,below, sloped]{$\extbdy\lipdom_1\cap\nbdy\lipdom$}(B);
\path(B)to node[midway,above, sloped]{$\extbdy\lipdom_1\cap\dbdy\lipdom$}(C);
\path(C)to node[midway,above=1ex, sloped]{$\extbdy\lipdom_1\cap\rbdy\lipdom$}(D);
\path(A)to node[midway,below, sloped]{$\intbdy\lipdom_1$}(D);
\end{scope}
\end{tikzpicture}
\caption{Example of a Lipschitz domain $\Omega$ with subdomain $\Omega_1$. }%
\label{fig:subdomain}%
\end{figure}

Recalling $(\grp,\twsthom)$ from above, with the tacit understanding that 
$(\lipdom,g,\potential,\robinpotential,
 \dbdy\lipdom,\nbdy\lipdom,\rbdy\lipdom,
 \grp,\twsthom)$ 
is admissible, we further assume that each element of $\grp$ maps $\lipdom_1$ onto itself;
since $\grp$ preserves $\lipdom$ and respects the decomposition \eqref{bdydecomp},
it follows that it also respects the decompositions \eqref{subdomainbdydecomps}.
Somewhat abusively, we shall write $\twstact{\twsthom}$ and $\invproj{\grp}{\twsthom}$
not only for the maps
\eqref{twistedrep} and \eqref{invariantprojection}
but also for their counterparts
with $\lipdom$ replaced by $\lipdom_1$,
which are well-defined under our assumptions.
The spaces
$
 \eigenspsym
   {\dint{\weakbf}_{\lipdom_1}}
   {\sim t}{\grp}{\twsthom}
$
and
$
 \eigenspsym
   {\nint{\weakbf}_{\lipdom_1}}
   {\sim t}{\grp}{\twsthom}
$
as in \eqref{inveigensp}, are then also well-defined.

\section{Fundamental tools}\label{sec:theory}

\subsection{Index and nullity bounds in the style of Montiel and Ros}
Recalling the notation and assumptions of Section \ref{sec:nota},
suppose now that we have not only
$\lipdom_1 \subset \lipdom$ as above,
but also (open) Lipschitz subdomains
$\lipdom_1,\ldots,\lipdom_n \subset \lipdom$
which are pairwise disjoint,
each of which satisfies the same assumptions
as $\lipdom_1$ in Section \ref{subsec:subdomains},
and whose closures cover $\closure{\lipdom}$. 
In particular, we assume that each element of the group $G$ maps each subdomain $\lipdom_i$ onto itself. 
We assume further that $G$ acts transitively on the connected components of $\lipdom$ and note that this last condition is always satisfied in the important special case that $\lipdom$ is connected.

\begin{proposition}
[Montiel--Ros bounds on the number of eigenvalues below a threshold]
\label{mr}
With assumptions as in the preceding paragraph
and notation as in Section \ref{sec:nota}, the following inequalities hold
for any $t \in \R$
\begin{enumerate}[label={\normalfont(\roman*)}]
\item \label{mrLower}
$\displaystyle
 \dim \eigenspsym{\weakbf}{<t}{\grp}{\twsthom}
 \geq
 \dim \eigenspsym
       {\dint{\weakbf}_{\lipdom_1}}
       {<t}{\grp}{\twsthom}
   +\sum_{i=2}^n \dim
     \eigenspsym{\dint{\weakbf}_{\lipdom_i}}
                {\leq t}{\grp}{\twsthom}
$, 
\item \label{mrUpper}
$\displaystyle
 \dim \eigenspsym{\weakbf}{\leq t}{\grp}{\twsthom}
 \leq
 \dim \eigenspsym
      {\nint{\weakbf}_{\lipdom_1}}
      {\leq t}{\grp}{\twsthom}
   +\sum_{i=2}^n \dim
     \eigenspsym{\nint{\weakbf}_{\lipdom_i}}
                {<t}{\grp}{\twsthom}
$.
\end{enumerate}
\end{proposition}

The statement and proof
of Proposition \ref{mr}
are adapted from Lemma 12
and Lemma 13 of \cite{MontielRos},
which concern the spectrum
of the Laplacian on branched coverings
of the round sphere
and rely on standard,
fundamental facts about
eigenvalues and eigenfunctions
of Schr\"{o}dinger operators,
much as in the proof
of the classical
Courant nodal domain theorem.
These arguments are readily applied
to more general Schr\"{o}dinger operators
on more general domains,
as observed for instance in
\cites{KapouleasWiygulIndex},
where such bounds in the style
of Montiel and Ros
played a major role
in the computation
of the index and nullity
of the $\xi_{g,1}$ Lawson surfaces.
Here, instead, we present an extended version allowing
for the imposition
of mixed (Robin and Dirichlet) boundary conditions
and invariance under a group action; as mentioned in the introduction, this level of generality is motivated by the goal of bounding (from above and below) the $G$-equivariant Morse index of free boundary minimal surfaces. 
(Our treatment of course includes the fundamental case when $G$ is the trivial group.)

\begin{proof}
Throughout the proof
we will make free use of the consequences
\eqref{basic_spectral_facts}
of the spectral theorem
for the various bilinear forms
appearing in the statement.
Fix $t \in \R$.
For \ref{mrLower}
we will verify injectivity
of the map
\begin{align*}
\dint{\iota}
  \colon
  \eigenspsym{\dint{\weakbf}_{\lipdom_1}}
             {<t}{\grp}{\twsthom}
    \oplus
    \bigoplus_{i=2}^n
    \eigenspsym{\dint{\weakbf}_{\lipdom_i}}
                {\leq t}{\grp}{\twsthom}
  &\to
  \eigenspsym{\weakbf}{<t}{\grp}{\twsthom}
\\
(u_1,u_2,\ldots,u_n)
  &\mapsto
  \eigenproj{\weakbf}{<t}
  \biggl(\sum_{i=1}^n U_i\biggr)
\end{align*}
where each $U_i$
is the extension
to $\lipdom$
of $u_i$
such that $U_i$ vanishes
on $\lipdom \setminus \lipdom_i$.
Clearly, each such extension lies
in the image of $\invproj{\grp}{\twsthom}$,
which, as observed above, commutes with
$\eigenproj{\weakbf}{< t}$,
so that the map is indeed well-defined
with its asserted target.
Now suppose that
$(u_1,\ldots,u_n)$
belongs to the domain
of $\dint{\iota}$,
and set
$
 v
 \vcentcolon=
 \sum_{i=1}^n
 U_i
$.
Then
$v \in \sobd{\dbdy\lipdom}(\lipdom,g)$
and
\begin{equation*}
\weakbf(v,v)
=
\sum_{i=1}^n
\dint{\weakbf}_{\lipdom_i}(u_i,u_i)
\leq
t\nm{v}_{L^2(\lipdom,g)}^2,
\end{equation*}
with equality possible only when
$u_1=0$.
To check injectivity
suppose next that
$\dint{\iota}(u_1,\ldots,u_n)=0$.
By definition of $\dint{\iota}$
this assumption means that
$v$ is $L^2(\lipdom,g)$-orthogonal
to $\eigenspsym{\weakbf}{<t}{\grp}{\twsthom}$,
and so in view of the
preceding inequality and \eqref{basic_spectral_facts}
we have
$v \in \eigenspsym{\weakbf}{=t}{\grp}{\twsthom}$.
Thus,
$v$ satisfies the elliptic equation
$(\Delta_g + \potential + t)u=0$;
moreover,
we must also have
$v|_{\lipdom_1}=u_1=0$,
but now the unique continuation principle
\cite{AronszajnUC}
implies that $v=0$,
whence
$(u_1,\ldots,u_n)=0$,
completing the proof of
\ref{mrLower}.

For \ref{mrUpper}
we verify injectivity of
\begin{align*}
\nint{\iota}
  \colon
  \eigenspsym{\weakbf}{\leq t}{\grp}{\twsthom}
  &\to
  \eigenspsym{\nint{\weakbf}_{\lipdom_1}}
             {\leq t}{\grp}{\twsthom}
    \oplus
    \bigoplus_{i=2}^n
    \eigenspsym{\nint{\weakbf}_{\lipdom_i}}
               {< t}{\grp}{\twsthom}
\\
u
  &\mapsto
  \biggl(
    \eigenproj{\nint{\weakbf}_{\lipdom_1}}{\leq t}
      u|_{\lipdom_1}, ~
    \eigenproj{\nint{\weakbf}_{\lipdom_2}}{< t}
      u|_{\lipdom_2}, ~
    \ldots, ~
    \eigenproj{\nint{\weakbf}_{\lipdom_n}}{< t} 
    u|_{\lipdom_n}
  \biggr)
\end{align*}
instead.
Note that
\begin{equation}
\label{restriction_obeys_syms_and_bcs}
u|_{\lipdom_i}
\in
\invproj{\grp}{\twsthom}\bigl(L^2(\lipdom_i,g)\bigr)
\cap
\sobd{\nint{\dbdy}\lipdom_i}(\lipdom_i,g)
\end{equation}
for each $i$;
in particular,
the left inclusion
and the commutativity
of $\invproj{\grp}{\twsthom}$
with each of the spectral projections
appearing in the definition
of $\nint{\iota}$ ensure
that the latter really is well-defined.
Suppose then that
$u$ belongs to the domain of
$\nint{\iota}$
and $\nint{\iota}u=(0,\ldots,0)$.
The second assumption
(making use of the right inclusion
in \eqref{restriction_obeys_syms_and_bcs}
in addition to \eqref{basic_spectral_facts})
implies
\begin{equation*}
\weakbf(u,u)
=
\sum_{i=1}^n
  \nint{\weakbf}_{\lipdom_i}
    (u|_{\lipdom_i}, u|_{\lipdom_i})
\geq
t\nm{u}_{L^2(\lipdom,g)}^2,
\end{equation*}
with equality possible
only when $u|_{\lipdom_1}=0$.
Recalling that, by assumption, $u\in \eigenspsym{\weakbf}{\leq t}{\grp}{\twsthom}$
we therefore conclude, appealing to \eqref{basic_spectral_facts}, that
$u \in \eigenspsym{\weakbf}{=t}{\grp}{\twsthom}$
and indeed this equality case holds.
In particular, $u$ satisfies the elliptic equation
$(\Delta_g + \potential + t)u=0$,
but then the condition $u|_{\lipdom_1}=0$
and the unique continuation principle
imply $u=0$, ending the proof.
\end{proof}

In particular, in our applications we will repeatedly (yet not always) appeal to the special case when $t=0$ and $\Omega$ (most often equal to the whole ambient manifold itself $M$) is partitioned in a finite collection of pairwise isometric domains:

\begin{corollary}
[Montiel--Ros index and nullity bounds from isometric pieces]
\label{cor:MontielRosHuman}
In the setting of the previous proposition
let us suppose the domains
$\lipdom_1, \ldots, \lipdom_n$
to be pairwise isometric
via isometries
of $\lipdom$.
Then 
\begin{enumerate}[label={\normalfont(\roman*)}]
\item\label{mrIsoPieces:below} 
$
 \symind{\grp}{\twsthom}(\weakbf)
 \geq
 n \symind{\grp}{\twsthom}
       (\dint{\weakbf}_{\lipdom_1})
   + (n-1)
   \symnul{\grp}{\twsthom}
       (\dint{\weakbf}_{\lipdom_1})
$, 
\item\label{mrIsoPieces:above} 
$
 \symind{\grp}{\twsthom}(\weakbf)
   +\symnul{\grp}{\twsthom}(\weakbf)
 \leq
  n \symind{\grp}{\twsthom}
      (\nint{\weakbf}_{\lipdom_1})
  + \symnul{\grp}{\twsthom}
        (\nint{\weakbf}_{\lipdom_1})
$.
\end{enumerate}
\end{corollary}

\begin{remark}
We further, explicitly note how the two inequalities given in the previous corollary jointly imply the ``compatibility condition'' that
\begin{equation}\label{eq:compatibility}
(n-1)
   \symnul{\grp}{\twsthom}
       (\dint{\weakbf}_{\lipdom_1})-\symnul{\grp}{\twsthom}
        (\nint{\weakbf}_{\lipdom_1})\leq n\Bigl(\symind{\grp}{\twsthom}
      (\nint{\weakbf}_{\lipdom_1})-\symind{\grp}{\twsthom}
       (\dint{\weakbf}_{\lipdom_1})\Bigr)
    \end{equation}
which in general has non-trivial content.
\end{remark}

\begin{remark}\label{rem:Connected}
The requirement that the domains in question be $G$-invariant implies, in certain examples, that some of them may in fact have to be taken disconnected. We will however discuss, in the next subsection, how this nuisance may actually be avoided in the totality of our later applications.
\end{remark}

\subsection{Reduction and extension of domain under symmetries}

With our standing assumptions on
$(\lipdom,g)$, $\weakbf$ and $(\grp,\twsthom)$
in place, encoded in the requirement that they determine admissible data,
we again assume that
$\lipdom_1, \ldots, \lipdom_n \subset \lipdom$
are pairwise disjoint
Lipschitz domains 
whose closures cover $\lipdom$.
However, for the specific purposes of this section,
we assume $\lipdom$ connected and,
rather than assuming
$\grp$-invariance of each $\lipdom_i$,
we instead suppose that
$\grp$ preserves the collection $\{\lipdom_i\}_{i=1}^n$ (while -- as per our general postulate -- also respecting the decomposition \eqref{bdydecomp}, which dictates the boundary conditions \eqref{eq:BoundaryConditions}), and acts transitively
on its elements
(so in particular the $\lipdom_i$
are pairwise isometric).
The (possibly trivial) subgroup of $\grp$ which preserves
$\lipdom_1$ we call $\stabgrp$.
Note that $\stabgrp$ preserves
$\intbdy\lipdom_1$ in particular.

For each
$p \in \intbdy\lipdom_1$
we define
\begin{equation*}
\grp_p
\vcentcolon=
\{
  \phi \in \grp
  \st
  \exists  U
    \underset{\clap{\scriptsize open}}{\subseteq}
    \intbdy\lipdom_1
  \quad
  p \in U
  \mbox{ and }
  \phi_{|U}=\operatorname{id}
\}.
\end{equation*}
Then $\grp_p$ is a subgroup of $\grp$ 
having order at most $2$,
as we now explain.
Let $\phi_1,\phi_2 \in \grp_p$.
Then we have open neighborhoods
$U_1,U_2$ of $p$ in
$\intbdy\lipdom_1$
with $U_i$ fixed pointwise by $\phi_i$.
By the Lipschitz assumption
there exists $q \in U_1 \cap U_2$
at which $\intbdy\lipdom_1$
has a well-defined outward unit conormal
$\eta_q$.
Then
for each $i$ we have
$(d_q\phi_i)(\eta_q)=\epsilon_i \eta_q$
for some $\epsilon_i=\pm 1$.
If an $\epsilon_i=+1$,
then,
since $\phi_i$ fixes $U_i$ pointwise
and $\lipdom$ is connected,
$\phi_i$ must be the identity on $\lipdom$
(which comes essentially by arguing e.\,g. as in Lemma 4.5 of \cite{CarLi19}).
If $\epsilon_1=\epsilon_2=-1$,
then
similarly $\phi_1 \circ \phi_2^{-1}$
is the identity on $\lipdom$,
establishing our claim.
Note also that the set $\{p \st \abs{\grp_p}=2\}$
is open in $\intbdy\lipdom_1$
and that for each $\chi \in \stabgrp$
the map
$
 \phi
 \mapsto
 \chi \circ \phi \circ \chi^{-1}
$
defines an isomorphism
from $\grp_p$ to $\grp_{\chi(p)}$
which commutes with $\twsthom$.

For each $p$ we next set
\begin{equation*}
\twsthom_p
\vcentcolon=
\begin{cases}
\hphantom{-}0 &
  \text{ if $\abs{\grp_p}=1$,}
\\
\hphantom{-}1 &
  \text{ if $\abs{\grp_p}=2$
    but $\twsthom(\grp_p)=\{+1\}$,}
\\
-1
  &
  \text{ if $\abs{\grp_p}=2$
      and $\twsthom(\grp_p)=\{+1,-1\}$,}
\end{cases}
\end{equation*}
and we in turn define
the subsets
$
 \partial_+\lipdom_1,
 \partial_-\lipdom_1
 \subseteq
 \intbdy\lipdom_1
$
by letting (respectively)
\begin{equation*}
\partial_\pm\lipdom_1
\vcentcolon=
\twsthom_p^{-1}(\pm 1).
\end{equation*}
With the aid of the foregoing observations we see that
$\partial_+\lipdom_1$
and $\partial_-\lipdom_1$
are open and disjoint,
and each is preserved by $\stabgrp$.
We now impose the additional assumption
that their closures cover
$\intbdy\lipdom_1$,
and finally we set $\weakbf_{\lipdom_1}
  \vcentcolon=
  \weakbf
    [
      \lipdom_1,g,q,r,
      \dbdy\lipdom_1,\nbdy\lipdom_1,\rbdy\lipdom_1
    ]$, where
\begin{align*}
\dbdy\lipdom_1
 &\vcentcolon=
  \partial_-\lipdom_1
    \cup
    (\extbdy\lipdom_1 \cap \dbdy\lipdom),
\\
\nbdy\lipdom_1
&\vcentcolon=
  \partial_+\lipdom_1
    \cup
    (\extbdy\lipdom_1 \cap \nbdy\lipdom),
\\
\rbdy\lipdom_1
&\vcentcolon=
  \extbdy\lipdom_1 \cap \rbdy\lipdom.
\end{align*}

\begin{lemma}
[Reduction and extension of domain under symmetries]
\label{dom_red_ext}
Under the above assumptions,
for every integer $i \geq 1$
\begin{equation*}
\eigenvalsym{\weakbf}{i}{\grp}{\twsthom}
=
\eigenvalsym{\weakbf_{\lipdom_1}}{i}{\stabgrp}{\twsthom},
\end{equation*}
and 
the $(\stabgrp,\twsthom)$-invariant eigenfunctions
of $\weakbf_{\lipdom_1}$
are the restrictions to $\lipdom_1$
of the $(\grp,\twsthom)$-invariant eigenfunctions
of $\weakbf$.
\end{lemma}

\begin{proof}
First observe that
\begin{equation*}
v
  \in
  \invproj{\grp}{\twsthom}
    \sobd{\dbdy\lipdom}(\lipdom,g)
~\Rightarrow~
v|_{\lipdom_1}
  \in 
  \invproj{\stabgrp}{\twsthom}
    \sobd{\dbdy\lipdom_1}(\lipdom_1,g),
\end{equation*}
using in particular the fact that
any $(\grp,\twsthom)$-invariant
function in $\sob(\lipdom,g)$
must have vanishing trace along
$\partial_-\lipdom_1$.
Next observe that our assumptions guarantee
that each $(\stabgrp,\twsthom)$-invariant
function $u$
on $\lipdom_1$
has a unique $(\grp,\twsthom)$-invariant
extension $\overline{u}$
to $\lipdom$.
This is also true of vector fields,
the action being
$\phi.X \vcentcolon= \twsthom(\phi)\phi_*X$.
Now suppose
$
 u
 \in
 \invproj{\stabgrp}{\twsthom}
  \sobd{\dbdy\lipdom_1}(\lipdom_1,g)
$.
Obviously
$
 \overline{u}
 \in  
  L^2(\lipdom,g)
$,
and we next check that in fact
$\overline{u} \in \sob(\lipdom,g)$
with $\nabla_g \overline{u} = \overline{\nabla_g u}$.

For this let $X$ be a smooth vector field
with support contained in $\lipdom$.
Let $Y \vcentcolon= \invproj{\grp}{\twsthom}X$
(meaning we average as in \eqref{invariantprojection}
but with the appropriate action for vector fields,
as above). Then
(writing, with slight abuse of notation, $L^2(\lipdom,g)$
and $L^2(\lipdom_1,g)$ also
for the Hilbert spaces of $L^2$ vector fields
on $\lipdom$ and $\lipdom_1$ respectively,
in metric $g$)
\begin{equation*}
\begin{aligned}
\sk{X,\overline{\nabla_g u}}_{L^2(\lipdom,g)}
&=
\sk{Y,\overline{\nabla_g u}}_{L^2(\lipdom,g)}
=
n
  \sk{Y|_{\lipdom_1},\nabla_g u}_{L^2(\lipdom_1,g)}
\\
&=
n
  \sk{1,\dive (uY|_{\lipdom_1})}_{L^2(\lipdom_1,g)}
  -n\sk{u,\dive Y|_{\lipdom_1}}_{L^2(\lipdom_1,g)}
\\
&=
n
  \sk{
    u|_{\partial\lipdom_1},
    g(\conormal{\lipdom_1}{g},Y|_{\partial\lipdom_1})
  }_{L^2(\partial\lipdom_1,g)}
  -\sk{\overline{u},\dive Y}_{L^2(\lipdom,g)}
\\
&=
0-\sk{\dive X,\overline{u}}_{L^2(\lipdom,g)};
\end{aligned}
\end{equation*}
in the third line
we have used the divergence theorem
(see for example Theorem 4.6
of \cite{EvansGariepy}
for a statement serving our assumptions)
with
$u|_{\partial\lipdom_1}$ of course
the trace of $u$
and $\conormal{\lipdom_1}{g}$
the almost everywhere defined outward unit conormal,
and in the fourth line
we have used the fact that
the $(\grp,\twsthom)$-invariance of $Y$
forces it to be (almost everywhere) orthogonal
to this last conormal
on $\partial_+\lipdom_1$,
while on the other hand,
as already noted above,
$u|_{\partial\lipdom_1}$
vanishes on $\partial_-\lipdom_1$.
Thus every element of
$
 \invproj{\stabgrp}{\twsthom}
 \sobd{\dbdy\lipdom_1}(\lipdom_1,g)
$
extends uniquely to an element of
$
 \invproj{\grp}{\twsthom}
 \sobd{\dbdy\lipdom}(\lipdom,g)
$.
It is now straightforward to verify that
for all $t \in \R$
restriction to $\lipdom_1$
furnishes a bijection
$
 \eigenspsym{\weakbf}{=t}{\grp}{\twsthom}
 \to
 \eigenspsym{\weakbf_{\lipdom_1}}{=t}{\stabgrp}{\twsthom}
$,
which implies the claims.
\end{proof}

For the purposes of our later geometric applications, it is convenient to focus on two special cases, which correspond to the examples we presented in Section \ref{subs:GroupAction}.

\begin{example}[Actions of order-2 groups] \label{ex:Group2El_Meta}
With respect to our general setup, let $\Omega=M$ and consider $G=\sk{\phi}$ where $\phi$ is a (non-trivial) isometric involution of $M$. Suppose further (which is not true in general) that the set of fixed points of the action divides $M$ into two open regions, which we shall label $\Omega_1, \Omega_2$. Then note that, arguing as above, one must have $\phi(\Omega_1)=\Omega_2$ (as well as $\phi(\Omega_2)=\Omega_1$). In particular $H$ is the trivial subgroup, just consisting of the identity element.
That said, there are two cases depending on the choice of twisting homomorphism $\sigma\colon G\to\left\{-1,1\right\}$ we consider:
\begin{enumerate}[label={(\arabic*)}]
    \item if we let $\sigma(\phi)=+1$ then $\partial_+\lipdom_1=\intbdy\lipdom_1, \partial_-\lipdom_1=\emptyset$ so we are considering the (non-equivariant) spectrum of $T_{\lipdom_1}$ adding a Neumann boundary condition along $\intbdy\lipdom_1$;
    \item if we let $\sigma(\phi)=-1$ then $\partial_+\lipdom_1=\emptyset, \partial_-\lipdom_1=\intbdy\lipdom_1$ so we are considering the (non-equivariant) spectrum of $T_{\lipdom_1}$ adding a Dirichlet boundary condition along $\intbdy\lipdom_1$.
\end{enumerate}
\end{example}

\begin{example}[Actions of self-congruences of two-sided hypersurfaces] \label{ex:SelfCongr_Meta}
Here we follow-up on the discussion of Example \ref{ex:SelfCongr}, but specified to $\Omega=M$ for $N=\B^3$ and  $G=\pri_n$ (i.\,e.  we postulate the ambient manifold to be the Euclidean ball, and the surface $M$ to have prismatic symmetry). 
We refer the reader to the first part of Section \ref{sec:FBMSgen} for basic recollections about this group action, and related ones.
We let $\Omega_1$ to be an open fundamental domain for this action (so that $M$ is covered by the closures of exactly $4n$ pairwise isometric domains); it follows that again $H$ is the trivial subgroup.
Considering the sign homomorphism $\sigma\colon\pri_n\to\left\{-1,+1\right\}$ defined in Example \ref{ex:SelfCongr}, then it is readily checked that $\partial_+\lipdom_1=\intbdy\lipdom_1$, $\partial_-\lipdom_1=\emptyset$ and so -- when applied to this case -- Lemma~\ref{dom_red_ext} compares (and proves equality of) the (fully-)equivariant spectrum of the problem, with the spectrum of a fundamental domain, with Neumann boundary conditions added on each interior side.
\end{example}

\subsection{Spectral stability}
\label{subsec:continuity}

As it has been anticipated in the introduction, in our applications we will analyze the spectrum of free boundary minimal surfaces obtained by gluing certain constituting blocks. 
In that respect, we will need to derive from ``geometric convergence'' results some corresponding ``spectral convergence'' results. Suppose we have a sequence
$
 \{(
  \lipdom_n, g_n, \potential_n,
  \robinpotential_n, \dbdy\lipdom_n,
  \nbdy\lipdom_n, \rbdy\lipdom_n,
  \grp_n,\twsthom_n
  )\}
$
of admissible data, as well as ``limit data''
$
 (
  \lipdom, g, \potential, \robinpotential,
  \dbdy\lipdom, \nbdy\lipdom, \rbdy\lipdom,
  \grp_\infty,\twsthom_\infty
 ),
$
satisfying all our assumptions on admissible data except that $\grp_\infty$ is possibly allowed to have infinite order. 
For instance, in our later applications $\grp_\infty$ is the compact Lie group $\Ogroup(2)$.
Although we originally introduced the notation
$\eigenvalsym{\weakbf}{i}{\grp_\infty}{\twsthom_\infty}$,
with $\weakbf$ the bilinear form associated to the foregoing data,
for $\grp_\infty$ finite, the notion remains well-defined for infinite $\grp_\infty$.
The quantities
$\symind{\grp_\infty}{\twsthom_\infty}(\weakbf)$
and $\symnul{\grp_\infty}{\twsthom_\infty}(\weakbf)$
are likewise defined in this setting; as a special case,
we can in turn define
$\equivind{\grp_\infty}(\weakbf)$
and $\equivnul{\grp_\infty}(\weakbf)$
for $\grp_\infty$ a suitable
infinite-order symmetry group of a hypersurface (as per Example \ref{ex:SelfCongr}). That being said, alongside $\weakbf[
  \lipdom, g, \potential, \robinpotential,
  \dbdy\lipdom, \nbdy\lipdom, \rbdy\lipdom]$,
we then have the corresponding sequence
$\{\weakbf_n\}$ with
$
 \weakbf_n
 \vcentcolon=
 \weakbf[\lipdom_n,g_n,\potential_n, \robinpotential_n,
         \dbdy\lipdom_n,\nbdy\lipdom_n,\rbdy\lipdom_n]
$. We will present some conditions on the data
that ensure
\begin{equation}
\label{spectralCtyMeaning}
\lim_{n \to \infty}
\eigenvalsym{\weakbf_n}{i}{\grp_n}{\twsthom_n}
=
\eigenvalsym{\weakbf}{i}{\grp_\infty}{\twsthom_\infty}
\mbox{ for all } i.
\end{equation}
As we are especially interested
in index and nullity,
we immediately point out that
\eqref{spectralCtyMeaning} implies
\begin{equation}
\label{implicationsOfSpecConvForIndAndNul}
\begin{gathered}
\symind{\grp_\infty}{\twsthom_\infty}(\weakbf)
  \leq
  \liminf_{n \to \infty}
    \symind{\grp_n}{\twsthom_n}(\weakbf_n),
\qquad
\limsup_{n \to \infty}
    \symnul{\grp_n}{\twsthom_n}(\weakbf_n)
  \leq
  \symnul{\grp_\infty}{\twsthom_\infty}(\weakbf),
\\
\limsup_{n \to \infty}
    \Bigl(
      \symind{\grp_n}{\twsthom_n}(\weakbf_n)
      +\symnul{\grp_n}{\twsthom_n}(\weakbf_n)
    \Bigr)
 \leq
 \symind{\grp_\infty}{\twsthom_\infty}(\weakbf)
   +\symnul{\grp_\infty}{\twsthom_\infty}(\weakbf).
\end{gathered}
\end{equation}

\begin{proposition}
\label{spectralCtyWrtCoeffsAndSyms}
Let
$
 (
  \lipdom, g, \potential, \robinpotential,
  \dbdy\lipdom, \nbdy\lipdom, \rbdy\lipdom,
  \grp_\infty,\twsthom_\infty
 )
$
satisfy all our assumptions on admissible data
except that we allow $\grp_\infty$ to have infinite order;
let $\weakbf$ be the bilinear form determined by the data.
Let
$
\{
 (
  \lipdom, g_n, \potential_n, \robinpotential_n,
  \dbdy\lipdom, \nbdy\lipdom, \rbdy\lipdom,
  \grp_n,\twsthom_n
 )
\}
$
be a sequence of admissible data,
with corresponding sequence
$\{\weakbf_n\}$ of bilinear forms.
Assume
\begin{equation*}
\sup_n \sup_\lipdom
\left(
\abs{g_n}_g + \abs{g_n^{-1}}_g
  +\abs{\potential_n} + \abs{\robinpotential_n}
\right)
<
\infty
\quad \mbox{and} \quad
(g_n,\potential_n,\robinpotential_n)
  \xrightarrow[n \to \infty]{\text{a.e. on } \lipdom}
  (g,\potential,\robinpotential).
\end{equation*}
Assume further that
\begin{enumerate}[label={\normalfont(\arabic*)}]
    \item $\grp_n \leq \grp_\infty$ for all $n$, and $\twsthom_n(\phi_n)=\twsthom(\phi_n)$ for all $n$ and all $\phi_n \in \grp_n$;
    \item for each $\phi \in \grp_\infty$ there exists a sequence $\{\phi_n\}$ such that:
    \begin{enumerate}[label={\normalfont(\alph*)}]
 \item $\phi_n \in \grp_n$ for all $n$,
\item $\phi_n^*
  \xrightarrow[n \to \infty]{}
  \phi^*$ strongly as linear endomorphisms of $L^2(\lipdom, g)$,
  \item 
$\twsthom_n(\phi_n)=\twsthom(\phi)$ for all $n$.
    \end{enumerate}
\end{enumerate}
Then
\begin{equation*}
\lim_{n \to \infty}
\eigenvalsym{\weakbf_n}{i}{\grp_n}{\twsthom_n}
=
\eigenvalsym{\weakbf}{i}{\grp_\infty}{\twsthom_\infty}
\text{ for all } i.
\end{equation*}
\end{proposition}

\begin{proof}
For expository convenience, we will first focus on the case when $\grp_n=\grp_\infty$ and $\sigma_n=\sigma_\infty$ for all $n$, thereby implicitly assuming  $(
  \lipdom, g, \potential, \robinpotential,
  \dbdy\lipdom, \nbdy\lipdom, \rbdy\lipdom,
  \grp_\infty,\twsthom_\infty
 )$ to be admissible data (in our standard sense); we shall simply denote by $G$ the group in question, and by $\sigma$ the associated homomorphism.

Fix the index $i\geq 1$.
We will start by showing that
\begin{equation}
\label{spectralUpperSemicty}
\limsup_{n \to \infty}
  \eigenvalsym{\weakbf_n}{i}{\grp}{\twsthom}
\leq
\eigenvalsym{\weakbf}{i}{\grp}{\twsthom}.
\end{equation}
For this we start with an
$L^2(\lipdom,g)$-orthonormal set
$\{u_j\}_{j=1}^i$
such that $u_j$ is
a $(\grp,\twsthom)$-invariant
eigenfunction of $\weakbf$
with eigenvalue
$\eigenvalsym{\weakbf}{j}{\grp}{\twsthom}$.
Then our assumptions on the coefficients
together with the dominated convergence theorem
imply that for all $1 \leq j,k \leq i$
\begin{equation*}
\begin{gathered}
\lim_{n \to \infty}
  \sk{u_j,u_k}_{L^2(\lipdom,g_n)}
  =
  \sk{u_j,u_k}_{L^2(\lipdom,g)},
\quad
\lim_{n \to \infty}
  \sk{u_j, \potential_n u_k}_{L^2(\lipdom,g_n)}
  =
  \sk{u_j, \potential u_k}_{L^2(\lipdom,g)},
\\
\lim_{n \to \infty}
  \int_\lipdom
  g_n(\nabla_{g_n}u_j,\nabla_{g_n}u_k)
  \,
  \hausint{d}{g_n}
  =
  \int_\lipdom
  g(\nabla_g u_j, \nabla_g u_k)
  \,
  \hausint{d}{g},
\\
\lim_{n \to \infty}
  \int_{\rbdy\lipdom}
  \robinpotential_n u_ju_k
  \,
  \hausint{d-1}{g_n}
  =
  \int_{\rbdy\lipdom}
  \robinpotential u_ju_k
  \,
  \hausint{d-1}{g}.
\end{gathered}
\end{equation*}

In conjunction with the min-max characterization
\eqref{symminmax}
this proves \eqref{spectralUpperSemicty}. 
To conclude it thus suffices to prove the complementary inequality
\begin{equation}
\label{spectralLowerSemicty}
\liminf_{n \to \infty}
  \eigenvalsym{\weakbf_n}{i}{\grp}{\twsthom}
\geq
\eigenvalsym{\weakbf}{i}{\grp}{\twsthom}.
\end{equation}
By \eqref{spectralUpperSemicty} the sequence $\eigenvalsym{\weakbf_n}{i}{\grp}{\twsthom}$ is bounded from above uniformly in $n$,
and by the min-max characterization
\eqref{symminmax}
of eigenvalues along with the assumed uniform bounds on
$\potential_n$ and $\robinpotential_n$ and the trace inequality \eqref{interpolatingTraceInequality}
it is also bounded from below. 
Therefore the left-hand side
of \eqref{spectralLowerSemicty} is a real number,
and, by passing to a subsequence of the data
if necessary (without renaming),
we in fact assume without loss of generality that
\begin{equation}
\label{wlogConvergence}
\{\eigenvalsym{\weakbf_n}{j}{\grp}{\twsthom}\}
\text{ converges to $\lambda_j^\infty \in \R$
      for each } j \leq i,
\end{equation}
with $\lambda_j^\infty$
the $\liminf$ of the $j$\textsuperscript{th}
$(\grp,\twsthom)$-eigenvalue of the original sequence.

For each $j \leq i$ and each $n$ let $v_j^{(n)}$ be a $(\grp,\twsthom)$-invariant
eigenfunction of $\weakbf_n$ with eigenvalue
$\eigenvalsym{\weakbf_n}{j}{\grp}{\twsthom}$
such that for each $n$ the set
$\{v_j^{(n)}\}_{j=1}^i$
is $L^2(\lipdom,g_n)$-orthonormal.
It follows from the assumed unit
$L^2(\lipdom,g_n)$
bounds on the $v_j^{(n)}$,
the definitions of eigenvalues and eigenfunctions,
the eigenvalue bound following from \eqref{wlogConvergence},
and the assumed bounds
on $\potential_n$ and $\robinpotential_n$
as well as $g_n$ and $g_n^{-1}$
that the sequence
$\nm{v_j^{(n)}}_{\sob(\lipdom,g_n)}$
is bounded uniformly in $n$.
(The assumptions on the metrics is needed there
to ensure that the constants
in the trace inequality
\eqref{interpolatingTraceInequality},
as applied here,
can be chosen independently of $n$.)
It then follows, in turn,
using again the assumed bounds on
$g_n$ and $g_n^{-1}$
that $\nm{v_j^{(n)}}_{\sob(\lipdom,g)}$
is likewise bounded.
Consequently,
passing to a further subsequence
if needed,
for each $j \leq i$
there exists
$v_j \in \sob(\lipdom,g)$
which is simultaneously
a limit in $L^2(\lipdom,g)$
and a weak limit in $\sob(\lipdom,g)$
of $v_j^{(n)}$ as $n\to\infty$.
Note in particular that each $v_j$
is $(\grp,\twsthom)$-invariant.

The dominated convergence theorem,
our assumptions on the metrics,
and the $L^2(\lipdom,g)$-convergence
for each $j$
of $\{v_j^{(n)}\}$ to $v_j$
imply that
$\{v_j\}_{j=1}^i$
is $L^2(\lipdom,g)$-orthonormal,
so in particular this finite family is linearly independent.
In the same fashion,
but also appealing to the assumptions
on the $\potential_n$,
we get for all $1 \leq j \leq i$
and all $w \in L^2(\lipdom,g)$
\begin{align*}
\lim_{n \to \infty}
  \sk[\Big]{
    v_j^{(n)}, w
  }_{L^2(\lipdom,g_n)}
&=\sk[\Big]{v_j, w}_{L^2(\lipdom,g)},
&
\lim_{n \to \infty}
  \sk[\Big]{
    \potential_n v_j^{(n)}, w
  }_{L^2(\lipdom,g_n)}
&=\sk[\Big]{\potential v_j, w}_{L^2(\lipdom,g)}.
\end{align*}
Thanks to the weak convergence
in $\sob(\lipdom,g)$
of $\{v_j^{(n)}\}$ to $v_j$
for each $j$
(and again using the dominated convergence theorem,
the assumptions on the metrics,
and the $L^2$ convergence of each $\{v_j^{(n)}\}$),
we further conclude
that for all $1 \leq j \leq i$
and all $w \in \sob(\lipdom,g)$
\begin{equation*}
\lim_{n \to \infty}
  \int_\lipdom
  g_n(\nabla_{g_n} v_j^{(n)}, \nabla_{g_n}w)
  \,
  \hausint{d}{g_n}
=
\int_\lipdom
  g(\nabla_g v_j, \nabla_g w)
  \,
  \hausint{d}{g}.
\end{equation*}
We use the trace inequality
\eqref{interpolatingTraceInequality}
in conjunction with boundedness in
$\sob(\lipdom,g)$
of $\{v_j^{(n)}\} \cup \{v_j\}$
and the convergence in
$L^2(\lipdom,g)$
for each $j$ of $v_j^{(n)}$
to $v_j$ to deduce that
we also have $L^2(\partial\lipdom,g)$-convergence of the traces.
As one consequence
we see that each $v_j$
in fact belongs to
$\sobd{\dbdy\lipdom}(\lipdom,g)$.
As another,
by virtue of the assumptions on the
$\robinpotential_n$
and once again
the dominated convergence theorem,
we obtain
for all $1 \leq j \leq i$
and $w \in \sob(\lipdom,g)$
\begin{equation*}
\lim_{n \to \infty}
  \int_{\rbdy\lipdom} \robinpotential_n v_j^{(n)}w
  \,
  \hausint{d-1}{g_n}
=
\int_{\rbdy\lipdom}
  \robinpotential v_j w
  \,
  \hausint{d-1}{g_n}.
\end{equation*}
From the definition
of the $v_j^{(n)}$,
the assumption \eqref{wlogConvergence},
and the above three displayed equations
we conclude that for all
$1 \leq j \leq i$
and $w \in \sobd{\dbdy\lipdom}(\lipdom,g)$
we eventually have
\begin{equation*}
\weakbf(v_j,w)
=
\lim_{n \to \infty} \weakbf_n(v_j^{(n)},w)
=
\lim_{n \to \infty}
  \eigenvalsym{\weakbf_n}{j}{\grp}{\twsthom}
  \sk[\big]{v_j^{(n)},w}_{L^2(\lipdom,g_n)}
=
\lambda_j^\infty \sk[\big]{v_j,w}_{L^2(\lipdom,g)}.
\end{equation*}
Specifically, for the second equality above we have used the fact that $v_j^{(n)}$ is an eigenfunction of $T_n$; together, the inequalities then show that $v_j$ is an eigenfunction of $T$. Since $\{v_j\}_{j=1}^i$ is a linearly independent
subset of
$\invproj{\grp}{\twsthom}\sobd{\dbdy\lipdom}(\lipdom,g)$,
it follows that
$\lambda_i^\infty \geq \eigenvalsym{\weakbf}{i}{\grp}{\twsthom}$,
completing the proof in the case of ``fixed symmetry group''.  

However, it is actually straightforward to generalize
the above argument to capture also continuity
in the symmetries. The proof above
goes through with mostly superficial modification, and
we address the only two salient points.
First, in proving \eqref{spectralUpperSemicty},
but with $(\grp,\twsthom)$
replaced on the left by
$(\grp_n,\twsthom_n)$
and on the right by
$(\grp_\infty,\twsthom_\infty)$,
note that each $u_j$,
now assumed $(\grp_\infty,\twsthom_\infty)$-invariant,
is by our hypotheses
also $(\grp_n,\twsthom_n)$-invariant for each $n$.
Second,
in proving the corresponding analogue
of \eqref{spectralLowerSemicty}
note that each $v_j$
is, as the $L^2(\lipdom,g)$ limit
of a sequence whose $n$\textsuperscript{th}
term is $(\grp_n,\twsthom_n)$-invariant,
by our hypotheses, itself
$(\grp_\infty,\twsthom_\infty)$-invariant.
\end{proof}

We now turn our attention to the related, yet different problem of handling controlled changes in the domain. We switch to slightly different notation, that is again tailor-made to best fit our later applications.

\begin{proposition}
\label{spectralContinuityWrtExcision}
Let
$
 (
  \lipdom, g, \potential, \robinpotential,
  \dbdy\lipdom, \nbdy\lipdom, \rbdy\lipdom,
  \grp,\twsthom
 )
$
be admissible data,
with corresponding bilinear form $\weakbf$. Suppose that
for any $\delta>0$ less than the injectivity radius of $(M,g)$, say $\delta_0$, we are given a Lipschitz domain $\Omega_{\delta}\subset \Omega$ such that $
 (
  \lipdom_\delta, g, \potential, \robinpotential,
  \dbdy\lipdom_\delta, \nbdy\lipdom_\delta, \rbdy\lipdom_\delta,
  \grp,\twsthom
 )
$
are also admissible data (with suitable restrictions of tensors and functions tacitly understood), and 
whose complement $K_{\delta}\vcentcolon=\Omega\setminus \Omega_{\delta}$ satisfies 
\begin{equation}\label{eq:containment}
\bigcup_{p\in S}\overline{B_{f_1(\delta)}(p)}\subset K_{\delta}\subset \bigcup_{p\in S}\overline{B_{f_2(\delta)}(p)}
\end{equation}
for some finite set of points $S\subset\overline{\Omega}$ and monotone functions $f_1, f_2\colon\Interval{0,\delta_0}\to\R_{\geq 0}$ such that $\lim_{\delta \to 0}f_2(\delta)=0$.
Consider the sets as in \eqref{subdomainbdydecomps} with $\Omega_{\delta}$ in lieu of $\Omega_1$
as well as the associated bilinear form
\[
\dint{\weakbf}_{\lipdom_\delta}
  \vcentcolon=
  \weakbf[
    \lipdom_\delta,g,\potential,\robinpotential,
    \dint{\dbdy}\lipdom_\delta,
    \dint{\nbdy}\lipdom_\delta,
    \dint{\rbdy}\lipdom_\delta
  ].
\]
Then for each integer $i \geq 1$ 
\begin{equation}\label{eq:MonotDir}
\eigenvalsym{\dint{\weakbf}_{\lipdom_\delta}}{i}{\grp}{\twsthom}\geq \eigenvalsym{\weakbf}{i}{\grp}{\twsthom},
\end{equation}
and we have
\begin{equation}\label{eq:ConvDir}
\lim_{\delta\to 0}
  \eigenvalsym{\dint{\weakbf}_{\lipdom_\delta}}{i}{\grp}{\twsthom}
=
\eigenvalsym{\weakbf}{i}{\grp}{\twsthom}.
\end{equation}
\end{proposition}

The conclusion simply relies on the fact that points have null $W^{1,s}$-capacity in $\R^n$ for $1\leq s\leq n$ and so, in particular, have null $W^{1,2}$-capacity in $\R^n$ for any $n\geq 2$; for the sake of completeness, we provide a self-contained argument focusing on the case of surfaces ($d=2$), where a logarithmic cutoff trick is required, and omit the simpler modifications for $d\geq 3$.

\begin{proof} 
Given any $u_{\delta}, v_{\delta}\in H^1_{\dint{\dbdy}\lipdom_\delta}(\lipdom_\delta)$, postulated to be $(\grp,\twsthom)$-invariant, it is standard to note that their extensions by $0$, say $\overline{u}_{\delta}, \overline{v}_{\delta}$ respectively, belong to $H^1_{\dbdy\lipdom}(\lipdom)$, that such functions are themselves $(\grp,\twsthom)$-invariant, and for any $\delta\in (0,\delta_0)$ there hold $\sk{u_\delta,v_\delta}_{L^2(\lipdom_\delta,g)}=\sk{\overline{u}_\delta,\overline{v}_\delta}_{L^2(\lipdom,g)}$ and $\dint{\weakbf}_{\lipdom_\delta}(u_\delta,u_\delta)=\weakbf(\overline{u}_\delta,\overline{u}_\delta)$. Hence, it follows at once from the variational characterization of eigenvalues, \eqref{symminmax}, that for each integer $i \geq 1$ we have indeed
$\eigenvalsym{\dint{\weakbf}_{\lipdom_\delta}}{i}{\grp}{\twsthom}\geq \eigenvalsym{\weakbf}{i}{\grp}{\twsthom}$, which is the first claim. 
Appealing again to the domain monotonicity, it actually suffices to check \eqref{eq:ConvDir} in the case when $K_\delta$ is in fact a union of metric balls, namely when we have equality in \eqref{eq:containment}, for $f_1=f_2$. To simplify the notation we can (without loss of generality, up to reparametrization) assume in fact $f_2(\delta)=\delta$ for any $\delta$ in the assumed domain. That said, given any $\overline{u}, \overline{v}\in H^1_{\dbdy\lipdom}(\lipdom)$, $(\grp,\twsthom)$-invariant, and $\delta>0$ (small as in the statement) one can simply define $u_\delta=\overline{u}\varphi_{\delta}, v_\delta=\overline{v}\varphi_{\delta} $ where (for $r\vcentcolon=d_g(p,q)$ and $p\in S$) we set
\[
\varphi_{\delta}(q)=
\begin{cases}
0 & \text{ if } r\leq \delta^{3/4} \\
3-4\frac{\log r}{\log \delta} & \text{ if } \delta^{3/4}\leq r\leq \delta^{1/2} \\
1 & \text{ otherwise. } 
\end{cases}
\]
It is then clear that $u_\delta, v_{\delta}\in H^1_{\dint{\dbdy}\lipdom_\delta}(\lipdom_\delta)$, that such functions are $(\grp,\twsthom)$-invariant, and, in addition, 
\[
\lim_{\delta\to 0}\dint{\weakbf}_{\lipdom_\delta}(u_\delta,u_\delta)
=T(\overline{u},\overline{u}), \
\lim_{\delta\to 0}\sk{u_\delta,v_\delta}_{L^2(\lipdom_\delta,g)}
=\sk{\overline{u},\overline{v}}_{L^2(\lipdom,g)}.
\]

Hence, again appealing to \eqref{symminmax}, we must conclude 
\begin{equation}\label{eq:USC}
\limsup_{\delta\to 0}
  \eigenvalsym{\dint{\weakbf}_{\lipdom_\delta}}{i}{\grp}{\twsthom}
  \leq 
\eigenvalsym{\weakbf}{i}{\grp}{\twsthom}.
\end{equation}
whence, combining this inequality with the one above, the conclusion follows.
\end{proof}

\begin{corollary}
\label{cor:MonotInd}
Given the setting and the assumptions of Proposition 
\ref{spectralContinuityWrtExcision}, we have
\[
\lim_{\delta\to 0}
  \symind{\grp}{\twsthom}(\dint{\weakbf}_{\lipdom_\delta})
=
\symind{\grp}{\twsthom}(\weakbf).
\]
\end{corollary}

\subsection{Conformal change in dimension two}
\label{subsec:conformal}

In this section we suppose, in addition
to the assumptions above, that
$d=\dim M=2$
and that we are given a smooth,
strictly positive,
$\grp$-invariant
function
$\conffact$ on $\closure{\lipdom}$.
Note that the above bilinear form $\weakbf$
of \eqref{bilinear_form_def}
is invariant under scaling, namely under the simultaneous transformations 
$g\mapsto \conffact^2g$,  
$\potential\mapsto \conffact^{-2}\potential$ and 
$\robinpotential \mapsto \conffact^{-1}\robinpotential$:
\begin{equation*}
\weakbf
  \bigl[
    \lipdom,
    \conffact^2 g,
    \conffact^{-2}\potential,
    \conffact^{-1} \robinpotential,
        \dbdy\lipdom,\nbdy\lipdom,\rbdy\lipdom\bigr]
=
\weakbf\bigl[\lipdom,g,\potential,\robinpotential,
        \dbdy\lipdom,\nbdy\lipdom,\rbdy\lipdom\bigr]
\end{equation*}
with the corresponding domains
$\sobd{\dbdy{\lipdom}}(\lipdom, \conffact^2g)$
and
$\sobd{\dbdy{\lipdom}}(\lipdom,g)$
agreeing as sets of functions
and having equivalent norms. This claim needs a clarification: the standard $H^1$-norms of  $\sobd{\dbdy{\lipdom}}(\lipdom, \conffact^2g)$ and $\sobd{\dbdy{\lipdom}}(\lipdom,g)$ are only equivalent up to constants that depend on the extremal (inf and sup) values of the conformal factor $\rho$. 

In general, the eigenvalues (as defined in Subsection \ref{subsec:BasicSpectral}) will be affected by the conformal scaling, and yet the index and nullity are nonetheless invariant when this operation is performed:

\begin{proposition}
[Invariance of index and nullity under conformal
change in dimension two]
\label{conformalInvariance}
With assumptions as in the preceding paragraph
\begin{equation*}
\symind{\grp}{\twsthom}(\weakbf, \conffact^2 g)
  =
  \symind{\grp}{\twsthom}(\weakbf, g)
\quad \mbox{ and } \quad
\symnul{\grp}{\twsthom}(\weakbf, \conffact^2 g)
  =
  \symnul{\grp}{\twsthom}(\weakbf, g).
\end{equation*}
\end{proposition}

\begin{proof}
By definition
$u \in \eigenspsym{\weakbf,g}{=0}{\grp}{\twsthom}$
if and only if $u$ is $(\grp,\twsthom)$-invariant
and $\weakbf(u,v)=0$
for all $v \in \sobd{\dbdy\lipdom}(\lipdom,g)$
(and likewise if each $g$ is replaced
by $\rho^2g$),
so the nullity equality is clear.
For the index,
because we can reverse
the roles of $g$ and $\conffact^2 g$
by replacing $\conffact$
with $\conffact^{-1}$,
it suffices to check
that the claim holds
with $\geq$ in place of $=$.
This follows at once
from the min-max characterization
\eqref{symminmax}
applied to the
$(\grp,\twsthom)$-eigenvalues
of $(\weakbf,\conffact^2 g)$,
by considering the ``competitor'' subspace
$
 \eigenspsym{\weakbf,g}{<0}{\grp}{\twsthom}
$ in the minimization problem therein, for $i=\symind{\grp}{\twsthom}(\weakbf, g)$.
\end{proof}

\section{Free boundary minimal surfaces in the ball: a first application}\label{sec:FBMSgen}

From now on, we specialize our study to the case when
$\closure{\Omega}=M$ is a properly embedded
free boundary minimal surface, henceforth denoted by $\Sigma$,
of the closed unit ball
$\B^3 \vcentcolon= \{(x,y,z)\in\R^3\st x^2+y^2+z^2 \leq 1\}$
in Euclidean space $(\R^3,\geuc)$.
Observe that, by the maximum principle,
every embedded free boundary minimal surface
is properly embedded.

As anticipated in the introduction, our task here will be to obtain quantitative estimates on the Morse index of free boundary minimal surfaces, hence our Schr\"{o}dinger operator
is the Jacobi (or stability) operator on $\Sigma$
acting on functions
subject to the Robin condition
\begin{equation}
\label{fbmsRobinConditionInBall}
du(\conormal{\Sigma}{\geuc})=u\quad\text{ on }\partial \Sigma,
\end{equation}
namely: $\potential=\abs{\twoff{\Sigma}}^2$,
the squared norm of the second fundamental form of $\Sigma$, and
$\dbdy\Sigma=\nbdy\Sigma=\emptyset$,
$\rbdy\Sigma=\partial\Sigma$,
$\robinpotential=1$. Correspondingly, as our bilinear form $\weakbf$
we will consider
the index (or stability or Jacobi) form of $\Sigma$,
which we will denote by $\indexform{\Sigma}$. 
We define the index and nullity of $\Sigma$ in the usual way, setting
\begin{equation*}
\ind(\Sigma)
  \vcentcolon=
  \ind(\indexform{\Sigma})
\quad \text{ and } \quad
\nul(\Sigma)
  \vcentcolon=
  \nul(\indexform{\Sigma}),
\end{equation*}
and we likewise define
the $\grp$-equivariant index and nullity of $\Sigma$,
$\equivind{\grp}(\Sigma)$
and $\equivnul{\grp}(\Sigma)$, in the sense of \eqref{eq:MeaningEquivariance},
when given a group $\grp<\Ogroup(3)$
of symmetries of $\Sigma$ one considers the associated sign homomorphism. More generally, we will also study the
$(\grp,\twsthom)$-index and $(\grp,\twsthom)$-nullity
of $\Sigma$,
$\symind{\grp}{\twsthom}(\Sigma)$
and $\symnul{\grp}{\twsthom}(\Sigma)$,
when given a group $\grp$ and, further, a homomorphism
$\twsthom \colon \grp \to \Ogroup(1)$
(thus, in either case, these expressions are to be understood by replacing $\Sigma$ by $\indexform{\Sigma}$).

It has already been mentioned above how
general lower bounds for the index,
linear in the topological data (genus and number of boundary components),
have been obtained in
\cite{AmbCarSha18-Index}, and by Sargent in \cite{Sar17} in the special case when the ambient manifold is a convex body in Euclidean $\R^3$. We begin this section by presenting an alternative lower bound
(Proposition \ref{indBelowPyr} below)
in terms of symmetries,
which, though much less general in nature,
nevertheless yields sharper lower bounds
for many of the known examples (in terms of the coefficients describing the linear growth rate as a function of the topological data).
Before proceeding, we pause to explain some notation we will find convenient.

\paragraph{Cylindrical coordinates and wedges.}
We shall describe points in Euclidean $\R^3$, endowed with standard Cartesian coordinates $(x,y,z)$, also in terms of cylindrical coordinates $(r,\theta,z)$, so that the point with cylindrical coordinates $(r_0,\theta_0,z_0)$ has Cartesian coordinates
$(x,y,z)=(r_0 \cos \theta_0,r_0 \sin \theta_0,z_0)$.  However we wish to stress that, for our purposes, it will be convenient to allow arbitrary real values for both $r$ and $\theta$; thus the triples $(r,\theta,z)$ and $(-r,\theta+\pi,z)$ describe the same point in Euclidean space.
Given real numbers $\alpha \leq \beta$,
we also define the closed wedge
\begin{equation}
\label{wedge}
\Wedge{\alpha}{\beta}
\vcentcolon=
\{
  (r \cos \theta, r \sin \theta, z)
  \st
  r \geq 0, ~
  \theta \in \IntervaL{\alpha,\beta}, ~
  z \in \R
\},
\end{equation}
with the half-plane
$\Wedge{\alpha}{\alpha}$
accommodated as a degenerate wedge. 
In particular, our convention implies  
\[
\{\theta=\alpha\}=\Wedge{\alpha}{\alpha}\cup\Wedge{\alpha+\pi}{\alpha+\pi}.
\]

\paragraph{Notation for symmetries.}
Given a plane $\Pi \subset \R^3$ through the origin,
we write $\refl_\Pi \in \Ogroup(3)$
for reflection through $\Pi$.
Similarly, given a directed line $\xi \subset \R^3$ through the origin
and an angle $\theta \in \R$,
we write $\rot_\xi^\theta$
for rotation about $\xi$ through angle $\alpha$
in the usual right-handed sense.
Typically we will be interested
not exclusively in
such a rotation $\rot_\xi^\theta$
but rather in the cyclic subgroup it generates,
with the result
that it will never really be important
to associate a direction to $\xi$.
Given symmetries
$\mathsf{T}_1, \ldots, \mathsf{T}_n \in \Ogroup(3)$,
we write
$\sk{\mathsf{T}_1, \ldots, \mathsf{T}_n}$
for the subgroup they generate.

The order-$2$ groups generated by reflections through planes
will figure repeatedly
in the sequel (beginning with the following proposition),
so for succinctness of notation,
given a plane $\Pi \subset \R^3$ through the origin,
we agree to set
$
 \cycgrp{\Pi}
 \vcentcolon=
 \sk{\refl_\Pi}
$.
In such context, consistently with the general convention we defined above, we will employ the apex $+$ (respectively: $-$) to denote functions that are even (respectively: odd) with respect to the reflection through $\Pi$.
Similarly (but less frequently),
if $\xi$ is a line through the origin in $\R^3$,
we will write
$\cycgrp{\xi}$
for the order-$2$ group
generated by reflection
$\refl_\xi$
through $\xi$
(equivalently rotation through angle $\pi$
in either sense about $\xi$).

We also pause to name the following
three subgroups of $\Ogroup(3)$,
which will be realized as subgroups
of the symmetry groups
of the examples we study below
and which partly pertain
to the statement of the next proposition:
for each integer $k \geq 1$ we set
\begin{equation}
\label{standardSymGrps}
\begin{aligned}
&\pyr_k
  \vcentcolon=
  \sk[\Big]{
    \refl_{\{\theta = -\frac{\pi}{2k}\}}, ~
    \refl_{\{\theta = \frac{\pi}{2k}\}}
  }
  \quad
  &&\text{(pyramidal group of order $2k$)},
\\
&\pri_k
  \vcentcolon=
  \sk[\Big]{
    \refl_{\{\theta = -\frac{\pi}{2k}\}}, ~
    \refl_{\{\theta = \frac{\pi}{2k}\}}, ~
    \refl_{\{z=0\}}
  }
  \quad
  &&\text{(prismatic group of order $4k$)},
\\
&\apr_k
  \vcentcolon=
  \sk[\Big]{
    \refl_{\{\theta = \frac{\pi}{2k}\}}, ~
    \rot_{\{y=z=0\}}^\pi
  }
  \quad
  &&\text{(antiprismatic group of order $4k$)}.
\end{aligned}
\end{equation}
Note in particular that we have
$
\pyr_k
=
\pri_k \cap \apr_k
$.

\begin{remark}
The above three groups
are so named because
they are the (maximal) symmetry groups
of, respectively,
a right pyramid, prism, or antiprism
over a regular $k$-gon.
See e.\,g. Section 2 of \cite{CSWnonuniqueness}
for pictures and further details,
but we caution that
the above definition
of the subgroup $\pri_k$
differs slightly from that given in
\cite{CSWnonuniqueness}:
the two subgroups
are conjugate to one another
via rotation through angle
$\pi/(2k)$ about the $z$-axis.
\end{remark}

With this terminology and notation in place, we can then proceed with the aforementioned lower index bound, which illustrates the Montiel--Ros methodology as developed in Section \ref{sec:theory} and is interesting in its own right.

\begin{proposition}
[Index lower bounds under pyramidal and prismatic symmetry;
cf. \cites{ChoeVision90,KapouleasWiygulIndex}]
\label{indBelowPyr}
Let $\Sigma$ be a connected, embedded
free boundary minimal surface in $\B^3$.
Assume that $\Sigma$ is not a disc
or critical catenoid,
that $\Sigma$ is invariant
under reflection through a plane $\Pi_1$,
and that
$\Sigma$ is also invariant
under rotation
through an angle
$\alpha \in \interval{0,2\pi}$
about a line $\xi \subset \Pi_1$.
Then $\alpha$ 
is a rational multiple of $2\pi$,
there is a largest integer $k \geq 2$
such that rotation about $\xi$
through angle $2\pi/k$
is also a symmetry of $\Sigma$,
and
\begin{enumerate}[label={\normalfont(\roman*)}]
\item\label{prop:indBelowPyr-i}
  $\ind(\Sigma) \geq 2k-1$,
\item\label{prop:indBelowPyr-ii}
 $\symind{\cycgrp{\Pi_1}}{-}(\Sigma) \geq k-1$,
 and
\item\label{indBelowPri}
  if $\Sigma$ is additionally invariant
under reflection through a plane $\Pi_\perp$ orthogonal to $\xi$,
then in fact
$\symind{\cycgrp{\Pi_\perp}}{+}(\Sigma) \geq 2k-1$.
\end{enumerate}
\end{proposition}

Note that the symmetries assumed in the preamble
of Proposition \ref{indBelowPyr}
generate, up to conjugacy in $\Ogroup(3)$,
the group $\pyr_k$ from \eqref{standardSymGrps},
while one instead obtains
(again up to conjugacy)
the group $\pri_k$
by adjoining
the additional symmetry
assumed in item \ref{indBelowPri}.

The proof below is an abstraction
and transplantation to the free boundary setting
of some index lower bounds obtained
in the course of \cite{KapouleasWiygulIndex}
and
drawing on ideas from \cite{MontielRos}.
The estimates ultimately
depend on a lower bound
on the number of nodal domains
of a suitable Jacobi field,
which was also the basis
for earlier index estimates
(of complete minimal surfaces in $\R^3$
and closed minimal surfaces in $\Sp^3$)
established by Choe in \cite{ChoeVision90}.

\begin{proof}
By excluding the discs and critical catenoids
we ensure that $\Sigma$ is not
$\Sp^1$-invariant about $\xi$,
implying the claim on $\alpha$
and the existence
of the rotational symmetry
about $\xi$ through angle
of the form $2\pi/k$,
as follows.
First,
if the cyclic subgroup
generated by rotation
about $\xi$ through angle $\alpha$
were not finite,
then it would be dense
in the $\SOgroup(2)$
subgroup of rotations about $\xi$,
but the symmetry group of $\Sigma$
is closed in $\Ogroup(3)$;
yet, as already observed,
our assumptions ensure that
$\Sigma$ has no $\SOgroup(2)$
symmetry subgroup.
Thus $\alpha$ must be
a rational multiple of $2\pi$, as claimed.
Now let $\beta$ be the least angle
in $\interval{0,2\pi}$
through which rotation about $\xi$
is generated by the assumed rotational symmetry
through angle $\alpha$,
and let $k$ be the least positive integer
such that $k\beta \geq 2\pi$.
Then rotation through
angle $k\beta-2\pi$,
which lies in $\Interval{0,\beta}$,
is also generated
by the assumed rotational symmetry.
The presumed minimality of $\beta$
then forces $\beta=2\pi/k$.

By composing the assumed symmetries,
it follows that $\Sigma$ is also invariant
under reflection through each of the
$k-1$ planes $\Pi_2,\ldots,\Pi_k$
containing $\xi$ and there meeting $\Pi_1$
at angle an integer multiple of $\pi/k$.
Now suppose $\Pi \in \{\Pi_i\}_{i=1}^k$.
We necessarily have
$\Pi \cap \Sigma \neq \emptyset$
(for example
since $\Pi$ separates $\B^3$ into two components
and is a plane of symmetry for $\Sigma$,
which is assumed to be connected).
Because $\Pi$ is a plane of symmetry
and $\Sigma$ is embedded,
these two surfaces must intersect
either orthogonally or tangentially,
but in the latter case
$\Sigma$ must be a disc,
which possibility we have excluded by assumption;
consequently, the intersection is orthogonal.
Moreover, by the symmetries 
each of the $2k$ components
$W_1, \ldots, W_{2k}$
of $\B^3 \setminus \bigcup_{i=1}^k \Pi_i$
then has non-trivial intersection
$\lipdom_i \vcentcolon= \Sigma \cap W_i$
with $\Sigma$. Without loss of generality, let us agree to label the domains under consideration in counterclockwise order such that $\lipdom_1,\ldots,\lipdom_k$ all lie on the same side of $\Pi_1$.

Note that the members of the family $\{\lipdom_i\}_{i=1}^{2k}$
are pairwise isometric
and each is connected.
(Indeed,
$\Sigma$ is itself connected,
so any two points in any single $\lipdom_i$
can be joined by some path in $\Sigma$,
but this path can leave $\lipdom_i$
only through the latter's intersection
with planes of symmetry,
so we can always produce a path connecting the two points
that is entirely contained in $\lipdom_i$,
by repeated reflection and replacement, if necessary.)
Furthermore, each $\lipdom_i$
has Lipschitz boundary contained in
$\Sp^2 \cup \bigcup_{i=1}^k \Pi_i$, because the intersection of $\Sigma$ with either $\Sp^2$ and any of the planes $\Pi_1,\ldots,\Pi_k$ is orthogonal (thus transverse),
and exactly $k$ of the $\lipdom_i$
lie on each side of $\Pi_1$.

Next,
letting $\kappa_\xi$ be a choice of (scalar-valued) Jacobi field
on $\Sigma$ induced by the rotations about $\xi$
and again using
the fact that $\Sigma$ is not rotationally symmetric
(and so, in particular, not planar either),
we conclude that $\kappa_\xi$ vanishes on
$\Sigma \cap \bigcup_{i=1}^k \Pi_i$ (because of the aforementioned orthogonality)
but does not vanish identically on any $\lipdom_i$.
As a result,
imposing,
for each $i$,
the Robin condition
\eqref{fbmsRobinConditionInBall}
on $\Sp^2 \cap \partial \lipdom_i$
and
the Dirichlet condition
on $\partial \lipdom_i \cap \bigcup_{i=1}^k \Pi_i$,
the corresponding nullity of $\lipdom_i$
is at least $1$.
An appeal to item \ref{mrIsoPieces:below}
of Corollary \ref{cor:MontielRosHuman} (for our claims \ref{prop:indBelowPyr-i} and \ref{indBelowPri}) and of Proposition \ref{cor:MontielRosHuman} (for our claim \ref{prop:indBelowPyr-ii})  
now completes the proof. Specifically:
\begin{itemize}
\item for our claim \ref{prop:indBelowPyr-i} we consider the partition of $\Sigma$ into the $2k$ domains $\Omega_1,\ldots,\Omega_{2k}$, and take $G$ to be the trivial group;
\item for our claim \ref{prop:indBelowPyr-ii} we take $G=\sk{\refl_{\Pi_1}}$ to be the group with two elements (as in Example \ref{ex:Group2El}), the homomorphism determined by $\sigma(\refl_{\Pi_1})=-1$ (thereby imposing \emph{odd} symmetry) and, correspondingly, we consider the partition of $\Sigma$ into $k$ domains obtianed by equivariant pairing, i.\,e. by taking $\lipdom_{i+1} \cup \lipdom_{2k-i}$ for $i=0,\ldots,k-1$;
\item for our claim \ref{indBelowPri} we consider the partition of $\Sigma$ into the $2k$ domains $\Omega_1 \ldots,\Omega_{2k}$, take $G=\sk{\refl_{\Pi_\perp}}$ to be the group with two elements and the homomorphism determined by $\sigma(\refl_{\Pi_\perp})=+1$ (thereby imposing \emph{even} symmetry).
\end{itemize}
Thereby the proof is complete.
\end{proof}

\section{Effective index estimates for two sequences of examples}
\label{sec:CSW}

\subsection{Review of the construction and lower index bounds}

Like we have already alluded to in the introduction, in \cite{CSWnonuniqueness} two families of
embedded free boundary minimal surfaces
in $\B^3$ were constructed by desingularizing
(in the spirit of
\cite{KapouleasEuclideanDesing})
the configurations
$-\K_0 \cup \K_0$
and $-\K_0 \cup \B^2 \cup \K_0$,
where $\K_0$ is the intersection
with $\B^3$ of a certain catenoid
having axis of symmetry $\{x=y=0\}$
and meeting $\partial \B^3$
(not orthogonally)
along the equator $\partial \B^2$
and orthogonally along one additional circle
of latitude at height $h>0$.

\begin{proposition}[Existence and basic properties of $\K_0$]
\label{K0ExistenceAndBasicProps}
There exists a minimal annulus $\K_0$ which is properly embedded in $\B^3$ and intersects the unit sphere $\partial\B^3$ exactly along the equator 
$\partial_0\K_0\vcentcolon=\partial\B^3\cap \{z=0\}$ and orthogonally along a circle of latitude at height $z=h\approx 0.87028$ which we denote by  $\partial_\perp\K_0\vcentcolon=\partial\K_0 \setminus\partial_0\K_0$. 
Moreover, $\K_0$ coincides with the surface of revolution of the graph of $r\colon[0,h]\to\interval{0,1}$ given by 
$r(\zeta)=(1/a)\cosh(a\zeta-s)$ 
for suitable $a\approx 2.3328$ and $s\approx 1.4907$. 
\end{proposition}

\begin{proof}
The existence of $\K_0$ is proven in \cite[Lemma 3.3]{CSWnonuniqueness}.  
For the numerical values of $a$, $h$ and $s$ we refer to \cite[Remark 3.9]{CSWnonuniqueness}.   
\end{proof}

That being said, these are (somewhat simplified) versions of the main existence results we proved in \cite{CSWnonuniqueness}.

\begin{theorem}
[Desingularizations of $-\K_0 \cup \K_0$
 \cite{CSWnonuniqueness}]
\label{zerogenExistenceAndBasicProps}
For each sufficiently large integer $n$
there exists in $\B^3$ a properly embedded
free boundary minimal surface
$\zerogen_n$
that has genus $0$, exactly
$n+2$ boundary components
and is invariant under the prismatic group $\pri_n$ from \eqref{standardSymGrps}.
Moreover
$\zerogen_n$ converges to $-\K_0 \cup \K_0$
in the sense of varifolds, with unit multiplicity, and smoothly away from the equator, as $n \to \infty$.
\end{theorem}

\begin{theorem}
[Desingularizations of $-\K_0 \cup \B^2 \cup \K_0$
 \cite{CSWnonuniqueness}]
\label{threebcExistenceAndBasicProps}
For each sufficiently large integer $m$
there exists in $\B^3$ a properly embedded
free boundary minimal surface
$\threebc_m$ that has genus $m$, exactly
$3$ boundary components
and is invariant under the antiprismatic group
$\apr_{m+1}$
from \eqref{standardSymGrps}.
Moreover
$\threebc_m$ converges to $-\K_0 \cup \B^2 \cup \K_0$
in the sense of varifolds, with unit multiplicity, and smoothly away from the equator, as $m \to \infty$.
\end{theorem}

\begin{proposition}[Lower bounds by symmetry on the index of the examples of \cite{CSWnonuniqueness}]
\label{pro:CheapLowerBods}
There exist $n_0,m_0>0$ such that we have the following index estimates for all integers $n>n_0$ and $m>m_0$
\begin{equation*}
\symind{\cycgrp{\{z=0\}}}{+}(\zerogen_n) \geq 2n-1
\quad \mbox{ and } \quad
\ind(\threebc_m) \geq 2m+1.
\end{equation*}
\end{proposition}

\begin{proof}
As stated in Theorem \ref{zerogenExistenceAndBasicProps}, $\zerogen_n$ is invariant under the action of the prismatic group $\pri_n$ which is generated by the reflections through the vertical planes $\{\theta=-\pi/(2n)\}$ and $\{\theta=\pi/(2n)\}$ and through the horizontal plane $\{z=0\}$. 
As a composition of the first two reflections, $\pri_n$ also contains the rotation by angle $2\pi/n$ about the vertical axis $\xi_0=\{r=0\}$. 
Applying Proposition~\ref{indBelowPyr}~\ref{indBelowPri} with $k=n$, $\xi=\xi_0$, $\Pi_1=\{\theta=\pi/(2n)\}$ and $\Pi_\perp=\{z=0\}$ we obtain 
\begin{align*}
\symind{\cycgrp{\{z=0\}}}{+}(\zerogen_n)&\geq 2n-1.
\end{align*}
Similarly, Theorem \ref{threebcExistenceAndBasicProps} states that $\threebc_m$ 
is invariant under the action of the antiprismatic group $\apr_{m+1}$ which contains the reflection through the vertical plane $\{\theta=\pi/(2(m+1))\}$ and also the rotation by angle $2\pi/(m+1)$ about the vertical axis $\xi_0$. 
Applying Proposition~\ref{indBelowPyr}~\ref{prop:indBelowPyr-i} then yields 
\begin{align*}
\ind(\threebc_m)&\geq 2m+1. \qedhere
\end{align*}
\end{proof}

In terms of topological data, the previous proposition (compared to \cite{AmbCarSha18-Index}) provides a coefficient 2 for the growth rate of the Morse index of $\zerogen_n$ (respectively: $\threebc_m$) with respect to the number of boundary components (respectively: of the genus), modulo an additive term. In fact, the lower bound on the Morse index of $\zerogen_n$ can be further improved via the following observation, which pertains the \emph{odd} contributions to the index instead (again with respect to reflections across the $\{z=0\}$ plane in $\R^3$); incidentally this is also an example of application of Proposition \ref{mr} to a collection of domains that are \emph{not} pairwise isometric.

\begin{proposition}\label{pro:OddLowerBods}
There exists $n_0>0$ such that we have the following index estimates for all integers $n>n_0$
\begin{equation*}
\symind{\cycgrp{\{z=0\}}}{-}(\zerogen_n)\geq3.
\end{equation*}
\end{proposition}

\begin{figure}%
\centering
\tdplotsetmaincoords{72}{112.5}
\begin{tikzpicture}[scale=\unitscale,tdplot_main_coords,line cap=round,line join=round,thin]
\draw(0,0,0)node[inner sep=0]{\includegraphics[width=\fbmsscale\textwidth,page=3]{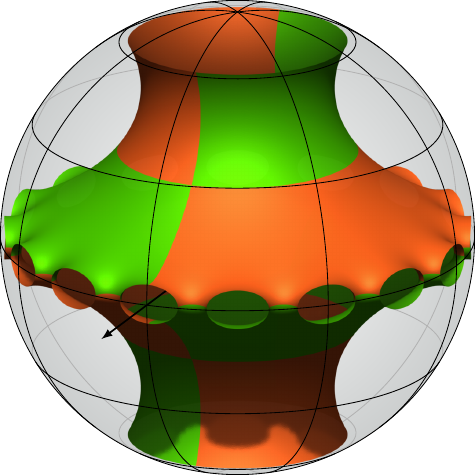}};
\draw(1,0,0)++(0.5,0,0)node[below left,inner sep=1pt]{$\xi$};
\end{tikzpicture} 
\hfill
\pgfmathsetmacro{\b}{0}
\pgfmathsetmacro{\a}{2.3328}
\pgfmathsetmacro{\h}{0.8703}
\pgfmathsetmacro{\t}{0.3616216}
\pgfmathsetmacro{\radius}{sqrt(\t*\t+(cosh(\a*\t-\a*\b-acosh(\a*sqrt(1-\b*\b)))/\a)^2)}
\begin{tikzpicture}[line cap=round,line join=round,scale=\unitscale,remember picture]
\fill[black!3](0,0)circle(1);
\draw[->,thick,black!66](15:0.1)arc(15:345:0.1);
\begin{scope}[very thick,variable=\z,samples=50,line cap=rect]
\draw[domain=\b:\t,red!75!black]plot({cosh(\a*\z-\a*\b-acosh(\a*sqrt(1-\b*\b)))/\a},{\z});
\draw[domain=\b:\t,green!75!black]plot({-cosh(\a*\z-\a*\b-acosh(\a*sqrt(1-\b*\b)))/\a},{\z});
\draw[domain=\t:\h,green!75!black]plot({cosh(\a*\z-\a*\b-acosh(\a*sqrt(1-\b*\b)))/\a},{\z});
\draw[domain=\t:\h,red!75!black]plot({-cosh(\a*\z-\a*\b-acosh(\a*sqrt(1-\b*\b)))/\a},{\z});
\draw[domain=\b:\t,green!75!black]plot({cosh(\a*\z-\a*\b-acosh(\a*sqrt(1-\b*\b)))/\a},{-\z});
\draw[domain=\b:\t,red!75!black]plot({-cosh(\a*\z-\a*\b-acosh(\a*sqrt(1-\b*\b)))/\a},{-\z});
\draw[domain=\t:\h,red!75!black]plot({cosh(\a*\z-\a*\b-acosh(\a*sqrt(1-\b*\b)))/\a},{-\z});
\draw[domain=\t:\h,green!75!black]plot({-cosh(\a*\z-\a*\b-acosh(\a*sqrt(1-\b*\b)))/\a},{-\z});
\end{scope}
\draw[->](-1,0)--(1.12,0)node[above left,inner sep=0]{$r\vphantom{|}$};
\draw[->](0,-0)--(0,1.12)node[below right,inner sep=0]{$~z$};
\draw(-1/\a,0.6)node[left]{$\K_0$};
\draw(-1/\a,-0.6)node[left]{$-\K_0$};
\draw plot[hdash](0,\h)node[left]{$h$};
\draw plot[hdash](0,\t)node[left]{};
\draw plot[bullet]({ sqrt(1-\b*\b)},\b);
\draw plot[bullet]({-sqrt(1-\b*\b)},\b);
\draw plot[bullet]({sqrt(1-\h*\h)},\h);
\draw plot[bullet]({-sqrt(1-\h*\h)},\h);
\draw plot[bullet]({sqrt(1-\h*\h)},-\h);
\draw plot[bullet]({-sqrt(1-\h*\h)},-\h);
\draw plot[bullet](0,0);
\draw(0,0)circle(1); 
\draw[dotted](0,\h)--({sqrt(1-\h*\h)},\h);
\draw[dotted](0,\t)--({sqrt(\radius*\radius-\t*\t)},\t);
\draw[dotted,thick](0,0)circle(\radius);
\end{tikzpicture}
\caption{Nodal domains of the function induced by rotations around the symmetry axis $\xi$.}%
\label{fig:nodal}%
\end{figure}

\begin{proof}
Let $\Pi_{1}$ denote a vertical plane of symmetry, passing through the origin, of the surface $\zerogen_n$ (which, we recall, has prismatic symmetry $\pri_n$), let $\xi$ be the line obtained as intersection of such a plane with $\{z=0\}$ and let finally $\Pi_2=\xi^\perp$ be the vertical plane, again passing through the origin, that is orthogonal to $\Pi_1$. 
Consider on $\zerogen_n$ the function $\kappa_{\xi}=K_{\xi}\cdot\nu$ where $K_{\xi}$ is the Killing vector field associated to rotations around $\xi$ (oriented either way) and $\nu$ is a choice of the unit normal to the surface in question. 
Clearly, the flow of $K_{\xi}$ generates a curve of free boundary minimal surfaces around $\zerogen_n$, hence the function $\kappa_{\xi}$ lies in the kernel of the Jacobi operator of $\zerogen_n$ and satisfies the natural Robin boundary condition along the free boundary.
Furthermore, concerning its nodal set, we first note it contains the curves $\zerogen_n\cap \{z=0\}$, and $\zerogen_n\cap \Pi_1$. 
We also claim that, for any sufficiently large $n$, the function $\kappa_{\xi}$ changes sign along the connected arc 
\begin{align}\label{eqn:20221029-arc}
\zerogen_n \cap \Pi^{+}_2 \cap\{z\geq z_0\}
\end{align}
where $\Pi^{+}_2$ denote either of the half-planes determined by $\Pi_1$ on $\Pi_2$ and $z_0>0$ is any sufficiently small value (as we are about to describe, stressing that we can choose it independently of $n$).
Since one has smooth convergence of $\zerogen_n$ to $-\K_0 \cup \K_0$ as $n\to\infty$ away from the equator, it suffices to verify an analogous claim for $\K_0$. 
In fact, it then follows from an explicit calculation that the function induced by rotations around the symmetry axis $\xi$ (the analogue of $\kappa_\xi$ on $\K_0$) has opposite signs on the two endpoints of the arc $\K_0\cap\Pi^{+}_2$ (see Figure \ref{fig:nodal}, right image), and so -- assuming without loss of generality it is negative on the equatorial point -- by continuity there exists $\overline{z}_0>0$ such that the same function is also strictly negative at all points of $\K_0\cap\Pi^{+}_2$ at height $z_0\in [0,\overline{z}_0]$. 
In particular, we can indeed choose one such value $z_0\in (0,\overline{z}_0)$ once and for all.

Hence, appealing to the aforementioned smooth convergence, by the intermediate value theorem for any sufficiently large $n$ there must be a point along the arc \eqref{eqn:20221029-arc} where $\kappa_{\xi}$ vanishes. 
Now, standard results about the structure of the nodal sets of eigenfunctions of Schrödinger operators ensure that such a zero is not isolated, but is either a regular point of a smooth curve or a branch point out of which finitely many smooth arcs emanate. 
In either case, combining all facts above we must conclude that on $\zerogen_n\cap \left\{z\geq 0\right\}$ the function $\kappa_\xi$ has at least four nodal domains, and thus an application of Proposition \ref{mr} with $t=0$, $G=\sk{\refl_{\Pi}}$ for $\Pi=\left\{z=0\right\}$ and $\sigma(\refl_{\Pi})=-1$ ensures the conclusion. 
\end{proof}

\begin{remark}\label{rem:NonEquivBound}
Note that the very same argument would lead, when applied with no equivariance constraint at all (i.\,e. when $G$ is the trivial group) to the conclusion that for any sufficiently large $n$ the index of $\zerogen_n$ is bounded from below by $7$, which however is a lot worse than the bound provided by combining Proposition \ref{pro:CheapLowerBods}  with Proposition \ref{pro:OddLowerBods}. 
Furthermore, we note that one can show that the function $\kappa_{\xi}$ has \emph{exactly} $8$ nodal domains and not more, as visualized in Figure \ref{fig:nodal}.
\end{remark}

\begin{remark}
Concerning the sharpness of the estimate given in Proposition \ref{pro:OddLowerBods}, we note that
numerical simulations of $\K_0$ with fixed lower boundary $\partial_0\K_0$ and upper boundary $\partial_\perp\K_0$ constrained to the unit sphere indicate that it has in fact index equal to $3$. 
Roughly speaking, one negative direction comes from ``pinching'' the catenoidal neck and the other two negative directions correspond to ``translations'' of $\partial_\perp\K_0$ on the northern hemisphere. 
\end{remark}

The rest of this section is aimed at obtaining \emph{upper} bounds on the Morse index of our examples, which is a more delicate task and one that relies crucially not only on the symmetries of the surfaces in question but also on the way they were actually constructed (which we encode in suitable convergence results).

\subsection{Equivariant index and nullity of the models}
For upper bounds we will exploit the regionwise convergence of the two families to the models glued together in their construction.
Therefore we first study the index and nullity on these models.

\paragraph{Equivariant index and nullity of $\K_0$.}
We begin with a summary of the properties of the minimal annulus $\K_0$ we will need. 
Let $\partial_0\K_0=\partial\K_0 \cap \{z=0\}$ and $\partial_\perp\K_0=\partial\K_0 \setminus \partial_0\K_0$ be as introduced in Proposition \ref{K0ExistenceAndBasicProps} so that $\partial_\perp\K_0$ is the boundary component along which $\K_0$ meets the sphere $\partial\B^3$ orthogonally. 
Referring to equation \eqref{bilinear_form_def}, we define
\begin{equation*}
\indexform{\K_0}_\neum
  \vcentcolon=
  \weakbf\Bigl[
    \K_0, ~
    g^{\K_0}, ~
    \potential
      \vcentcolon=
      \abs[\big]{\twoff{\K_0}}^2, ~
    \robinpotential \vcentcolon= 1, ~
    \dbdy\K_0 \vcentcolon= \emptyset, ~
    \nbdy\K_0 \vcentcolon= \partial_0\K_0, ~
    \rbdy\K_0 \vcentcolon= \partial_\perp\K_0
  \Bigr]
\end{equation*}
(where we abuse notation 
in that by $\K_0$
we really mean its topological interior)
to be the Jacobi form
of $\K_0$ subject to
the natural geometric Robin condition
\eqref{fbmsRobinConditionInBall}
on $\partial_\perp\K_0$
and to the Neumann condition
on $\partial_0\K_0$.
Clearly,
for each $k \geq 1$
the pyramidal group $\pyr_k$
from \eqref{standardSymGrps}
preserves $\K_0$
and each of its boundary
components individually.

\begin{lemma}
[$\pyr_k$-equivariant index and nullity of $\K_0$]
\label{K0IndexAndNullity}
With notation as above,
for each sufficiently large integer $k$
\begin{equation*}
\equivind{\pyr_k}(\indexform{\K_0}_\neum)=1
\quad \text{ and } \quad
\equivnul{\pyr_k}(\indexform{\K_0}_\neum)=0.
\end{equation*}
\end{lemma}

\begin{proof}
We shall start by recalling \cite[Lemma 4.4]{CSWnonuniqueness}, which states that when imposing the Dirichlet condition on $\partial_0\K_0$ and the Robin condition on $\partial_\perp\K_0$, 
then the Jacobi operator acting on $\pyr_k$-equivariant functions on $\K_0$ is invertible provided that $k$ is sufficiently large, which means that the equivariant nullity vanishes in this case. 
Considering the coordinate function $u=z$ on $\K_0$, which is harmonic, satisfies the Dirichlet condition on $\partial_0\K_0$ and the Robin condition on $\partial_\perp\K_0$, it is also evident that the equivariant index is at least $1$ in this case (cf. \cite[Lemma 7.2]{CSWnonuniqueness}). 
This implies that when instead the Neumann condition is imposed on $\partial_0\K_0$, the equivariant index is again at least $1$. 
Below we prove that it is exactly $1$ and the equivariant nullity is exactly $0$ in the Neumann case by showing that the second eigenvalue is strictly positive. 
(We note here, incidentally, that this information also proves that a posteriori the equivariant index is also exactly $1$ in the case that a Dirichlet condition is imposed on $\partial_0\K_0$.)

Let $a,h,s>0$ and $r(\zeta)=(1/a)\cosh(a\zeta-s)$ be as in Proposition \ref{K0ExistenceAndBasicProps}. 
In particular, we have $(r')^2+1=\cosh^2(a\zeta-s)$. 
Thus, when $\K_0$ is parametrized as a surface of revolution in terms of the coordinates $(\theta, \zeta)$ with profile function $r(\zeta)$, the metric $g_{\K_0}$ and the squared norm of the second fundamental form $A_{\K_0}$ on $\K_0$ are given by 
\begin{align*}
g_{\K_0}&=\bigl((r')^2+1\bigr)\,d\zeta^2+r^2\,d\theta^2,
\\
\abs{A_{\K_0}}^2&=\frac{(-r'')^2}{((r')^2+1)^3}+\frac{1}{((r')^2+1)^2r^2}
=\frac{a^2+a^{-2}}{\cosh^4(a\zeta-s)}.
\end{align*}
The outward unit conormal along $\partial_\perp\K_0=\K_0\cap\{\zeta=h\}$ is given by 
\[
\eta_{\K_0}
=\frac{1}{\sqrt{(r')^2(h)+1}}\partial_\zeta
=\frac{1}{\cosh(ah-s)}\partial_\zeta
=\frac{1}{ar(h)}\partial_\zeta.
\]
Assume, for the sake of a contradiction, that $\lambda_2=\lambda_2^{\pyr_k,\nrmlsgn{}}\leq 0$, where we are considering the spectrum of the Jacobi operator of $\K_0$ acting on $\pyr_k$-equivariant functions (cf. Example \ref{ex:SelfCongr}), and subject to the boundary conditions described above. 
Then, by first invoking the Courant nodal domain theorem as in the proof of \cite[Lemma 4.4]{CSWnonuniqueness} we may assume that the associated eigenfunction $u_2$ is rotationally symmetric provided that $k$ is sufficiently large, i.\,e. $u_2$ only depends on $\zeta$ and not on $\theta$.  

That said, let $u$ be a function on $\K_0$ which is rotationally symmetric, i.\,e. constant in $\theta$. 
Then 
\begin{align*}
\Delta_{\K_0}u&=\frac{1}{\cosh^2(a\zeta-s)}\frac{\partial^2u}{\partial\zeta^2},
\end{align*}
and we shall consider the Jacobi operator $J=\Delta_{\K_0}+\abs{A_{\K_0}}^2$ 
and the eigenvalue problem 
\begin{align*}
\left\{\begin{aligned}
Ju&=-\lambda u \\
u'(0,\cdot)&=0                        &(\text{Neumann condition on }&\partial_0\K_0)     \\		
u'(h,\cdot)&=\cosh(a h-s)\,u(h,\cdot) &(\text{Robin   condition on }&\partial_\perp\K_0)
\end{aligned}\right.
\end{align*}
Since $u_2$ must change sign, there exists $z_0\in\interval{0,h}$ such that $u_2(z_0)=0$. 
Multiplying the eigenvalue equation 
\begin{align}
\label{eqn:eigenvalue-equation}
\frac{\partial^2u_2}{\partial\zeta^2}
+\frac{a^2+a^{-2}}{\cosh^2(a\zeta-s)}u_2
=-\lambda_2u_2\cosh^2(a\zeta-s)
\end{align}
with $u_2$ and integrating from $\zeta=0$ to $\zeta=z_0$, we obtain 
\begin{align*}
\int^{z_0}_{0}-\lambda_2u_2^2\cosh^2(a\zeta-s)\,d\zeta
&=-\int^{z_0}_{0}\abs{u_2'}^2\,d\zeta
+\int^{z_0}_{0}\frac{a^2+a^{-2}}{\cosh^2(a\zeta-s)}u_2^2\,d\zeta.
\end{align*}
Since $u(z_0)=0$, we can obtain the Poincaré-type inequality 
\begin{align*}
\int^{z_0}_{0}\abs{u_2(\zeta)}^2\,d\zeta
&=\int^{z_0}_{0}\abs[\Big]{\int^{\zeta}_{z_0}u_2'(t)\,dt}^2\,d\zeta	
\leq \int^{z_0}_{0}(z_0-\zeta)\int^{z_0}_{\zeta}\abs{u_2'(t)}^2\,dt\,d\zeta
\leq \frac{z_0^2}{2}\int^{z_0}_{0}\abs{u_2'(\zeta)}^2\,d\zeta. 
\end{align*}
Hence, 
\begin{align*}
\int^{z_0}_{0}-\lambda_2u_2^2\cosh^2(a\zeta-s)\,d\zeta
&\leq\int^{z_0}_{0}\biggl(\frac{a^2+a^{-2}}{\cosh^2(a\zeta-s)}-\frac{2}{z_0^2}\biggr)u_2^2\,d\zeta.
\end{align*}
The right-hand side is negative if 
\begin{align*}
z_0 <\sqrt{\frac{2}{a^2+a^{-2}}}\approx0.5962	
\end{align*} 
and so, in this case, we conclude $\lambda_2>0$, a contradiction.

Integrating the eigenvalue equation \eqref{eqn:eigenvalue-equation} instead from $\zeta=z_0$ to $\zeta=h$ and recalling the Robin condition $u'(h)=\cosh(ah-s)u(h)$ along $\partial_\perp\K_0$  we obtain the alternative estimate
\begin{align*}
\int^{h}_{z_0}-\lambda_2u_2^2\cosh^2(a\zeta-s)\,d\zeta
&=\abs{u_2(h)}^2\cosh(ah-s)
-\int^{h}_{z_0}\abs{u_2'}^2\,d\zeta
+\int^{h}_{z_0}\frac{a^2+a^{-2}}{\cosh^2(a\zeta-s)}u_2^2\,d\zeta
\\
&\leq
\Bigl((h-z_0)\cosh(ah-s)-1\Bigr)\int^{h}_{z_0}\abs{u_2'}^2\,d\zeta
+\int^{h}_{z_0}\frac{a^2+a^{-2}}{\cosh^2(a\zeta-s)}u_2^2\,d\zeta
\\
&\leq
\biggl( a^2+a^{-2} +\frac{2}{(h-z_0)^2}\Bigl((h-z_0)\cosh(ah-s)-1\Bigr)\biggr)
\int^{h}_{z_0}u_2^2\,d\zeta
\end{align*}
provided that $(h-z_0)\cosh(ah-s)-1<0$. 
Now the right-hand side is negative if $z_0>0.4443$.

Since the intervals $[0,0.5962]$ and $[0.4443,h]$ intersect, we anyway obtain a contradiction. 
Thus, we confirm the claim $\lambda_2>0$, as desired. 
\end{proof}

Observing (as we have already done in the previous proof) that any eigenfunction ``generating'' the index
in Lemma \ref{K0IndexAndNullity} is rotationally invariant,
we have the following obvious corollary
(which in fact can conversely
be used to prove the lemma,
with the aid of
Proposition \ref{spectralCtyWrtCoeffsAndSyms}).
In the statement $\grp^{\K_0}$ denotes the subgroup of $\Ogroup(3)$ preserving $\K_0$. 
Note that $\grp^{\K_0}$ consists of rotations about the $z$-axis and reflections through planes containing the $z$-axis.
In particular $\grp^{\K_0}$ is isomorphic to $\Ogroup(2)$, and each element of $\grp^{\K_0}$ preserves either choice of unit normal of $\K_0$.

\begin{corollary}
[Fully equivariant index and nullity of $\K_0$]
\label{K0RotSymIndexAndNullity}
With notation as above
and recalling the comments immediately preceding
Proposition \ref{spectralCtyWrtCoeffsAndSyms}, there holds
\begin{equation*}
\equivind{\grp^{\K_0}}(\indexform{\K_0}_\neum)=1
\quad \text{ and } \quad
\equivnul{\grp^{\K_0}}(\indexform{\K_0}_\neum)=0.
\end{equation*}
\end{corollary}

\paragraph{Equivariant index and nullity of $\B^2$.} 
The analysis for the flat disc $\B^2$
(featured in the construction of just
one of the families)
is trivial, and the conclusions are as follows;
in the statement we write
$\indexform{\B^2}_N$
for the index form of $\B^2$
as a minimal surface with boundary
in $(\R^3,\geuc)$
subject to the Neumann boundary condition, namely
\begin{equation*}
\indexform{\B^2}_\neum
  \vcentcolon=
  \weakbf\Bigl[
    \B^2, ~
    g^{\B^2}, ~
    0, ~
    0, ~
     \emptyset, ~
    \nbdy\B^2 \vcentcolon= \partial\B^2, ~
    \emptyset
  \Bigr].
\end{equation*}

\begin{lemma}
[($\apr_{m+1}$-equivariant) index and nullity of $\B^2$]
\label{DiscIndexAndNullity}
With  notation as above,
\begin{equation*}
\ind(\indexform{\B^2}_N)=0
\quad \text{ and } \quad
\nul(\indexform{\B^2}_N)=1.
\end{equation*}
Moreover, for each integer $m \geq 0$ the antiprismatic group $\apr_{m+1}$ preserves $\B^2$ and
\begin{equation*}
\equivind{\apr_{m+1}}(\indexform{\B^2}_N)
=\equivnul{\apr_{m+1}}(\indexform{\B^2}_N)
=0.
\end{equation*}
\end{lemma}

\begin{proof}
The first line of equalities is clear,
since the Jacobi operator on $\B^2$
is simply the standard Laplacian,
whose Neumann kernel is spanned by the constants (to rule out index one can for instance just appeal to the Hopf boundary point lemma).
The invariance of $\B^2$ under each $\apr_{m+1}$
is obvious,
and the proof is then completed
by the observation that the constants
are not $\apr_{m+1}$-equivariant
(for any $m \geq 0$).
\end{proof}

From Proposition \ref{DiscIndexAndNullity}
we immediately obtain,
analogously to Corollary \ref{K0RotSymIndexAndNullity}
from Proposition \ref{K0IndexAndNullity},
the following corollary.
In the statement $\Ogroup(2)$ refers
to the group of intrinsic isometries of $\B^2$
(extended to isometries of $\R^2$),
rather than to some subgroup of $\Ogroup(3)$,
and we write
$1$ and $\det$
for respectively the trivial
and determinant homomorphisms
$\Ogroup(2) \to \Ogroup(1)$.
The $(\Ogroup(2),1)$-invariant
functions on $\B^2$ are thus
the radial functions,
while the space of $(\Ogroup(2),\det)$-invariant
functions is trivial.

\begin{corollary}
[Indices and nullities of $\B^2$ under $\Ogroup(2)$ actions]
\label{DiscRotSymIndexAndNullity}
With notation as above we have
\begin{align*}
\symind{\Ogroup(2)}{1}(\indexform{\B^2}_\neum)&=0,
&
\symnul{\Ogroup(2)}{1}(\indexform{\B^2}_\neum)&=1,
&
\symind{\Ogroup(2)}{\det}(\indexform{\B^2}_\neum)
=\symnul{\Ogroup(2)}{\det}(\indexform{\B^2}_\neum)&=0.
\end{align*}
\end{corollary}

\paragraph{Equivariant index and nullity of $\two{\tow}$ and $\three{\tow}$.}

We recall how, away from the equator $\Sp^1$,
the surfaces $\zerogen_n$ and $\threebc_m$ are constructed
as graphs over (subsets of) $-\K_0 \cup \K_0$ and $-\K_0 \cup \B^2 \cup \K_0$.
In the vicinity of $\Sp^1$ the surfaces are instead modeled on certain singly periodic minimal surfaces
that belong to a family discovered by Karcher \cite{KarcherScherk}
and generalize the classical singly periodic minimal surfaces of Scherk \cite{Scherk}.
We now summarize the key properties of such models, to the extent needed later. 

\begin{proposition}[Desingularizing models]
\label{towExistenceAndBasicProps}
There exist in $\R^3$
complete, connected, properly embedded
minimal surfaces
$\two{\tow}$ and $\three{\tow}$
having the following properties,
which uniquely determine the surfaces
up to congruence:
\begin{enumerate}[label={\normalfont(\roman*)}]
\item
$\two{\tow}$ and $\three{\tow}$
are periodic in the $y$ direction
with period $2\pi$
and the corresponding quotient surfaces
have genus zero. 
\item
$\two{\tow}$ and $\three{\tow}$
are invariant under
$\refl_{\{x=0\}}$,
$\refl_{\{y=\pi/2\}}$,
and $\refl_{\{y=-\pi/2\}}$. 
\item
$\two{\tow}$ is invariant under
$\refl_{\{z=0\}}$
and $\three{\tow}$ under 
$\refl_{\{y=z=0\}}$. 
\item
$\two{\tow}$ has four ends
and $\three{\tow}$ has six ends,
all asymptotically planar. 
\item
Each of $\two{\tow}$ 
and $\three{\tow}$
has an end contained in
$\{x \leq 0\} \cap \{z \geq 0\}$
whose asymptotic plane intersects $\{z=0\}$
at the same angle $\omega_0>0$ 
at which $\K_0$ intersects $\B^2$,
and $\three{\tow}$ has additionally
$\{z=0\}$ as an asymptotic plane. 
\item\label{towExistenceAndBasicProps-fb}
$\two{\towdom}\vcentcolon=\two{\tow}\cap\{x \leq 0\} \cap \{\abs{y} \leq \pi/2\}$ and 
$\three{\towdom}\vcentcolon=\three{\tow}\cap\{x \leq 0\} \cap \{\abs{y} \leq \pi/2\}$ 
are connected free boundary minimal surfaces
in the half slab
$\{x \leq 0\} \cap \{\abs{y} \leq \pi/2\}$,
with
$\two{\towdom}$ invariant under $\refl_{\{z=0\}}$
and $\three{\towdom}$ invariant under $\refl_{\{y=z=0\}}$ 
(cf. Figure \ref{fig:Mfb}).
\item
Each of $\two{\towdom} \setminus \{z=0\}$ and $\three{\towdom} \setminus \{y=z=0\}$ has exactly two connected components. 
\item
$\two{\tow}$ has no umbilics,
while the set of umbilic points of $\three{\tow}$
is $\{(0,n\pi,0)\st n \in \Z\}$. 
\item
The Gauss map $\two{\nu}$ of $\two{\tow}$
restricted to the closure of either component
of $\two{\towdom} \setminus \{z=0\}$
is a bijection onto a solid spherical triangle with all sides geodesic segments of length $\pi/2$
(in other words: a quarter hemisphere),
less a point in the interior of one side. 
\item
The Gauss map $\three{\nu}$ of $\three{\tow}$
restricted to the closure of either component
of $\three{\towdom} \setminus \{y=z=0\}$
is a bijection onto
a spherical lune of dihedral angle $\pi/2$
(in other words: a half hemisphere),
less one vertex and a point in the interior of one side.
\end{enumerate}
\end{proposition}

We refer the reader to Section 3 and Appendix A of \cite{CSWnonuniqueness} for further details and a fine analysis of the properties of both surfaces in question. 
The free boundary minimal surfaces $\two{\towdom}$ and $\three{\towdom}$ are visualized in Figure \ref{fig:Mfb}. 

\begin{figure}%
\pgfmathsetmacro{\thetaO}{45}
\pgfmathsetmacro{\phiO}{135}
\pgfmathsetmacro{\angle}{25.38}
\pgfmathsetmacro{\globalscale}{0.75}
\tdplotsetmaincoords{\thetaO}{\phiO}
\pgfmathsetmacro{\offset}{0.64}
\tdplottransformmainscreen{-sin(\angle)}{cos(\angle)}{0}
\pgfmathsetmacro{\VecHx}{\tdplotresx}
\pgfmathsetmacro{\VecHy}{\tdplotresy}
\tdplottransformmainscreen{-cos(\angle)}{-sin(\angle)}{0}
\pgfmathsetmacro{\VecVx}{\tdplotresx}
\pgfmathsetmacro{\VecVy}{\tdplotresy}
\tdplottransformmainscreen{0}{0}{1}
\pgfmathsetmacro{\VecZx}{\tdplotresx}
\pgfmathsetmacro{\VecZy}{\tdplotresy} 
\begin{tikzpicture}[tdplot_main_coords,line cap=round,line join=round,scale=\globalscale]
\draw(0,0,0)node[scale=\globalscale]{\includegraphics[page=1]{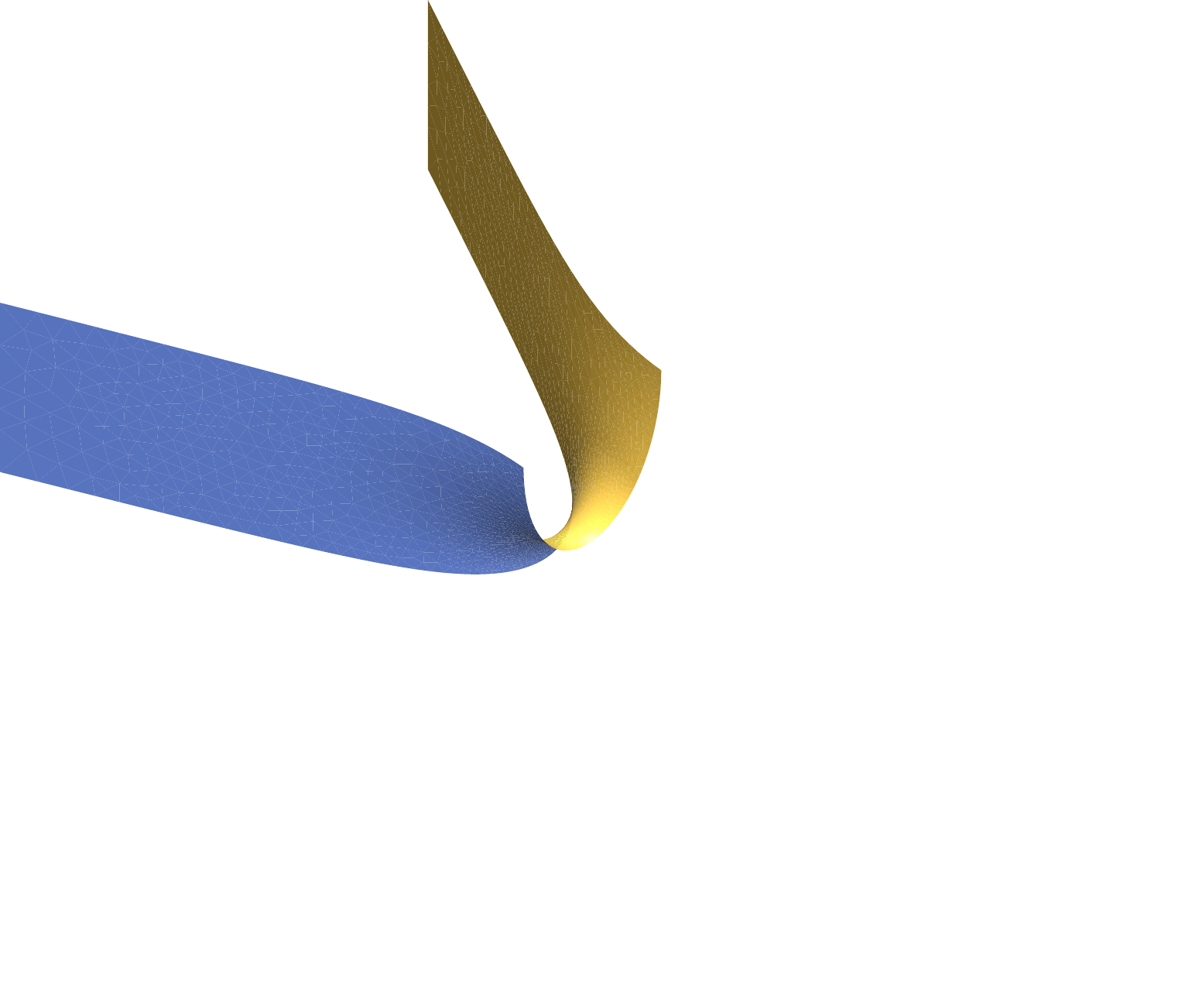}};
\pgfresetboundingbox
\draw[->,semithick,semithick](0,-0,0)--(0,2,0)node[below right,inner sep=1pt]{$x$};
\draw[->,semithick,semithick](0,0,-pi/2)--(0,0,5*pi/2+0.5)node[above,inner sep=1pt]{$y$};
\draw[->,semithick,semithick](-pi,0,0)--(pi+1,0,0)node[left,inner sep=1pt]{$z$};
\pgfmathsetmacro{\radius}{7.75}
\tdplotdrawarc[thick]{(\offset,0,pi/2)}{4}{-90}{-90+\angle}{cm={\VecHx ,\VecHy ,\VecVx ,\VecVy ,(0,0)},anchor=165}{$\omega_0$}
\draw[densely dashed]
(-0,0,     pi/2)--(\offset,0,pi/2) 
(\offset,0,pi/2)--({\radius*sin(\angle)+\offset},{-\radius*cos(\angle)},pi/2)coordinate(P1)
(\offset,0,pi/2)--(\offset,-4,pi/2)
(-0,0,     -pi/2)--(\offset,0,-pi/2) 
(\offset,0,-pi/2)--({\radius*sin(\angle)+\offset},{-\radius*cos(\angle)},-pi/2)coordinate(P2);
\draw[stealth-stealth,thick](P1)--(P2)node[midway,left,inner sep=1.5pt,cm={\VecHx ,\VecHy ,\VecZx ,\VecZy ,(0,0)}]{$\pi$}; 
\end{tikzpicture}
\hfill
\pgfmathsetmacro{\offset}{(tan(\angle)*(cos(\angle)-1/2)*ln((2/cos(\angle))-1)-2*sin(\angle)*ln((1/cos(\angle))-1))}
\begin{tikzpicture}[tdplot_main_coords,line cap=round,line join=round,scale=\globalscale]
\draw(0,0,pi/2)node[scale=\globalscale]{\includegraphics[page=2]{figures-fmbs.pdf}};
\pgfresetboundingbox
\draw[->,semithick](0,0,-pi/2)--(0,0,5*pi/2+0.5)node[above,inner sep=1pt]{$y$};
\draw[->,semithick](-pi,0,0)--(pi+1,0,0)node[left,inner sep=1pt]{$z$};
\draw[->,semithick](0,-8.66,0)--(0,2,0)node[below right,inner sep=1pt]{$x$}; 
\tdplottransformmainscreen{-cos(\angle)}{-sin(\angle)}{0}
\pgfmathsetmacro{\VecVx}{\tdplotresx}
\pgfmathsetmacro{\VecVy}{\tdplotresy}
\pgfmathsetmacro{\radius}{7}
\tdplotdrawarc[thick]{(\offset,0,pi/2)}{4}{-90}{-90+\angle}{cm={\VecHx ,\VecHy ,\VecVx ,\VecVy ,(0,0)},anchor=165}{$\omega_0$}
\draw[densely dashed](-0,0,pi/2)--(\offset,0,pi/2) 
(\offset,0,pi/2)--({\radius*sin(\angle)+\offset},{-\radius*cos(\angle)},pi/2)coordinate(P1)
(\offset,0,pi/2)--(\offset,-4,pi/2)
(-0,0,-pi/2)--(\offset,0,-pi/2) 
(\offset,0,-pi/2)--({\radius*sin(\angle)+\offset},{-\radius*cos(\angle)},-pi/2)coordinate(P2);
,semithick,semithick\draw[stealth-stealth,thick](P1)--(P2)node[midway,left,inner sep=1.5pt,cm={\VecHx ,\VecHy ,\VecZx ,\VecZy ,(0,0)}]{$\pi$}; 
\end{tikzpicture}
\caption{The minimal surfaces $\two{\towdom}$ (left) and $\three{\towdom}$ (right) as defined in Proposition \ref{towExistenceAndBasicProps} \ref{towExistenceAndBasicProps-fb}.}%
\label{fig:Mfb}%
\end{figure}

Next, we want to examine
the index and nullity
of $\two{\towdom}$
and $\three{\towdom}$
as free boundary minimal surfaces
in the half slab
$\{x \leq 0\} \cap \{\abs{y} \leq \pi/2\}$.
Because the boundary of such a domain is piecewise planar,
the corresponding Robin condition
associated with the index forms
of these surfaces is in fact homogeneous (Neumann). 

Let us prove an ancillary result.
We will observe
(in the proof of Lemma \ref{towIndexAndNullity},
to follow shortly)
that by virtue of
the behavior of the Gauss maps
described in Proposition \ref{towExistenceAndBasicProps}
the analysis of the index and nullity
of $\two{\towdom}$ and $\three{\towdom}$
reduces to the following
index and nullity computations
for boundary value problems on suitable Lipschitz domains of $\Sp^2$.

\begin{lemma}
[Index and nullity of $\Delta_{\gsph}+2$ on images of Gauss maps of $\two{\towdom}$ and $\three{\towdom}$]
\label{lowSpectraInSphere}
Set
\begin{align*}
\two{\lipdom_{\Sp^2}}
&\vcentcolon=\Sp^2\cap\{x>0\}\cap\{y>0\}\cap\{z>0\},
\\
\three{\lipdom_{\Sp^2}}
&\vcentcolon=\Sp^2\cap\{x>0\}\cap\{y>0\}.
\end{align*}
Then we have the following indices and nullities,
where 
the final row holds for any
$\zeta \in \interval{-1,1}$
and, throughout,
$\weakbf$ 
is the bilinear form \eqref{bilinear_form_def}
with $\lipdom$ as indicated,
$g=\gsph$ the round metric,
$\potential=2$
(so associated to the Schr\"{o}dinger operator
$\Delta_{\gsph}+2$),
$
 \rbdy\lipdom
 =
 \emptyset
$,
$\dbdy\lipdom$
as indicated,
and
$
 \nbdy\lipdom
 =
 \partial\lipdom
   \setminus
   \closure{\dbdy\lipdom}
$:
\begin{equation}\label{fig:ChartSphere}
\begin{array}{c|c||c|c|}
\lipdom & \dbdy\lipdom & \ind(\weakbf) & \nul(\weakbf) \\
\hline\hline
\multirow{2}{*}{$\two{\lipdom_{\Sp^2}}$} 
& \emptyset & 1 & 0 \\
\cline{2-4}
& \{z=0\} & 0 & 1 \\
\hline
\multirow{3}{*}{$\three{\lipdom_{\Sp^2}}$} 
& \emptyset & 1 & 1 \\
\cline{2-4} 
& \{x=0\} & 0 & 1 \\
\cline{2-4}
& \{x=0\} \cap \{z>\zeta\} & 1 & 0 \\
\hline
\end{array}
\end{equation}
\end{lemma}

\begin{proof}
By Lemma \ref{dom_red_ext}
we can fill in the first four rows
by identifying the index and nullity
of $\Delta_{\gsph}+2$ on the entire sphere
subject to appropriate symmetries,
the relevant spherical harmonics
being simply the restrictions
of affine functions on $\R^3$.
Lemma \ref{dom_red_ext}
is not directly applicable to the final row,
but
by the min-max characterization
\eqref{symminmax}
of eigenvalues
the $i$\textsuperscript{th}
eigenvalue for the bilinear form
specified in that row
must lie between
the $i$\textsuperscript{th}
eigenvalues of the forms
specified in the two preceding rows
($\geq$ that of the third row
and $\leq$ that of the fourth);
moreover,
the unique continuation principle
implies that both inequalities must be strict
($>$ and $<$).
The entries of the final row now follow,
concluding the proof.
\end{proof}

We shall fix components
of $\two{\towdom} \setminus \{z=0\}$
and $\three{\towdom} \setminus \{y=z=0\}$ once and for all
and write $\two{\lipdom}$ and $\three{\lipdom}$
for their respective interiors:
it follows from
Proposition \ref{towExistenceAndBasicProps}
that $\two{\nu}|_{\two{\lipdom}}$
and $\three{\nu}|_{\three{\lipdom}}$
are diffeomorphisms onto their images,
which we can and will identify
with, respectively,
the triangle $\two{\lipdom_{\Sp^2}}$
and lune $\three{\lipdom_{\Sp^2}}$
of Lemma \ref{lowSpectraInSphere}, and in particular
\begin{align*}
\{x=0\} \cap \partial \two{\lipdom_{\Sp^2}}
  &=
  \two{\nu}(\{x=0\} \cap \partial \two{\lipdom}),\\
\{y=0\} \cap \partial \two{\lipdom_{\Sp^2}} 
  &=
  \closure{\two{\nu}(\{y=\pm \pi/2\} \cap \partial \two{\lipdom})},\\
 \{z=0\} \cap \partial \two{\lipdom_{\Sp^2}}
  &=
  \two{\nu}(\{z=0\} \cap \partial \two{\lipdom}), 
\shortintertext{and} 
\{x=0\} \cap \partial \three{\lipdom_{\Sp^2}}
  &=
  \closure{
    \three{\nu}
      (
        (\{x=0\} \cup \{y=z=0\})
        \cap
        \partial \three{\lipdom}
      )
  }, \\
 \{y=0\} \cap \partial \three{\lipdom_{\Sp^2}} 
  &=
  \closure{\three{\nu}(\{y=\pm \pi/2\} \cap \partial \three{\lipdom})}.
\end{align*}
In what follows, recalling e.\,g. that the index of a minimal surface, when finite, can be computed by exhaustion (cf. \cite{Fis85}) we conveniently introduce this notation, which pertains 
certain truncations of $\two{\tow}$, $\three{\tow}$,
$\two{\towdom}$, and $\three{\towdom}$. 
\begin{figure}\centering
\pgfmathsetmacro{\angle}{25.38}
\pgfmathsetmacro{\offset}{0.64}
\pgfmathsetmacro{\radius}{6}
\pgfmathsetmacro{\zvec}{0.7}
\pgfmathsetmacro{\Rpar}{3}
\pgfmathsetmacro{\spar}{6}
\pgfmathsetmacro{\globalscale}{0.85}
\pgfmathsetmacro{\xmax}{\textwidth/2.001cm/\globalscale}
\begin{tikzpicture}[line cap=round,line join=round,scale=\globalscale]
\draw(0,\offset)++(\angle:\radius)coordinate(W1)  
     (0,\offset)++(180-\angle:\radius)coordinate(W2) 
     (0,-\offset)++(180+\angle:\radius)coordinate(W3) 
     (0,-\offset)++(-\angle:\radius)coordinate(W4)
; 
\draw[fill=black!45](W1)
..controls+(180+\angle:0.78*\radius)and+(0,\zvec)..(0.83,0)
..controls+(0,-\zvec)and+(180-\angle:0.78*\radius)..(W4)
..controls(0,-\offset)and(0,-\offset)..(W3) 
..controls+(\angle:0.78*\radius)and+(0,-\zvec)..(-0.83,0)
..controls+(0,\zvec)and+(-\angle:0.78*\radius)..(W2)
..controls(0,\offset)and(0,\offset)..cycle; 
\begin{scope}
\clip(0,0)circle(\Rpar); 
\draw[fill=black!15](W1)
..controls+(180+\angle:0.78*\radius)and+(0,\zvec)..(0.83,0)
..controls+(0,-\zvec)and+(180-\angle:0.78*\radius)..(W4)
..controls(0,-\offset)and(0,-\offset)..(W3) 
..controls+(\angle:0.78*\radius)and+(0,-\zvec)..(-0.83,0)
..controls+(0,\zvec)and+(-\angle:0.78*\radius)..(W2)
..controls(0,\offset)and(0,\offset)..cycle; 
\end{scope}
\begin{scope}[color={cmyk,1:magenta,0.5;cyan,1},semithick]
\draw[dotted]
(0,0)--++( \angle: \spar)node[below,pos=0.95,sloped]{$s$}--([turn] 90:1.33)
(0,0)--++(-\angle: \spar)node[above,pos=0.95,sloped]{$s$}--([turn]-90:1.33)
(0,0)--++( \angle:-\spar)node[above,pos=0.95,sloped]{$s$}--([turn] 90:1.33)
(0,0)--++(-\angle:-\spar)node[below,pos=0.95,sloped]{$s$}--([turn]-90:1.33)
;
\draw[-stealth](0,0)--++( \angle: 1)node[right]{$\tau^{(1)}$};
\draw[-stealth](0,0)--++(-\angle: 1)node[right]{$\tau^{(4)}$};
\draw[-stealth](0,0)--++( \angle:-1)node[left]{$\tau^{(3)}$};
\draw[-stealth](0,0)--++(-\angle:-1)node[left]{$\tau^{(2)}$};
\end{scope}
\draw[->](-\xmax,0)--(\xmax,0)node[above left]{$x$};
\draw[->](0,-4.5)--(0,4.5)node[below right]{$z$};
\draw[dashed](0,0)circle(\Rpar);
\draw[thick,latex-latex,black!50](0,0)--(0,-\Rpar)node[near end,right,inner sep=1pt]{$R_1$};
\draw[densely dotted](0,\offset)--++(\angle:\radius) 
     (0,\offset)--++(180-\angle:\radius) 
     (0,-\offset)--++(180+\angle:\radius) 
		 (0,-\offset)--++(-\angle:\radius) 
;
\draw[](W1)--++(\angle:3)   node[midway,sloped,above]{$W_1$} node[pos=-0.5,sloped,above]{$W_1(s)$}
     (W2)--++(180-\angle:3) node[midway,sloped,above]{$W_2$} node[pos=-0.5,sloped,above]{$W_2(s)$}
     (W3)--++(180+\angle:3) node[midway,sloped,below]{$W_3$} node[pos=-0.5,sloped,below]{$W_3(s)$}
		 (W4)--++(-\angle:3)    node[midway,sloped,below]{$W_4$} node[pos=-0.5,sloped,below]{$W_4(s)$}
;
\path(W4)..controls(0,-\offset)and(0,-\offset)..(W3)
node[midway,above right=1ex,fill=black!15,circle,inner sep=0pt]{$C$};
\draw(\radius*0.7,0)arc(0:\angle:\radius*0.7)node[midway,left]{$\omega_0$};
\end{tikzpicture}
\caption{A view of $\two{\tow}(s)$.} 
\label{fig:wings}%
\end{figure}%
To do so, we first fix $R_1>0$ large enough such that $\two{\tow} \setminus \{x^2+z^2=R_1^2\}$ consists of five connected components, one component $C$ in $\{x^2+z^2<R_1^2\}$
and four components $W_1$, $W_2$, $W_3$, $W_4$ in the complement,
each of which is a graph over (a subset of)
an asymptotic half plane (see Figure \ref{fig:wings}).
For each $W_i$ let $\tau^{(i)}$
be a unit vector parallel to the 
asymptotic half plane
of $W_i$, perpendicular to the $y$-axis
(the axis of periodicity), and directed
away 
from $\partial W_i$ toward the corresponding end, namely (up to relabeling)
\begin{align*}
\tau^{(1)}&=(\cos\omega_0,~0,~\sin\omega_0)=-\tau^{(3)}, &
\tau^{(2)}&=(-\cos\omega_0,~0,~\sin\omega_0)=-\tau^{(4)},  
\end{align*}
where we recall that $\omega_0>0$ is the angle at which $\K_0$ intersects $\B^2$. 
Now, given $s>R_1$, we define the truncations
\begin{align}
\notag
W_i(s)&\vcentcolon=\mathrlap{
W_i\cap\{\tau^{(i)}\cdot (x,y,z)\leq s\},
}
\\\notag
\two{\tow}(s)&\vcentcolon=\overline{C} \cup \bigcup_{i=1}^4 W_i(s), 
&\three{\tow}(s)&\mbox{ analogously (for six ends)},
\\\label{truncatedTowdom}
\two{\tow}_{-}(s)&\vcentcolon=\two{\tow}(s) \cap \{x \leq 0\}, 
&\three{\tow}_{-}(s)&\vcentcolon=\three{\tow}(s) \cap \{x \leq 0\},
\\[1ex]\notag
\two{\towdom}(s)&\vcentcolon=\two{\tow}(s) \cap \two{\towdom},
&\three{\towdom}(s)&\vcentcolon=\three{\tow}(s) \cap \three{\towdom}.
\intertext{For each $\epsilon,\epsilon'>0$ we then set similarly 
$\three{\towdom}(\epsilon^{-1},\epsilon')\vcentcolon=\three{\towdom}(\epsilon^{-1})\cap\{ x^2 + y^2 + z^2 > \epsilon'\}$ and}
\notag
\two{\lipdom}(\epsilon)
&\vcentcolon=\two{\lipdom} \cap \two{\towdom}(\epsilon^{-1}),
&
\three{\lipdom}(\epsilon,\epsilon')
&\vcentcolon=\three{\lipdom}\cap \three{\towdom}(\epsilon^{-1},\epsilon'),
\end{align}
truncating $\two{\lipdom}$ and $\three{\lipdom}$
at (affine) distance $\epsilon^{-1}$
and excising from $\three{\lipdom}$
a disc with radius $\sqrt{\epsilon'}$ and center at the umbilic $(0,0,0)$. 
We then in turn define
$\two{\lipdom_{\Sp^2}}(\epsilon)\vcentcolon=
\two{\nu}\bigl(\two{\lipdom}(\epsilon)\bigr)\subset \two{\lipdom_{\Sp^2}}$ as well as
$\three{\lipdom_{\Sp^2}}(\epsilon,\epsilon')\vcentcolon=
\three{\nu}\bigl(\three{\lipdom}(\epsilon,\epsilon')\bigr)\subset \three{\lipdom_{\Sp^2}}$.
As a direct consequence of Lemma \ref{lowSpectraInSphere} and Proposition \ref{spectralContinuityWrtExcision} we get what follows.

\begin{corollary}\label{cor:ApproxCutSphere}
In the setting above, consider for any $\epsilon,\epsilon'>0$ the Schr\"{o}dinger operator
$\Delta_{\gsph}+2$ on the domains given, respectively, by $\two{\lipdom_{\Sp^2}}(\epsilon)$ and 
$\three{\lipdom_{\Sp^2}}(\epsilon,\epsilon')$ and 
subject to
any of the boundary conditions
specified in the table \eqref{fig:ChartSphere},
where the boundary is contained, respectively, in 
$\partial \three{\lipdom_{\Sp^2}}$
and $\partial \two{\lipdom_{\Sp^2}}$
and subject to Dirichlet conditions elsewhere.
In other words, let
$\two{\weakbf}$
be either bilinear form
corresponding to the top two rows
of \eqref{fig:ChartSphere},
let $\three{\weakbf}$
be any bilinear form
corresponding to the bottom three rows
of \eqref{fig:ChartSphere},
and consider also the bilinear forms
\begin{align*}
\two{\weakbf}_\epsilon
  &\vcentcolon=
  \dint{(\two{\weakbf})}_{\two{\lipdom_{\Sp^2}}(\epsilon)}
  =
  \weakbf
  \Bigl[
    \two{\lipdom_{\Sp^2}}(\epsilon),
    \gsph,
    2,
    0,
    \dbdy\two{\lipdom_{\Sp^2}}
      \cup
      (
        \partial \two{\lipdom_{\Sp^2}}(\epsilon)
        \setminus
        \partial \two{\lipdom_{\Sp^2}}
      ),
    \nbdy\two{\lipdom_{\Sp^2}},
    \emptyset
  \Bigr]
\\
\three{\weakbf}_{\epsilon,\epsilon'}
  &\vcentcolon=
  \dint{(\three{\weakbf})}_{\three{\lipdom_{\Sp^2}}(\epsilon,\epsilon')}
  =
  \weakbf
  \Bigl[
    \three{\lipdom_{\Sp^2}}(\epsilon,\epsilon'),
    \gsph,
    2,
    0,
    \dbdy\three{\lipdom_{\Sp^2}}
      \cup
      (
        \partial \three{\lipdom_{\Sp^2}}(\epsilon,\epsilon')
        \setminus
        \partial \three{\lipdom_{\Sp^2}}
      ),
    \nbdy\three{\lipdom_{\Sp^2}},
    \emptyset
  \Bigr]
\end{align*}
using the notation \eqref{DirAndNeumInternalizationsOfBilinearForm}.
Then there exists $\epsilon_0>0$ such that for all $0<\epsilon,\epsilon'<\epsilon_0$
\begin{equation*}
\ind(\two{\weakbf_{\epsilon}})=\ind(\two{\weakbf})
\quad \text{ and } \quad
\ind(\three{\weakbf_{\epsilon,\epsilon'}})=\ind(\three{\weakbf}).
\end{equation*}
\end{corollary}

In particular, we can derive these geometric conclusions:

\begin{corollary}[Index of $\two{\towdom}$ and $\three{\towdom}$]
We have the following even and odd indices for $\two{\towdom}$ and $\three{\towdom}$.
\[
\begin{array}{c|c||c|c|}
S & \grp & \symind{\grp}{+}(S) & \symind{\grp}{-}(S)
\\\hline\hline
\two{\towdom} & \cycgrp{\{z=0\}} & 1 & 0
\\\hline
\three{\towdom} & \cycgrp{\{y=z=0\}} & 1 & 1
\\\hline
\end{array}
\]
\end{corollary}

\begin{proof}
We will verify (as a sample) the even index asserted in the second row of the table; the other claims are checked in the same fashion. 
The Gauss map of a minimal surface in $\R^3$ is (anti)conformal away from its umbilics,
with conformal factor (one half of) the pointwise square
of the norm of its second fundamental form,
so by Proposition \ref{conformalInvariance},
for each $\epsilon,\epsilon'>0$,
the index of $\three{\lipdom}_{\Sp^2}(\epsilon,\epsilon')$
with the foregoing boundary conditions (as in Corollary \ref{cor:ApproxCutSphere}, according to the third row of the table in Lemma \ref{lowSpectraInSphere})
agrees also with the index
of $\three{\lipdom}(\epsilon,\epsilon')$
subject to the corresponding boundary conditions.
By Lemma \ref{dom_red_ext}
this last index agrees with the
$\cycgrp{\{y=z=0\}}$-even index
of $\three{\tow_{\mathrm{fb}}}(\epsilon^{-1},\epsilon')$
subject to the Dirichlet condition
along the excisions
and the Neumann condition everywhere else.
Hence, thanks to Corollary \ref{cor:ApproxCutSphere}, such a value of the index is equal to $1$ for any sufficiently small $\epsilon, \epsilon'$. 
We now conclude, first letting $\epsilon'\to 0$ and appealing to Proposition \ref{spectralContinuityWrtExcision} to control the effect
of the excision near $(0,0,0)$, and then appealing to the aforementioned characterization of the Morse index via exhaustions, 
that $\three{\towdom}$ indeed has
$\cycgrp{\{y=z=0\}}$-index $1$. 
\end{proof}

For use in the following subsection we fix a smooth cutoff function
$
 \Psi
 \colon
 \Interval{0,\infty}
 \to
 \IntervaL{0,1}
$
that is
constantly $1$
on $\{x \leq 1\}$
and
constantly $0$
on $\{x \geq 2\}$,
and we define on $\two{\tow}$ and $\three{\tow}$
the functions and metrics
\begin{equation}
\label{towerCutoffsConfFactsConfMets}
\begin{aligned}
\two{\psi}
&\vcentcolon=
  (\Psi \circ \abs{x})|_{\two{\tow}},
&\qquad
\two{\conffact}
  &\vcentcolon=
  \sqrt{
    \two{\psi}
      +\frac{1}{2}\abs[\Big]{\twoff{\two{\tow}}}^2
      (1-\two{\psi})
},
&\qquad
\two{h}
  &\vcentcolon=
  (\two{\conffact})^2\met{\two{\tow}},
\\
\three{\psi}
  &\vcentcolon=
  (\Psi \circ \abs{x})|_{\three{\tow}},
&
\three{\conffact}
  &\vcentcolon=
  \sqrt{
    \three{\psi}
      +\frac{1}{2}\abs[\Big]{\twoff{\three{\tow}}}^2
      (1-\three{\psi})
},
&
\three{h}
  &\vcentcolon=
  (\three{\conffact})^2\met{\three{\tow}}.
\end{aligned}
\end{equation}
Note that $\two{\conffact}$
is invariant under
$\refl_{\{z=0\}}$,
$\three{\conffact}$
under
$\refl_{\{y=z=0\}}$,
and both are invariant under
$\refl_{\{x=0\}}$,
$\refl_{\{y=-\pi/2\}}$,
and $\refl_{\{y=\pi/2\}}$. It is natural to associate to $\two{\towdom}$,
regarded as
a free boundary minimal surface
in the slab
$\{x \leq 0\} \cap \{\abs{y} \leq \pi/2\}$,
the stability form
$\indexform{\two{\towdom}}$,
defined at least on smooth functions of compact support
by
\begin{equation*}
\indexform{\two{\towdom}}(u,v)
\vcentcolon=
  \int_{\two{\towdom}}
    \met{\two{\tow}}\bigl(\nabla_{\met{\two{\tow}}}u, \nabla_{\met{\two{\tow}}}v\bigr)
    \, \hausint{2}{\met{\two{\tow}}}
  -\int_{\two{\towdom}}
    \abs[\Big]{\twoff{\two{\tow}}}_{\met{\two{\tow}}}^2uv
    \, \hausint{2}{\met{\two{\tow}}}.
\end{equation*}
From the identity
\begin{equation*}
\indexform{\two{\towdom}}(u,v)
=
  \int_{\two{\towdom}}
  \two{h}\bigl(\nabla_{\two{h}}u, \nabla_{\two{h}}v\bigr)
    \, \hausint{2}{\two{h}}
  -\int_{\two{\towdom}}
    \abs[\Big]{\twoff{\two{\tow}}}_{\two{h}}^2uv
    \, \hausint{2}{\two{h}}
\end{equation*}
and the manifest boundedness of
$\abs[\big]{\twoff{\two{\tow}}}_{\two{h}}^2 =(\two{\conffact})^{-2}\abs[\big]{\twoff{\two{\tow}}}_{\met{\two{\tow}}}^2$
we see that
$\indexform{\two{\towdom}}$
is in fact well-defined on
$\sob(\two{\towdom},\two{h})$.
Likewise,
the analogously defined
$\indexform{\three{\towdom}}$
is well-defined on
$\sob(\three{\towdom},\three{h})$.

We now point out that we can identify
the interiors of $\two{\towdom}$
and $\three{\towdom}$ under respectively the metrics
$\two{h}$ and $\three{h}$
as Lipschitz domains
as
in the setting of Section \ref{sec:nota}.
Concretely,
we first consider the Riemannian quotients
$\two{\widetilde{\tow}}$
and $\three{\widetilde{\tow}}$
of $(\two{\tow},\two{h})$
and $(\three{\tow},\three{h})$
under a fundamental period.
Then $\two{\widetilde{\tow}}$
is diffeomorphic to $\Sp^2$
with four points removed
and $\three{\widetilde{\tow}}$ to $\Sp^2$
with six points removed.
By virtue of \eqref{towerCutoffsConfFactsConfMets}
and the behavior of the Gauss maps
as outlined in Proposition \ref{towExistenceAndBasicProps},
we can in fact choose the last two diffeomorphisms
so that they are isometries on neighborhoods of the punctures.
In this way we obtain
smooth Riemannian compactifications.
By composing the defining projection
of each tower
onto its quotient by a fundamental period
with the corresponding embedding into the compactification
we identify (via isometric embedding)
the interior of $\two{\towdom}$ under $\two{h}$
and the interior of $\three{\towdom}$ under $\three{h}$
with Lipschitz domains
$\two{\liptowdom}$ and $\three{\liptowdom}$
in the two respective compactifications,
and we likewise identify
$\partial\two{\towdom}$ and $\partial\three{\towdom}$
with subsets of
$\partial \two{\liptowdom}$
and $\partial \three{\liptowdom}$
respectively. 
Of course, the role of the ``ambient manifold'' for such Lipschitz domains is played respectively by the Riemannian manifolds $(\Sp^2,\two{h})$ and $(\Sp^2,\three{h})$; here, with slight abuse of notation, we have tacitly extended the metrics in question across the four and six punctures respectively.

Next, recalling the definition of $T$ from \eqref{bilinear_form_def}, we define the bilinear form
\begin{equation*}
\indexform{\two{\liptowdom}}
\vcentcolon=
\weakbf
\left[
  \two{\liptowdom},~
  \two{h},~
  \potential
    =
    (\two{\conffact})^{-2}
    \abs[\Big]{\twoff{\two{\tow}}}_{\met{\two{\tow}}}^2,~
  \robinpotential=0,~
  \dbdy\two{\liptowdom}=\emptyset,~
  \nbdy\two{\liptowdom}=\partial\two{\liptowdom},~
  \rbdy\two{\liptowdom}=\emptyset
\right],
\end{equation*}
where
(as we shall do generally in the sequel
for functions defined on $\two{\tow}$ or $\three{\tow}$,
without further comment)
for the potential we tacitly interpret the right-hand side
as a function on $\two{\liptowdom}$; we define
$\indexform{\three{\liptowdom}}$ in analogous fashion.
We then have (cf. Section \ref{subsec:conformal}) the equalities
\begin{equation}\label{eq:ConformalInvJacobi}
\indexform{\two{\liptowdom}}
  =
  \indexform{\two{\towdom}}
  \text{ on } \sob(\two{\towdom},\two{h})
\quad \text{ and } \quad
\indexform{\three{\liptowdom}}
  =
  \indexform{\three{\towdom}}
  \text{ on } \sob(\three{\towdom},\three{h}).
\end{equation}

\begin{lemma}[Index and nullity of $\indexform{\two{\liptowdom}}$ and $\indexform{\three{\liptowdom}}$]
\label{towIndexAndNullity}
With definitions as in the preceding paragraph we have the following indices and nullities.
\[
\begin{array}{c|c||c|c||c|c|}
S & \grp
  & \symind{\grp}{+}(\indexform{S})
  & \symnul{\grp}{+}(\indexform{S})
  & \symind{\grp}{-}(\indexform{S})
  & \symnul{\grp}{-}(\indexform{S}) \\
\hline\hline
\two{\liptowdom} & \cycgrp{\{z=0\}} & 1 & 0 & 0 & 1 \\
\hline
\three{\liptowdom} & \cycgrp{\{y=z=0\}} & 1 & 1 & 1 & 0 \\
\hline
\end{array}
\]
\end{lemma}

\begin{proof}
The first row follows from a direct application of Proposition \ref{conformalInvariance}
in conjunction with the first two rows of the table in Lemma \ref{lowSpectraInSphere}.
Indeed, in this case there are no umbilic points in play (for, recall, $\two{\tow}$ has no umbilic points) and the Gauss map
furnishes an (anti)conformal map from the compactified quotient onto $\Sp^2$. 
For $\three{\tow}$, however,
the corresponding conformal factor degenerates
at the umbilic at $(0,0,0)$, as all of its translates.
Nevertheless,
aided by Lemma \ref{dom_red_ext}
and Corollary \ref{cor:MonotInd}
we can verify the indices in the second row
in much the same fashion,
applying Proposition \ref{conformalInvariance}
on suitable subdomains (obtained by removing smaller and smaller neighborhoods of the origin).

For the nullities, however, we
employ an \emph{ad hoc} argument, since one cannot expect an analogue of the aforementioned Corollary \ref{cor:MonotInd} to hold true in general. That said, we observe first that
the translations in the $z$ direction
induce a non-trivial, smooth, bounded,
$(\cycgrp{\{y=z=0\}},+)$-invariant
(scalar-valued)
Jacobi field on $\three{\tow}$ 
which readily implies it to define an element of $\sob(\three{\towdom},\three{h})$. 
This shows,
in view of \eqref{eq:ConformalInvJacobi}, 
that the nullities in question
are at least the values indicated in the table. On the other hand,
(appealing to Lemma \ref{dom_red_ext}
for the regularity)
each element, say $u\colon \three{\liptowdom}\to\R$, of the eigenspace
with eigenvalue zero
corresponding to the nullities in question is smooth and bounded. 
If we restrict it to $\three{\lipdom} \subset\three{\liptowdom}$ and consider the precomposition with the inverse of the Gauss map (which, let us recall, yields an (anti)conformal diffeomorphism $\nu^{\three{\tow}}\colon\three{\lipdom}\to \three{\lipdom_{\Sp^2}}$), then the resulting function $u_0\vcentcolon = u\circ (\nu^{\three{\tow}})^{-1}$ satisfies $(\Delta_{\gsph}+2)u_0=0$ and so we get an element contributing to $\nul(T)$ where $T$ is as encoded in the third (respectively: the fifth) row of the table \eqref{fig:ChartSphere} when starting from the $(\cycgrp{\{y=z=0\}},+)$-invariant (respectively: the $(\cycgrp{\{y=z=0\}},-)$-invariant) problem on $\three{\liptowdom}$. 
It is clear that one thereby gets injective maps of vector spaces, and so from  Lemma \ref{lowSpectraInSphere}
\[
\symnul{\grp}{+}(\indexform{\three{\liptowdom} })\leq 1, \quad \symnul{\grp}{-}(\indexform{\three{\liptowdom}})\leq 0
\]
which in particular implies that such maps are, a posteriori, linear isomorphisms, and thus completes the proof.
\end{proof}

When we wish to consider
the sets $\two{\towdom}(s)$ and $\three{\towdom}(s)$
endowed respectively with
the metrics $\two{h}$ and $\three{h}$,
we shall denote them by
$\two{\liptowdom}(s)$ and $\three{\liptowdom}(s)$.
Recalling the notation of Subsection \ref{subsec:subdomains},
we further define
\begin{equation}
\label{dirNeumInternalizationsOnTowdom}
\indexform{\two{\liptowdom}(s)}_\dir
  \vcentcolon=
  \dint{
    \Bigl(
      \indexform{\two{\liptowdom}}
    \Bigr)
  }_{\two{\liptowdom}(s)}
\quad \text{ and } \quad
\indexform{\two{\liptowdom}(s)}_\neum
  \vcentcolon=
  \nint{
    \Bigl(
      \indexform{\two{\liptowdom}}
    \Bigr)
  }_{\two{\liptowdom}(s)}.
\end{equation}
In short, we are adjoining respectively Dirichlet or Neumann boundary conditions along the cuts.

\begin{lemma}[Spectra of $\indexform{\two{\liptowdom}(s)}$ and $\indexform{\three{\liptowdom}(s)}$]
\label{spectraOnTruncatedTowers}
For each integer $i \geq 1$
\begin{alignat*}{2}
\lim_{s \to \infty}
 \eigenvalsym
  {\indexform{\two{\liptowdom}(s)}_\dir}{i}
  {\cycgrp{\{z=0\}}}{\pm}
&=
\lim_{s \to \infty}
\eigenvalsym
  {\indexform{\two{\liptowdom}(s)}_\neum}{i}
  {\cycgrp{\{z=0\}}}{\pm}
&&=
\eigenvalsym
  {\indexform{\two{\liptowdom}}}{i}
  {\cycgrp{\{z=0\}}}{\pm},
\\[1ex]
\lim_{s \to \infty}
 \eigenvalsym
  {\indexform{\three{\liptowdom}(s)}_\dir}{i}
  {\cycgrp{\{y=z=0\}}}{\pm}
&=
\lim_{s \to \infty}
\eigenvalsym
  {\indexform{\three{\liptowdom}(s)}_\neum}{i}
  {\cycgrp{\{y=z=0\}}}{\pm}
&&=
\eigenvalsym
  {\indexform{\three{\liptowdom}}}{i}
  {\cycgrp{\{y=z=0\}}}{\pm},
\end{alignat*}
for any consistent choice of $+$ or $-$
on both sides of each equality.
\end{lemma}

\begin{proof}
We will write down the proof of the two equalities in the first line for the $+$ choice,
as the remaining cases can be proved
in the same way.
First note that
Proposition \ref{spectralContinuityWrtExcision}
gives us
\begin{equation*}
\lim_{s \to \infty}
 \eigenvalsym
  {\indexform{\two{\liptowdom}(s)}_\dir}{i}
  {\cycgrp{\{z=0\}}}{+}
=
\eigenvalsym
  {\indexform{\two{\liptowdom}}}{i}
  {\cycgrp{\{z=0\}}}{+}.
\end{equation*}
Using the min-max characterization
\eqref{symminmax} of eigenvalues
we then also get
\begin{equation*}
\limsup_{s \to \infty}
 \eigenvalsym
  {\indexform{\two{\liptowdom}(s)}_\neum}{i}
  {\cycgrp{\{z=0\}}}{+}
\leq
\limsup_{s \to \infty}
 \eigenvalsym
  {\indexform{\two{\liptowdom}(s)}_\dir}{i}
  {\cycgrp{\{z=0\}}}{+}
=
\eigenvalsym
  {\indexform{\two{\liptowdom}}}{i}
  {\cycgrp{\{z=0\}}}{+}.
\end{equation*}
The key step now toward the goal of establishing
\begin{equation*}
\liminf_{s \to \infty}
 \eigenvalsym
  {\indexform{\two{\liptowdom}(s)}_\neum}{i}
  {\cycgrp{\{z=0\}}}{+}
\geq
\eigenvalsym
  {\indexform{\two{\liptowdom}}}{i}
  {\cycgrp{\{z=0\}}}{+}
\end{equation*}
(which completes the proof)
is to construct a family
of (appropriately symmetric)
linear extension operators
$
 E_s
 \colon
 \sob(\two{\liptowdom}(s))
 \to 
 \sob(\two{\liptowdom})
$
uniformly bounded in $s$,
assuming $s \geq s_0$ for some universal $s_0>0$. 
With these extensions in hand it is straightforward,
for example, to adapt the argument for
\eqref{spectralLowerSemicty}
in the proof of Proposition \ref{spectralCtyWrtCoeffsAndSyms}.

We now outline the construction of the $E_s$ extension operators.
By the imposed symmetry
(in the case under discussion
even reflection through $\{z=0\}$)
and by taking $s$ large enough,
it suffices to specify the extension
on a single end $W$,
a graph over a subset of the corresponding
asymptotic plane $\Pi$ (with $\tau$ the corresponding defining vector, recalling the notation
preceding \eqref{truncatedTowdom}).
Let $\varpi \colon W \to \Pi$
be the associated projection.
By partitioning the given function
using appropriately chosen
smooth cutoff functions
(fixed independently of $s$),
it in fact suffices
to consider the extension
problem for a function
$v \in \sob(W \cap \two{\towdom}(s),\two{h})$
such that the support of $\varpi^*v$
is compactly contained in the rectangle
(expressed in the notation
of \eqref{truncatedTowdom})
\begin{equation*}
\{
  0
  <
  \tau\cdot (x,y,z)
  \leq
  s
\}
\cap
\{-\pi \leq 2y \leq \pi\}.
\end{equation*}
We can extend $\varpi^*v$
via even reflection through
the $s$ side of the above rectangle,
thereby obtaining an extension of
$v$ to an element of
$\sob(W,\two{h})$.
The asymptotic convergence
of $W$ to $\Pi$,
the monotonic decay
of $\two{\conffact}$ along
$W$ toward $\infty$,
and the conformal invariance
(in the current two-dimensional setting)
of the Dirichlet energy
ensure that this extension
has the desired properties.
\end{proof}

\subsection{Deconstruction of the surfaces and regionwise geometric convergence}
\label{subs:deconstruction}

We first take a moment to briefly review the constructions of the surfaces from \cite{CSWnonuniqueness}.
First 
(cf. \cite[Section 3]{CSWnonuniqueness}),
an approximate minimal surface in $\B^3$,
called the initial surface,
whose boundary is contained in $\partial\B^3$
and which meets $\partial\B^3$ exactly orthogonally,
is fashioned by hand,
via suitable interpolations, from the models
($\K_0$, $\two{\tow}$ or $\three{\tow}$,
and for $\threebc_m$ also $\B^2$).
Second
(cf. \cite[Section~5]{CSWnonuniqueness}),
the final exact solution
is identified as the normal graph of a small function
over the approximate solution. For what pertains this second step
we wish only to highlight that
the assignment of graph to function
is made using not
the usual Euclidean metric $\geuc$
but instead an $\Ogroup(3)$-invariant metric
(fixed once and for all,
independently of the data $n$ or $m$)
conformally Euclidean,
and
called the auxiliary metric.
On a neighborhood of the origin
this metric agrees
exactly with the Euclidean one,
while on a neighborhood
of $\partial\B^3=\Sp^2$
it agrees exactly
with the cylindrical metric
on $\Sp^2 \times \R$;
this last property
and the orthogonality of the intersection
of the initial surface with $\partial\B^3$
ensure that the boundary of the resulting graph
is also in $\partial\B^3$.
We will write
$\initzerogen_n$ and $\initthreebc_m$
for the initial surfaces
and
$
 \two{\varpi}_n
 \colon
 \zerogen_n \to \initzerogen_n
$
and
$
 \three{\varpi}_m
 \colon
 \threebc_m \to \initthreebc_m
$
for the nearest-point projections
under the above auxiliary metric.

Turning to the first step, actually
(because of the presence of a cokernel)
one constructs for each given $n$ or $m$
not just a single initial surface
but a (continuous) one-parameter family of them.
In the construction this parameter is treated
as an unknown
and is determined only in the second step,
simultaneously with the defining function
for the final surface.
Here, however, we can take the construction for granted
and accordingly speak of a single initial surface,
whose defining parameter value
is some definite (though not explicit) function
of $n$ or $m$ as appropriate. Nevertheless we must explain that this parameter
enters the construction
at the level of the building blocks,
except for $\B^2$, which is unaffected,
as follows.
First, the catenoidal annulus $\K_0$
is just one in a family $\K_\epsilon$
(cf. the beginning of Subsection 3.1
in \cite{CSWnonuniqueness})
of such annuli, all rotationally symmetric about the $z$-axis,
depending smoothly on $\epsilon$.
The details are not critical here,
but each $\K_\epsilon$ is the intersection
with $\B^3$ of a complete catenoid
with axis the $z$-axis,
and $\K_\epsilon$ meets $\Sp^2$
at two circles of latitude,
the upper one a circle of orthogonal intersection
and the lower one the circle at height $z=\epsilon$.
Similarly, from $\two{\tow}$ and $\three{\tow}$
we define, by explicit graphical deformation,
families which here we will call
$\two{\tow}_\delta$ and $\three{\tow}_\delta$
(cf. the beginning of Subsection 3.2
of \cite{CSWnonuniqueness}).
These deformations are the identity on
the ``cores''
of $\two{\tow}$ and $\three{\tow}$ 
and smoothly transition to translations on the ends,
in the $z$-direction,
up or down depending on the end,
and
through a displacement
determined by $\delta$.
Importantly, all the $\two{\tow}_\delta$
and $\three{\tow}_\delta$
have the same symmetries as $\two{\tow}$
and $\three{\tow}$ respectively. Now the datum $n$ determines building blocks
$\two{\tow}_{\two{\delta}(n)}$ and
$\K_{\two{\epsilon}(n)}$,
while the datum $m$ determines building blocks
$\three{\tow}_{\three{\delta}(m)}$,
$\K_{\three{\epsilon}(m)}$,
and $\B^2$.

We next define maps 
$\two{\Phi}_n$ and $\three{\Phi}_m$
(\cite[(3.37)]{CSWnonuniqueness})
from neighborhoods
of
$
 \frac{1}{n}\two{\tow}_{\two{\delta}(n)}
 \cap
 \{x \leq 0\}
$
and
$
 \frac{1}{m+1}\three{\tow}_{\three{\delta}(m)}
 \cap
 \{x \leq 0\}
$
respectively into $\B^3$,
so as  to ``wrap'' the cores of these surfaces
around the equator $\Sp^1$ approximately isometrically
but to take their asymptotic half planes
(in $\{x \leq 0\}$) onto
$\pm \K_{\two{\epsilon}(n)}$
in the first case
and onto $\pm \K_{\three{\epsilon}(m)}$
and $\B^2$ in the second. Thus, just referring to the family $\zerogen_n$ for the sake of brevity, we  truncate the surface $\two{\tow}_{\two{\delta}(n)}$
by intersecting with $\{x \geq -n^{3/4}\}$, and then apply $\two{\Phi}_n$ to the scaled-down by a factor $1/n$ truncated surface. The image is embedded (for $n$ large enough)
and contained in the ball,
in fact contained in a tubular neighborhood
of $\Sp^1$ with radius of order $n^{-1/4}$.

Near the two truncation boundary components
the surface is a small graph over either
$\pm\K_{\two{\epsilon}(n)}$.
We smoothly cut off the defining function
in a $\frac{1}{n}$-neighborhood of the boundary
to make the surface exactly catenoidal there
and then extend using these annuli
on the other side of the truncation boundary
all the way to $\partial\B^3$.
The result is our initial surface
$\initzerogen_n$.
The initial surface $\initthreebc_m$
is constructed analogously,
now also smoothly transitioning from the middle
truncation boundary to coincide with $\B^2$
on neighborhood of the origin. In what follows we will distill those objects and ancillary results that are needed for the spectral convergence theorems we will prove in Section \ref{subs:RegionSpectConv}.

\paragraph{Decompositions.}
Recalling \eqref{truncatedTowdom}
for the definition of the below domains,
our construction in \cite{CSWnonuniqueness} provides, in particular, smooth maps
\begin{align*}
\varphi^{\two{\towr_n}}
  \colon
  \two{\tow}_{-}(n^{5/8})
&\to\zerogen_n,
&
\varphi^{\three{\towr_m}}
  \colon
  \three{\tow}_{-}((m+1)^{5/8})
&\to\threebc_m,
\end{align*}
which are smooth coverings of their images. 
For all $0<s\leq\sqrt{n}$ or, respectively, $0<s\leq \sqrt{m+1}$ we in turn define
\begin{align*}
\two{\towr_n}(s)
&\vcentcolon=
  \varphi^{\two{\towr_n}}\bigl(\two{\tow}_{-}(s)\bigr)
  \subset
  \zerogen_n,
&
\three{\towr_m}(s)
&\vcentcolon=
  \varphi^{\three{\towr_m}}\bigl(\three{\tow}_{-}(s)\bigr)
  \subset
  \threebc_m.
\end{align*}
In practice,
in addition to the upper bound required on $s$,
we will be interested only in $s$ greater
than a universal constant set by
$\two{\tow}$ and $\three{\tow}$:
we want to truncate far enough out
(in the domain)
that near and beyond the truncation boundary
the surface is already the graph
of a small function over the asymptotic planes.
In a typical application to follow
we will take $s$ large in absolute terms
and then take $n$ or $m$ large with respect to $s$,
so we will not always repeat either restriction.
When they do hold,
$\zerogen_n \setminus \two{\towr}_n(s)$
consists of two connected components
and
$\threebc_m \setminus \three{\towr}_m(s)$
consists of three,
and we define
\begin{equation*}
\begin{aligned}
\two{\catr}_n(s)
  &\vcentcolon=
  \text{the closure of the component of }
  \zerogen_n \setminus \two{\towr}_n(s)
  \text{ on which $z$ is maximized,}
\\
\three{\catr}_m(s)
  &\vcentcolon=
  \text{the closure of the component of }
  \threebc_m \setminus \three{\towr}_m(s)
  \text{ on which $z$ is maximized,}
\\
\three{\discr}_m(s)
  &\vcentcolon=
  \text{the closure of the component of }
  \threebc_m \setminus \three{\towr}_m(s)
  \text{ that contains the origin.}
\end{aligned}
\end{equation*}
Observe that each $\two{\towr_n}(s)$ is invariant under $\refl_{\{z=0\}}$,
that the interiors of
$\two{\towr_n}(s)$,
$\two{\catr_n}(s)$,
and $\refl_{\{z=0\}}\two{\catr_n}(s)$
are pairwise disjoint,
and that
the last three regions
cover $\zerogen_n$. In particular, considering the interior of such sets, one thereby determines a candidate partition for the application of Proposition \ref{mr}.
Similarly,
$\three{\towr_m}(s)$
and $\three{\discr_m}(s)$
are invariant under $\refl_{\{y=z=0\}}$;
the interiors of
$\three{\towr_m}(s)$,
$\three{\discr_m}(s)$,
$\three{\catr_m}(s)$,
and
$\refl_{\{y=z=0\}}\three{\catr_m}(s)$
are pairwise disjoint,
also such four surfaces
cover $\threebc_m$. 

We agree to distinguish the choices $s=\sqrt{n}$
and $s=\sqrt{m+1}$ by omission of the parameter value:
\begin{align*}
\two{\towr_n}&\vcentcolon=\two{\towr_n}(\sqrt{n}),
&
\two{\catr}_n&\vcentcolon=\two{\catr}_n(\sqrt{n}),
\\[.5ex]
\three{\towr_m}&\vcentcolon=\three{\towr_m}(\sqrt{m+1}),
&
\three{\catr}_m&\vcentcolon=\three{\catr}_m(\sqrt{m+1}),
&
\three{\discr}_m&\vcentcolon=\three{\discr}_m(\sqrt{m+1}),
\end{align*}
as visualized in Figure \ref{fig:decomposition}. 
We also define the dilated truncations
(cf. Figure \ref{fig:rescaled-pieces})
\begin{align*}
\two{\towrdom{n}}(s)
  &\vcentcolon=
  n\Bigl(
    \two{\towr}_n (s)
    \cap
    \Wedge{-\pi/(2n)}{\pi/(2n)}
  \Bigr)
  =
  n\varphi^{\two{\towr_n}}(\two{\towdom}(s)),
&
\two{\towrdom{n}}
  &\vcentcolon=
  \two{\towrdom{n}}(\sqrt{n}),
\\
\three{\towrdom{m}}(s)
  &\vcentcolon=
(m+1)\Bigl(
  \three{\towr}_m(s)
  \cap
  \Wedge{-\pi/(2(m+1))}{\pi/(2(m+1))}
\Bigr)
=
(m+1)\varphi^{\three{\towr_m}}(\three{\towdom}(s)),
&
\three{\towrdom{m}}
  &\vcentcolon=
  \three{\towrdom{m}}(\sqrt{m+1}),
\end{align*}
where the notation for wedges has been given in \eqref{wedge},
and finally introduce the transition regions
\begin{align*}
\two{\transr_n}(s)
&\vcentcolon=
  \closure{
    \two{\towrdom{n}}
    \setminus
    \two{\towrdom{n}}(s)
  },
&
\three{\transr_m}(s)
&\vcentcolon=
  \closure{
    \three{\towrdom{m}}
    \setminus
    \three{\towrdom{m}}(s)
  }. 
\end{align*}

\begin{figure}%
\pgfmathsetmacro{\thetaO}{72}
\pgfmathsetmacro{\phiO}{90-180/20} 
\tdplotsetmaincoords{\thetaO}{\phiO}
\begin{tikzpicture}[tdplot_main_coords,line cap=round,line join=round,scale=\unitscale,baseline={(0,0,0)}]
\draw(0,0,0)node{\includegraphics[width=\fbmsscale\textwidth,page=4]{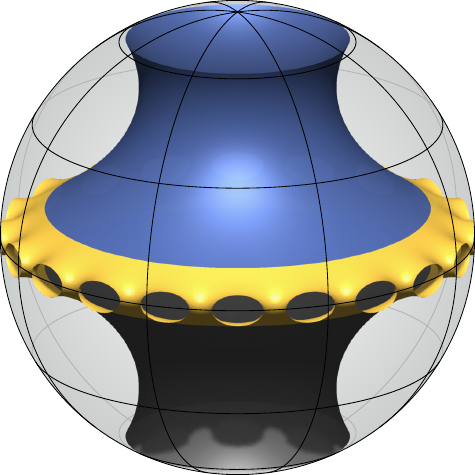}};
\draw(1,0,0.125)node{$\two{\towr}_n$};
\draw(0,0,0.33)node{$\two{\catr}_n$};
\end{tikzpicture}
\hfill
\pgfmathsetmacro{\phiO}{157} 
\tdplotsetmaincoords{\thetaO}{\phiO}
\begin{tikzpicture}[tdplot_main_coords,line cap=round,line join=round,scale=\unitscale,baseline={(0,0,0)}]
\draw(0,0,0)node{\includegraphics[width=\fbmsscale\textwidth,page=5]{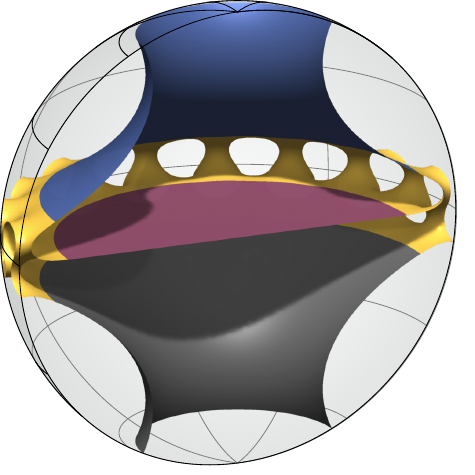}};
\tdplottransformmainscreen{-1}{0}{0}
\pgfmathsetmacro{\VecHx}{\tdplotresx}
\pgfmathsetmacro{\VecHy}{\tdplotresy}
\tdplottransformmainscreen{0}{-1}{0}
\pgfmathsetmacro{\VecVx}{\tdplotresx}
\pgfmathsetmacro{\VecVy}{\tdplotresy}
\draw(0,-0.2,0)node[cm={\VecHx ,\VecHy ,\VecVx ,\VecVy ,(0,0)}]{\huge$\three{\discr}_m$};
\draw(0,0,0.87)node{$\three{\catr}_m$};
\draw(0.9,0,0.125)node{$\three{\towr}_m$};
\end{tikzpicture}
\caption{Decomposition of $\zerogen_n$ (left) and $\threebc_m$ (right, cutaway view). }%
\label{fig:decomposition}%
\vspace*{3em}
\pgfmathsetmacro{\globalscale}{1.33}
\pgfmathsetmacro{\thetaO}{72}
\pgfmathsetmacro{\phiO}{45}
\tdplotsetmaincoords{\thetaO}{\phiO}
\pgfmathsetmacro{\height}{4.33}
\begin{tikzpicture}[scale=\globalscale,tdplot_main_coords,line cap=round,line join=round,baseline={(0,0,0)},semithick]
\draw(0,0,0)--({-sqrt(20)},0,0)coordinate(x0);
\draw[loosely dotted](x0)--++(-0.6,0,0);
\begin{scope}[thin,densely dotted]
\draw({20*(-1+cos(9/2))},{20*sin( 9/2)},0)--({(-20+(20-sqrt(20))*cos(9/2))},{(20-sqrt(20))*sin( 9/2)},0);
\draw(0,0,0)node[scale=\globalscale]{\includegraphics[page=6]{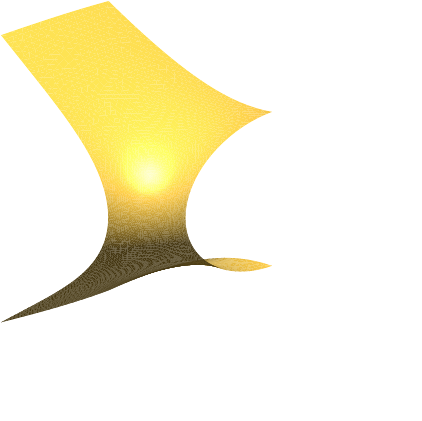}};
\draw({20*(-1+cos(9/2))},{20*sin(-9/2)},0)--({(-20+(20-sqrt(20))*cos(9/2))},{(20-sqrt(20))*sin(-9/2)},0);
\coordinate(NW)at(current bounding box.north west);
\coordinate (Shift) at (-20,0,0);
\tdplotsetrotatedcoordsorigin{(Shift)}
\tdplotsetthetaplanecoords{180/20/2}
\tdplotdrawarc[tdplot_rotated_coords]{(0,0,0)}{20}{90+11}{90-11}{}{} 
\tdplotsetthetaplanecoords{-180/20/2}
\tdplotdrawarc[tdplot_rotated_coords]{(0,0,0)}{20}{90+11}{90-11}{}{} 
\tdplotdrawarc{(Shift)}{20}{-180/20/2}{180/20/2}{}{} 
\end{scope}
\pgfresetboundingbox
\path(NW);
\draw[->](0,0,0)--(2,0,0)node[below]{$x$};
\tdplottransformmainscreen{0}{1}{0}
\pgfmathsetmacro{\VecHx}{\tdplotresx}
\pgfmathsetmacro{\VecHy}{\tdplotresy} 
\tdplottransformmainscreen{-cos(26)}{0}{sin(26)}
\pgfmathsetmacro{\VecVx}{\tdplotresx}
\pgfmathsetmacro{\VecVy}{\tdplotresy} 
\draw( -3.5,0,2.5)node[below=1ex, cm={\VecHx ,\VecHy ,\VecVx ,\VecVy ,(0,0)}]{\large$\two{\towrdom{n}}$}; 
\end{tikzpicture}
\hfill
\begin{tikzpicture}[scale=\globalscale,tdplot_main_coords,line cap=round,line join=round,baseline={(0,0,0)},semithick]
\draw[black!66,thin,densely dotted]({20*(-1+cos(9/2))},{20*sin( 9/2)},0)--({(-20+(20-sqrt(20))*cos(9/2))},{(20-sqrt(20))*sin( 9/2)},0)
;
\draw(0,0,0)node[scale=\globalscale]{\includegraphics[page=7]{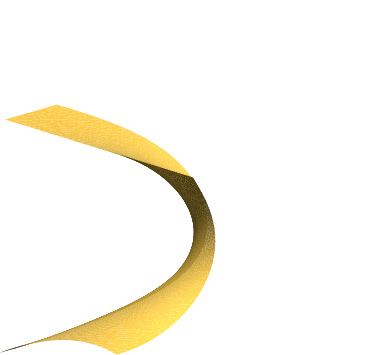}};
\draw[](0,0,0)--({-sqrt(20)},0,0)coordinate(x0);
\draw[loosely dotted](x0)--++(-0.6,0,0);
\draw(0,0,0)node[scale=\globalscale]{\includegraphics[page=8]{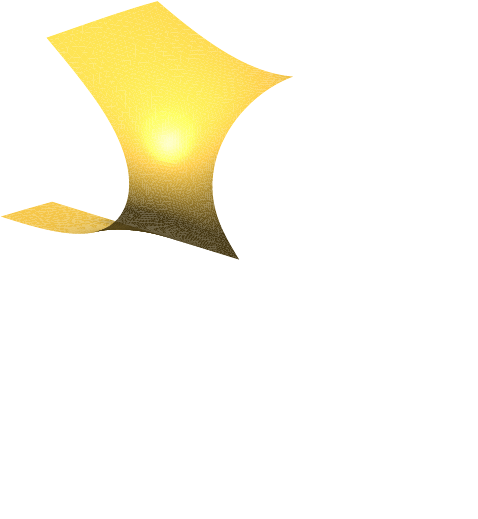}};
\begin{scope}[thin,densely dotted]
\coordinate(NW)at(current bounding box.north west);
\coordinate (Shift) at (-20,0,0);
\tdplotsetrotatedcoordsorigin{(Shift)}
\tdplotsetthetaplanecoords{180/20/2}
\tdplotdrawarc[tdplot_rotated_coords]{(0,0,0)}{20}{90+11}{90-11}{}{} 
\tdplotsetthetaplanecoords{-180/20/2}
\tdplotdrawarc[tdplot_rotated_coords]{(0,0,0)}{20}{90+11}{90-11}{}{} 
\tdplotdrawarc{(Shift)}{20}{-180/20/2}{180/20/2}{}{} 
\draw({20*(-1+cos(9/2))},{20*sin(-9/2)},0)--({(-20+(20-sqrt(20))*cos(9/2))},{(20-sqrt(20))*sin(-9/2)},0);
\end{scope}
\pgfresetboundingbox
\path(NW)--++(0,0,-9.1);%
\draw[->](0,0,0)--(2,0,0)node[below]{$x$};
\tdplottransformmainscreen{0}{1}{0}
\pgfmathsetmacro{\VecHx}{\tdplotresx}
\pgfmathsetmacro{\VecHy}{\tdplotresy} 
\tdplottransformmainscreen{-cos(25.38)}{0}{ sin(25.38)}
\pgfmathsetmacro{\VecVx}{\tdplotresx}
\pgfmathsetmacro{\VecVy}{\tdplotresy} 
\draw( -3,0,3.5)node[below=1ex,cm={\VecHx ,\VecHy ,\VecVx ,\VecVy ,(0,0)}]{\large$\three{\towrdom{m}}$};
\end{tikzpicture}
\caption{The dilated truncations $\two{\towrdom{n}}$ (left) and $\three{\towrdom{m}}$ (right). }%
\label{fig:rescaled-pieces}%
\end{figure}

\paragraph{Geometric estimates.}
Before proceeding,
we declare the following
abbreviated notation
for the metrics and second fundamental forms
on $\two{\towrdom{n}}$ and $\three{\towrdom{m}}$
(induced by their inclusions in $(\R^3,\met{\R^3})$):
\begin{align*}
\two{g}_n
&\vcentcolon=
  \met{\two{\towrdom{n}}},
&
\three{g}_m
&\vcentcolon=
  \met{\three{\towrdom{m}}}, 
&
\two{A_n}
&\vcentcolon=
  \twoff{\two{\towrdom{n}}},
&
\three{A_m}
&\vcentcolon=
  \twoff{\three{\towrdom{m}}}.
\end{align*}
In analogy with
\eqref{towerCutoffsConfFactsConfMets}
we first write $\two{\psi}_n$, $\three{\psi}_m$
for the unique functions
on $\two{\towrdom{n}}$, $\three{\towrdom{m}}$
such that
\begin{equation*}
\two{\psi}
  =
  \left(n\circ\varphi^{\two{\towr}_n}\right)^*
  \two{\psi}_n,
\qquad
\three{\psi}
  =
  \left((m+1)\circ\varphi^{\three{\towr}_m}\right)^*
  \three{\psi}_m
\end{equation*}
and then in turn define
\begin{equation}
\label{towrCutoffsConfFactsConfMets}
\begin{aligned}
\two{\conffact_n}
  &\vcentcolon=
  \sqrt{
    \two{\psi_n}
    +\frac{1}{2}\abs[\big]{\two{A_n}}_{\two{g}_n}^2
      (1-\two{\psi_n})
    +e^{-2n} 
},
&\qquad
\two{h_n}
  &\vcentcolon=
  (\two{\conffact}_n)^2\two{g}_n,
\\
\three{\conffact_m}
  &\vcentcolon=
  \sqrt{
    \three{\psi_m}
    +\frac{1}{2}\abs[\big]{\three{A_m}}_{\three{g}_m}^2
      (1-\three{\psi_m})
    +e^{-2m}  
},
&
\three{h}_m
  &\vcentcolon=
  (\three{\conffact}_m)^2\three{g}_m.
\end{aligned}
\end{equation}
The terms $e^{-2n}$ and $e^{-2m}$ above
are included to ensure the conformal factors
vanish nowhere.
For the sake of brevity, and consistently with the notation adopted in the previous subsections, 
we set
\begin{equation*}
\begin{aligned}
&\two{\liptowrdom{n}}
  \vcentcolon=
  (\two{\towrdom{n}},\two{h}_n),
\qquad
&&\two{\liptowrdom{n}}(s)
  \vcentcolon=
  (\two{\towrdom{n}}(s),\two{h}_n)
\\
&\three{\liptowrdom{m}}
  \vcentcolon=
  (\three{\towrdom{m}},\three{h}_m),
\qquad
&&\three{\liptowrdom{m}}(s)
  \vcentcolon=
  (\three{\towrdom{m}}(s),\three{h}_m),
\end{aligned}
\end{equation*}
so that $\two{\liptowrdom{n}}$
and $\three{\liptowrdom{m}}$
and their truncations $\two{\liptowrdom{n}}(s)\subset \two{\towrdom{n}}$ and $\three{\towrdom{m}}(s)\subset \three{\towrdom{m}}$
are always understood as being equipped with the conformal metrics $\two{h}_n$ and $\three{h}_m$, rather than $\two{g}_n$ and $\three{g}_m$.

\begin{lemma}
[Convergence of $\two{\towrdom{n}}(s)$ and $\three{\towrdom{m}}(s)$]
\label{towConv}
For every $s>0$ there exists $m_s>0$ such that for every integer $m>m_s$
\begin{enumerate}[label={\normalfont(\roman*)}]
\item
the region
$\three{\towrdom{m}}(s)$
is defined
and is the diffeomorphic image
under $(m+1)\varphi^{\three{\towr_m}}$
of $\three{\towdom}(s)$,
\item
$
 (m+1)\varphi^{\three{\towr_m}}
 \bigl(\three{\towdom}(s) \cap \{x=0\}\bigr)
 =
 \three{\towrdom{m}}(s) \cap (m+1)\Sp^2
$,
\item
$\varphi^{\three{\towr_m}}$
commutes with $\refl_{\{z=0\}}$,
and
\item 
$
 \three{\towr_m}(s)
 =
 (m+1)^{-1}
 \apr_{m+1}\three{\towrdom{m}}(s)
$
is a surface with smooth boundary.
\end{enumerate}
Moreover, for every $s>0$ and $\alpha \in \interval{0,1}$ 
\begin{enumerate}[resume*]
\item
$
 \left((m+1)\circ\varphi^{\three{\towr_m}}\right)^*
 \three{g}_m
 \xrightarrow[m \to \infty]
   {C^{1,\alpha}(\three{\towdom}(s),\met{\three{\tow}})}
 \met{\three{\tow}}
$ and
\item
$
 \left((m+1)\circ \varphi^{\three{\towr_m}}\right)^*
 \three{A}_m
 \xrightarrow[m \to \infty]
   {C^{0,\alpha}(\three{\towdom}(s), \met{\three{\tow}})}
 \twoff{\three{\tow}}
$.
\end{enumerate}
All the above statements have analogues for $\zerogen_n$ in place of $\threebc_m$, mutatis mutandis.
\end{lemma}

The first four claims are immediate from the definitions,
while the convergence assertions are ensured,
in the case of $\threebc_m$,
by the following estimates
from \cite{CSWnonuniqueness},
the case of $\zerogen_n$
being completely analogous.
Namely,
the estimate \cite[(5.20)]{CSWnonuniqueness}
provides $C^{2,\alpha}$
bounds for the defining function
of $\threebc_m$ as a graph over
the corresponding initial surface,
so controlling the
projection map
$\three{\varpi}_m$ from $\threebc_m$
to the initial surface.
The same estimate \cite[(5.20)]{CSWnonuniqueness}
also bounds the parameter value
for the initial surface
from the one-parameter family
that is selected
to produce the final one.
On the other hand,
\cite[Proposition 3.18]{CSWnonuniqueness}
provides estimates on the initial surface,
in terms of the datum $g$
as well as the value of the continuous parameter.
(As an aid to extracting the required information,
we point out that
the map $\varpi^{\vphantom{|}}_{M_{m,\xi}}$
in \cite[(3.43)]{CSWnonuniqueness}
is essentially
(that is: up to some quotienting
and the exact extent of the domains)
the inverse of the map
$\three{\varpi}_{m-1} \circ \varphi^{\three{\towr_{m-1}}}$
of the present article.)

Let us consider the other portions of our surfaces. 
By construction
$\two{\varpi}_n(\two{\catr}_n)$
and $\three{\varpi}_m(\three{\catr}_m)$
(subsets of the initial surfaces)
are graphs
(under the Euclidean metric $\geuc$)
over subsets of
$\K_{\two{\epsilon}(n)}$
and $\K_{\three{\epsilon}(m)}$,
and $\three{\varpi}_m(\three{\discr}_m)$
a graph over $\B^2$.
Thus, by composition with a further projection,
we obtain injective maps
$
 \two{\varpi}_n(\two{\catr}_n)
 \to
 \K_{\two{\epsilon}(n)}
$,
$
 \three{\varpi}_m(\three{\catr}_m)
 \to
 \K_{\three{\epsilon}(m)}
$,
and
$
 \three{\discr}_m
 \to 
 \B^2
$.
Moreover,
the image of each of these three maps
is $\Ogroup(2)$ invariant:
the image of the third is a disc
with radius tending to $1$ as $m \to \infty$,
the image of the second is a catenoidal annulus
with upper boundary circle
coinciding with that of
$\K_{\three{\epsilon}(m)}$
and lower boundary circle
tending to that of
$\K_{\three{\epsilon}(m)}$
as $m \to \infty$;
the image of the first
admits an analogous description.

In particular,
by composing further with dilations
of scale factor tending to $1$,
we obtain diffeomorphisms
\begin{equation*}
\varphi^{\three{\discr}_m}
\colon
\B^2\to
\three{\discr}_m;
\end{equation*}
similarly reparametrizing in the radial direction one also obtains diffeomorphisms
\begin{equation*}
\varphi^{\two{\catr}_n}
  \colon
  \K_0
  \to 
  \two{\catr}_n,
\qquad
\varphi^{\three{\catr}_m}
  \colon
  \K_0
  \to 
  \three{\catr}_m.
\end{equation*}
The inverses of these maps may be regarded
as small perturbations
(for $n$ and $m$ large)
of nearest-point projection
onto $\R^2 \subset \B^2$
or onto the complete catenoid
containing $\K_0$, as appropriate.
Somewhat more formally,
by reference to \cite{CSWnonuniqueness}
(specifically Proposition 3.18
and estimate (5.20) therein),
much as in the proof of
Lemma \ref{towConv},
we confirm the following properties
of $\two{\catr}_n$, $\three{\catr}_m$,
and $\three{\discr}_m$.

\begin{lemma}
[Convergence of $\two{\catr_n}$ and $\three{\catr_m}$]
\label{catConv}
There exists $m_0>0$ such that for each integer $m>m_0$

\begin{enumerate}[label={\normalfont(\roman*)}]
\item $\varphi^{\three{\catr}_m}$ is defined and
a diffeomorphism from $\K_0$ onto $\three{\catr}_m$,

\item $\varphi^{\three{\catr}_m}$
commutes with each element of $\pyr_{m+1}$,
and

\item $\varphi^{\three{\catr}_m}$ takes
the upper boundary component of $\K_0$
to the upper boundary component of $\three{\catr}_m$.
\end{enumerate}
Moreover,
for every $\alpha \in \interval{0,1}$
\begin{enumerate}[resume*]
\item
$
 (\varphi^{\three{\catr}_m})^*
   \met{\threebc_m}\Big|_{\three{\catr}_m}
 \xrightarrow[m \to \infty]{C^{1,\alpha}(\K_0,\met{\K_0})}
 \met{\K_0}
$ and

\item
$
 (\varphi^{\three{\catr}_m})^*
   {\twoff{\threebc_m}}\Big|_{\three{\catr}_m}
   \xrightarrow[m \to \infty]
     {C^{0,\alpha}(\K_0, \met{\K_0})}
   \twoff{\K_0}
$.
\end{enumerate}

All the above statements have analogues
for $\zerogen_n$ in place of $\threebc_m$,
mutatis mutandis.
\end{lemma}

\begin{lemma}
[Convergence of $\three{\discr_g}$]
\label{discConv}
There exists $m_0>0$ such that for each integer $m>m_0$

\begin{enumerate}[label={\normalfont(\roman*)}]
\item $\varphi^{\three{\discr_m}}$ is defined and
a diffeomorphism from $\B^2$ onto $\three{\discr}_m$
and

\item $\varphi^{\three{\discr}_m}$
commutes with each element of $\apr_{m+1}$.
\end{enumerate}

Moreover,
for each $\alpha \in \interval{0,1}$

\begin{enumerate}[resume*]
\item
$
 (\varphi^{\three{\discr}_m})^*
   \met{\threebc_m}\Big|_{\three{\discr}_m}
 \xrightarrow[m \to \infty]{C^{1,\alpha}(\B^2,\met{\B^2})}
 \met{\B^2}
$ and

\item
$
 (\varphi^{\three{\discr}_m})^*
   {\twoff{\threebc_m}}\Big|_{\three{\discr}_m}
   \xrightarrow[m \to \infty]{C^{0,\alpha}(\B^2,\met{\B^2})}
   0
$.
\end{enumerate}
\end{lemma}

Last we focus on the transition regions.
Let us agree to write
$\two{t}_n$
and $\three{t}_m$
for the distance functions
on $n\K_{\two{\epsilon}(n)}$
and $(m+1)\K_{\three{\epsilon}(m)}$
from their respective lower boundary
circles.
By construction
(assuming $s$ large enough in absolute terms)
$n\two{\varpi}_n(n^{-1}\two{\transr_n}(s))$
has two connected components,
one a graph over the
catenoidal annular wedge
\begin{equation*}
\{
  s
  \leq
  \two{t}_n
  \leq
  \sqrt{n}
\}
\cap
\Wedge{-\pi/(2n)}{\pi/(2n)}
\subset n\K_{\two{\epsilon}(n)}
\end{equation*}
and the other 
the reflection of this last one
through $\{z=0\}$,
while
$
 (m+1)
 \three{\varpi}_m
 ((m+1)^{-1}\three{\transr_n}(s))
$
has three connected components,
one a graph over the
planar annular wedge
\begin{equation*}
\{
 s
 \leq
 (m+1)-r
 \leq
 \sqrt{m+1}
\}
\cap
\Wedge{-\pi/(2(m+1))}{\pi/(2(m+1))}
\cap (m+1)\B^2,
\end{equation*}
another
a graph over the
catenoidal annular wedge
\begin{equation*}
\{
  s
  \leq
  \three{t}_m
  \leq
  \sqrt{m+1}
\}
\cap \Wedge{-\pi/(2(m+1))}{\pi/(2(m+1))}
\subset (m+1)\K_{\three{\epsilon}(m)},
\end{equation*}
and the third the reflection
of this last one through
$\{y=z=0\}$.

Projecting onto these
rotationally invariant sets
and parametrizing them by arc length $t$
in the ``radial'' direction and
$
 \vartheta
 \vcentcolon=
 n\theta
$
or, respectively,
$
 \vartheta
 \vcentcolon=
 (m+1)\theta,
$
in the angular direction
(with $\theta$ restricted to
the appropriate interval containing $0$),
we obtain injective maps
\begin{equation*}
\varphi^{\two{\transr_n}(s),\K}
  \colon
  \IntervaL[\Big]{s,\sqrt{n}}
    \times \IntervaL[\Big]{-\frac{\pi}{2},\frac{\pi}{2}}
  \to
  \two{\transr}_n,
\qquad
\varphi^{\three{\transr_m}(s),\K},
\varphi^{\three{\transr_m}(s),\B^2}
  \colon
  \IntervaL[\Big]{s,\sqrt{m+1}}
    \times \IntervaL[\Big]{-\frac{\pi}{2},\frac{\pi}{2}}
  \to 
  \three{\transr}_m
\end{equation*}
whose images
are components of
$\two{\transr_n}(s)$
and $\three{\transr_m}(s)$
that generate the latter regions
under $\cycgrp{\{z=0\}}$
and $\cycgrp{\{y=z=0\}}$
respectively.

\begin{lemma}[Estimates on $\two{\transr_n}(s)$ and $\three{\transr_m}(s)$]
\label{transitionEstimates}
Let $\alpha \in \interval{0,1}$.
There exists $s_0>0$ such that for each $s>s_0$ there exists $m_s>0$ such that for every integer $m>m_s$
\begin{enumerate}[label={\normalfont(\roman*)}]
\item
\label{transr_cat_A}
$
 (\varphi^{\three{\transr_m}(s),\K})^*
   \abs{\three{A}_m}_{\three{g}_m}^2(t,\vartheta)
 =
 a_1(t)m^{-2}
   + a_2(t,\vartheta)e^{-t/4}
$
for some smooth functions $a_1,a_2$
having $C^{0,\alpha}(dt^2+d\vartheta^2)$ norm
bounded independently of $m$ and $s$,

\item\label{transr-ii}
$
 (\varphi^{\three{\transr_m}(s),\B^2})^*
   \abs{\three{A}_m}_{\three{g}_m}^2(t,\vartheta)
 =
 a_3(t,\vartheta)e^{-t/4}
$
for some smooth function $a_3$
having $C^{0,\alpha}(dt^2+d\vartheta^2)$ norm
bounded independently of $m$ and $s$,

\item
\label{transr_cat_met}
$
 (\varphi^{\three{\transr_m}(s),\K})^*\three{g}_m
 =
 dt^2
   +(1+m^{-1}tf^1(t))
        d\vartheta^2
 +f^1_{uv}(t,\vartheta)
   e^{-t/4} \, du \, dv
$
for some smooth functions $f^1,f^1_{uv}$
having $C^{1,\alpha}(dt^2+d\vartheta^2)$ norm
bounded independently of $m$ and $s$,

\item\label{transr-iv}
$
 (\varphi^{\three{\transr_m}(s),\B^2})^*\three{g}_m
 =
 dt^2
   +(1+m^{-1}tf^2(t))
     d\vartheta^2
 +f^2_{uv}(t,\vartheta)
   e^{-t/4} \, du \, dv
$
for some smooth functions $f^2,f^2_{uv}$
having $C^{1,\alpha}(dt^2+d\vartheta^2)$ norm
bounded independently of $m$ and $s$,

\item
\label{transr_cat_lap}
$
 \Delta_{(\varphi^{\three{\transr_m}(s),\K})^*\three{g}_m}
 =
 \partial_t^2
  +m^{-1}c_1^t(t)
    \partial_t 
  +
  (
    1
    +m^{-1/2}
      b_1^{\vartheta\vartheta}(t)
  )
      \partial_\vartheta^2
  +e^{-t/4}(
    b_2^{uv}(t,\vartheta)
      \partial_u\partial_v
    +c_2^u(t,\vartheta)
      \partial_u
   )
$
for some smooth functions
$b_1^{\vartheta\vartheta}, b_2^{uv},
c_1^t, c_2^u$
having $C^{0,\alpha}(dt^2+d\vartheta^2)$
norm bounded independently of $m$ and $s$,
and

\item\label{transr-vi}
$
 \Delta_{(\varphi^{\three{\transr_m}(s),\B^2})^*\three{g}_m}
 =
 \partial_t^2
  +m^{-1}c_3^t(t)
    \partial_t 
  +
  (1+m^{-1/2}
    b_3^{\vartheta\vartheta}(t)
  )\partial_\vartheta^2
  +e^{-t/4}(
    b_4^{uv}(t,\vartheta)
      \partial_u\partial_v
    +c_4^u(t,\vartheta)
      \partial_u
   )
$
for some smooth functions
$b_3^{\vartheta\vartheta}, b_4^{uv},
c_3^t, c_4^i$
having $C^{0,\alpha}(dt^2+d\vartheta^2)$
norm bounded independently of $m$ and $s$.
\end{enumerate}
It is understood that, in items \ref{transr_cat_met}, \ref{transr-iv}, \ref{transr_cat_lap}, \ref{transr-vi} one sums over $u,v \in \{t,\vartheta\}$. 

Furthermore,
\begin{enumerate}[resume*]
\item\label{transr-vii}
$
  \displaystyle{
    \lim_{s \to \infty} \lim_{m \to \infty}
  }
  \hausmeas{2}{\three{h_m}}(\three{\transr_m}(s))
 =
 0
$.
\end{enumerate}
The same claims hold for $\two{\transr_n}(s)$,
mutatis mutandis.
\end{lemma}

\begin{proof}
Again the estimates are ultimately justified
by reference to the construction \cite{CSWnonuniqueness},
most specifically (5.20) and Proposition 3.18 therein. 
That said, we also note how claim \ref{transr_cat_lap} follows easily from \ref{transr_cat_met}, as does claim \ref{transr-vi} from \ref{transr-iv}; 
furthermore, it is clear that the justification of \ref{transr-ii} is analogous to (in fact simpler than) \ref{transr_cat_A}, and \ref{transr-iv} is analogous to \ref{transr_cat_met}.
As a result, we briefly explain the ideas behind
the elementary computations required for the proof,
in the case of $\three{\transr_m}(s)$, so with regard to items \ref{transr_cat_A} and \ref{transr_cat_met}.

The projection of this region onto the blown-up
initial surface $(m+1)\initthreebc_m$
is itself constructed as a graph over
$(m+1)\K_{\three{\epsilon}(m)}$ or $\B^2$.
Estimate \cite[(5.20)]{CSWnonuniqueness}
ensures that $m\three{\epsilon}(m)$
is bounded uniformly in $m$. The defining function of the above graph
is obtained by ``transferring''
the defining functions
of the corresponding ends of $\three{\tow}$
over their asymptotic planes.
These defining functions decay exponentially
in the distance along the planes.
In turn $\three{\transr_m}(s)$
is a graph over this portion of the initial surface
with defining function that is also guaranteed
(by \cite[(5.20)]{CSWnonuniqueness})
to decay exponentially, though
a priori at a slower rate;
we have chosen $1/4$ somewhat arbitrarily.
This accounts for all exponential factors
appearing in the estimates.

The $m$-dependent terms in the
estimates for the metric (and Laplacian)
arise simply from the choice of $(t,\vartheta)$
coordinates on disc and catenoidal models.
The $m^{-2}$ term in the first item
arises from scaling
the second fundamental form of the
``asymptotic'' catenoid to this component
(while the corresponding term for the disc
vanishes).
With the estimates for the second fundamental form
in place,
the final item -- the area estimate --
follows 
(recalling the definitions
\eqref{towrCutoffsConfFactsConfMets})
from the bound
\begin{equation*}
\int_{-\pi/2}^{\pi/2}
\int_s^{\sqrt{m+1}} \bigl(a_1m^{-2} + a_2e^{-t/4}\bigr) \, dt
\, d\vartheta
\leq
C\Bigl(m^{-3/2} + e^{-s/4}\Bigr),
\end{equation*}
and the analogous estimate concerning the disk-type component instead.
\end{proof}

\subsection{Regionwise spectral convergence}
\label{subs:RegionSpectConv}

For each region $S$ among $\two{\towr_n}$, $\three{\towr_m}$,
$\two{\catr_n}$, $\three{\catr_n}$,
and $\three{\discr_m}$ (depicted in Figure \ref{fig:decomposition})
we write
$\indexform{S}_\neum$
for the Jacobi form of $S$
as a minimal surface in $\B^3$
with boundary,
subject to the Robin condition
\eqref{fbmsRobinConditionInBall}
where $\partial S$ meets $\partial \B^3$
and subject to the Neumann condition elsewhere:
recalling \eqref{DirAndNeumInternalizationsOfBilinearForm}, we set
\begin{equation*}
\indexform{S}_\neum
\vcentcolon=
  \begin{cases}
    \nint{\Bigl(\indexform{\zerogen_n}\Bigr)}_S
      & \mbox{for } S \subset \zerogen_n
    \\[1ex]
    \nint{\Bigl(\indexform{\threebc_m}\Bigr)}_S
      & \mbox{for } S \subset \threebc_m
  \end{cases}
\end{equation*}
(where on the right-hand side we slightly abuse notation
in that in place of $S$ we really mean its interior).
Similarly,
for $S$ either $\two{\towrdom{n}}$
or $\three{\towrdom{m}}$
we write
$
 \indexform{S}_\neum
 $
for the Jacobi form of $S$
as a minimal surface in
either $n\B^3$ or $(m+1)\B^3$,
subject to the Robin condition
either $du(\eta)=n^{-1}u$ or $du(\eta)=(m+1)^{-1}u$
where $\partial S$ meets
either $n\Sp^2$ or $(m+1)\Sp^2$, respectively,
and subject to the Neumann condition elsewhere.
Keeping in mind the statement of Proposition \ref{mr}, we stress that the adjunction of Neumann conditions in the ``interior'' boundaries is motivated by our task of deriving \emph{upper} bounds on the Morse index of our examples. Recalling the notation $\two{\liptowrdom{n}}$ and $\three{\liptowrdom{n}}$, we remark that 
the bilinear forms
$\indexform{S}_\neum$
and
$\indexform{\widehat{S}}_\neum$
agree by definition
for each $S$ as above,
but whenever we refer to
the eigenvalues,
eigenfunctions,
index, and nullity
of the latter
we shall always mean
those defined with respect
to the $\two{h}_n$ or $\three{h}_m$ metric.

In the notation of \eqref{bilinear_form_def}
we have in particular
(cf. Proposition \ref{conformalInvariance})
\begin{equation}
\label{towrdomJacobiFormInTwoMetrics}
\begin{aligned}
\indexform{\two{\towrdom{n}}}_\neum
&=
\weakbf
  \Bigl[
    \two{\towrdom{n}},~
    \two{g}_n,~
    \two{q_n}=\abs{\two{A}_n}_{\two{g}_n}^2,~
    \two{r_n}=n^{-1},~
    \\
    &\qquad \qquad
    \dbdy\two{\towrdom{n}}=\emptyset,~
    \nbdy\two{\towrdom{n}}
      =
      \partial\two{\towrdom{n}}
        \setminus n\Sp^2,~
    \rbdy\two{\towrdom{n}}
      =
      \partial\two{\towrdom{n}}
        \setminus \closure{\nbdy\two{\towrdom{n}}}
  \Bigr]
\\
&=
\weakbf
  \Bigl[
    \two{\towrdom{n}},~
    \two{h_n},~
    \Bigl(\two{\rho_n}\Bigr)^{-2}\two{q_n},~
    \Bigl(\two{\rho_n}\Bigr)^{-1}n^{-1},~
    \emptyset,~
    \partial\two{\towrdom{n}}
        \setminus n\Sp^2,~
    \partial\two{\towrdom{n}}
        \setminus \closure{\nbdy\two{\towrdom{n}}}
  \Bigr]
  \\
  &=\indexform{\two{\liptowrdom{n}}}_\neum
\end{aligned}
\end{equation}
and similarly for
$ \indexform{\three{\towrdom{m}}}_\neum=\indexform{\three{\liptowrdom{m}}}_\neum$.
Observe further (cf. Lemma \ref{dom_red_ext}
and Proposition \ref{conformalInvariance})
\begin{align*}
\equivind{\pri_n}(\indexform{\two{\towr_n}}_\neum)
&=\symind{\cycgrp{\{z=0\}}}{+}
\Bigl(
\indexform{\two{\liptowrdom{n}}}_\neum
\Bigr),
&
\equivind{\pyr_n}(\indexform{\two{\towr_n}}_\neum)
&=
\ind
\Bigl(
\indexform{\two{\liptowrdom{n}}}_\neum
\Bigr),
\\
\equivind{\apr_{m+1}}(\indexform{\three{\towr_m}}_\neum)
&=\symind{\cycgrp{\{y=z=0\}}}{-}
\Bigl(
\indexform{\three{\liptowrdom{m}}}_\neum
\Bigr),
&
\equivind{\pyr_{m+1}}(\indexform{\three{\towr_m}}_\neum)
&=
\ind
\Bigl(
\indexform{\three{\liptowrdom{m}}}_\neum
\Bigr),
\end{align*}
and likewise for the corresponding nullities.

\begin{lemma}[Equivariant index and nullity on $\two{\catr_n}$, $\three{\catr_m}$, and $\three{\discr_m}$]
\label{catrAndDiscrIndexAndNullity}
There exist $n_0,m_0>0$ such that we have the following indices and nullities for all integers $n>n_0$ and $m>m_0$:
\[
\begin{array}{c|c||c|c|}
S & \grp
  & \equivind{\grp}(\indexform{S}_\neum)
  & \equivnul{\grp}(\indexform{S}_\neum) \\
\hline\hline
\two{\catr_n} & \pyr_n & 1 & 0 \\
\hline
\three{\catr_m} & \pyr_{m+1} & 1 & 0 \\
\hline
\three{\discr_m} & \apr_{m+1} & 0 & 0 \\
\hline
\end{array}
\]
Additionally,
still assuming $m>m_0$
we have the upper bound
\begin{equation*}
\equivind{\pyr_{m+1}}\Bigl(\indexform{\three{\discr_m}}_\neum\Bigr)
     +\equivnul{\pyr_{m+1}}\Bigl(\indexform{\three{\discr_m}}_\neum\Bigr)
\leq
1.
\end{equation*}
\end{lemma}

\begin{proof}
We use the convergence described in
Lemma \ref{catConv} and Lemma \ref{discConv}
along with Proposition~\ref{spectralCtyWrtCoeffsAndSyms}
to compare the low eigenvalues
of the regions in question
with those of their limiting models,
as recorded in Lemma \ref{K0IndexAndNullity}
and Lemma \ref{DiscIndexAndNullity}.
\end{proof}

While we have cut the surfaces
$\zerogen_n$ and $\threebc_m$
in such a way that
the resulting regions
$\two{\catr_n}$ and $\three{\catr_m}$
converge uniformly
to $\K_0$
and likewise $\three{\discr_m}$ to $\B^2$,
thereby securing the preceding lemma
in a straightforward fashion,
the cases of $\two{\towr_n}$
and $\three{\towr_m}$ are more subtle.
Our approach here
(especially the proof of eigenfunction
bounds in Lemma \ref{UniformBoundsOnTower}
and their application to
Lemma \ref{towrdomLowerBounds})
draws inspiration from
the analysis Kapouleas makes
of the invertibilty of the Jacobi operator
on ``extended standard regions''
in many gluing constructions;
for a specific example,
concerning Scherk towers glued to catenoids,
we refer the reader to
the proof of 
\cite[Lemma 7.4]{KapouleasEuclideanDesing}.

To proceed, recalling
\eqref{DirAndNeumInternalizationsOfBilinearForm},
for each $s>0$ and each integer $n$
(sufficiently large in terms of $s$)
we define
\begin{equation*}
\indexform{\two{\liptowrdom{n}}(s)}_\dir
  \vcentcolon=
  \dint{
    \left(
      \indexform{\two{\liptowrdom{n}}}_\neum
    \right)_{\two{\liptowrdom{n}}(s)}
  }
\quad \mbox{and} \quad
\indexform{\two{\liptowrdom{n}}(s)}_\neum
  \vcentcolon=
  \nint{
    \left(
      \indexform{\two{\liptowrdom{n}}}_\neum
    \right)_{\two{\liptowrdom{n}}(s)}
  }
\end{equation*}
and analogously for $\three{\liptowrdom{m}}(s)$
in place of $\two{\liptowrdom{n}}(s)$.

\begin{lemma}
[Spectral convergence for $\two{\liptowrdom{n}}(s)$
and $\three{\liptowrdom{m}}(s)$]
\label{towrdomCptSpecConv}

With the above notation,
we have
\begin{equation*}
\begin{aligned}
\eigenvalsym{\indexform{\two{\liptowdom}}}
  {i}{\cycgrp{\{z=0\}}}{\pm}
&=
\lim_{s \to \infty}
  \lim_{n \to \infty}
    \eigenvalsym
      {\indexform{\two{\liptowrdom{n}}(s)}_\dir}
      {i}{\cycgrp{\{z=0\}}}{\pm}
\\
&=
\lim_{s \to \infty}
  \lim_{n \to \infty}
    \eigenvalsym
      {\indexform{\two{\liptowrdom{n}}(s)}_\neum}
      {i}{\cycgrp{\{z=0\}}}{\pm}
\end{aligned}
\end{equation*}
for each integer $i \geq 1$
and each common choice of sign $\pm$
on both sides of each equation.
The analogous statements hold,
mutatis mutandis,
for $\three{\liptowrdom{m}}$
in place of $\two{\liptowrdom{n}}$.
\end{lemma}

\begin{proof}
Fix $i$.
By Lemma \ref{towConv}
and Proposition \ref{spectralCtyWrtCoeffsAndSyms}
for each $s>0$
we have
\begin{equation*}
\begin{aligned}
\lim_{n \to \infty}
  \eigenvalsym
    {\indexform{\two{\liptowrdom{n}}(s)}_\dir}
    {i}{\cycgrp{\{z=0\}}}{+}
=
\eigenvalsym
  {\indexform{\two{\liptowdom}(s)}_\dir}{i}
  {\cycgrp{\{z=0\}}}{+},
\\
\lim_{n \to \infty}
  \eigenvalsym
    {\indexform{\two{\liptowrdom{n}}(s)}_\neum}
    {i}{\cycgrp{\{z=0\}}}{+}
=
\eigenvalsym
  {\indexform{\two{\liptowdom}(s)}_\neum}{i}
  {\cycgrp{\{z=0\}}}{+}.
\end{aligned}
\end{equation*}
An application of
Lemma \ref{spectraOnTruncatedTowers}
completes the proof in this case,
and the proofs
of the remaining three cases
are structurally identical to this one.
\end{proof}

\begin{lemma}
[Eigenvalue upper bounds
on $\two{\liptowrdom{n}}$
and $\three{\liptowrdom{m}}$]
\label{towrdomUpperBounds}
With the above notation,
we have
\begin{align*}
\limsup_{n \to \infty}
    \eigenvalsym{\indexform{\two{\liptowrdom{n}}}_\neum}{i}{\cycgrp{\{z=0\}}}{\pm}
&\leq
  \eigenvalsym{\indexform{\two{\liptowdom}}}{i}{\cycgrp{\{z=0\}}}{\pm},
\\[1ex]
\limsup_{m \to \infty}
    \eigenvalsym{\indexform{\three{\liptowrdom{m}}}_\neum}{i}{\cycgrp{\{y=z=0\}}}{\pm}
&\leq
  \eigenvalsym{\indexform{\three{\liptowdom}}}{i}{\cycgrp{\{y=z=0\}}}{\pm}
\end{align*}
for each integer $i \geq 1$ and each common choice of sign $\pm$ on both sides of each equation.
\end{lemma}

\begin{proof}
We give the proof for the $+$ choice
on both sides of the top equation,
the proofs for the remaining three cases
being identical in structure to this one.
Fix $i \geq 1$.
By \eqref{symminmax},
considering extensions by zero
of functions corresponding
to the right-hand side below
to obtain valid test functions
corresponding to the left, we get at once the inequality
\begin{equation*}
\eigenvalsym
  {\indexform{\two{\liptowrdom{n}}}_\neum}{i}
  {\cycgrp{\{z=0\}}}{+}
\leq
\eigenvalsym
  {\indexform{\two{\liptowrdom{n}}(s)}_\dir}
  {i}{\cycgrp{\{z=0\}}}{+}
\end{equation*}
for all $s>0$ and all $n$ sufficiently large
in terms of $s$
that $\two{\liptowrdom{n}}(s)$ is defined.
We then finish by applying Lemma \ref{towrdomCptSpecConv}.
\end{proof}

\begin{lemma}
[Uniform bounds on eigenvalues and eigenfunctions of
$\indexform{\two{\liptowrdom{n}}}_\neum$
and
$\indexform{\three{\liptowrdom{m}}}_\neum$]
\label{UniformBoundsOnTower}
For each integer $i \geq 1$
there exist
$C_i,k_i>0$
such that
for each integer $k>k_i$
and
whenever
$\lambda_i^{(k)}$
is the $i$\textsuperscript{th}
eigenvalue
of
$\indexform{\two{\liptowrdom{k}}}_\neum$
or
$\indexform{\three{\liptowrdom{k}}}_\neum$
and $v_i^{(k)}$ is any corresponding eigenfunction
of unit $L^2$-norm
(under either
$\two{h}_k$
or $\three{h}_k$
as appropriate),
we have the bounds
\begin{equation*}
\max
 \Bigl\{
    \abs{\lambda_i^{(k)}},
    \nm{v_i^{(k)}}_{\sob},
    \nm{v_i^{(k)}}_{C^0}
  \Bigr\}\leq C_i
\end{equation*}
(where the $\sob$ norm
is defined via
either
$\two{h_n}$ or $\three{h_m}$
as applicable
and we emphasize that $C_i$
does not depend on $k$).
\end{lemma}

\begin{proof}
We will give the proof for
$\two{\liptowrdom{n}}$,
that for $\three{\liptowrdom{m}}$
being identical in structure.
Fix $i \geq 1$
and let $\lambda^{(n)}$ and $v^{(n)}$
be as in the statement
for each integer $n$
(suppressing the fixed index $i$);
it is our task to show that
by assuming $n$ large enough in terms of just $i$
we can ensure the asserted bounds
on $\lambda^{(n)}$ and $v^{(n)}$.
In particular our assumptions include
the normalization
$\nm{v^{(n)}}_{L^2(\two{\towrdom{n}},\two{h_n})}=1$. 

Lemma \ref{towrdomUpperBounds}
provides an upper bound
on $\lambda^{(n)}$,
independent of $n$.
We deduce a lower bound on $\lambda^{(n)}$
as follows. Keeping in mind the min-max characterization
\eqref{symminmax}
we observe that in the ratio
\begin{equation}
\label{lowestEvalRayleigh}
\frac
{\sk[\big]{u, \potential_nu}_{L^2(\two{\towrdom{n}},\two{h}_n)}
  +\sk[\big]{
    u|_{\rbdy\two{\towrdom{n}}},
    \robinpotential_n
      u|_{\rbdy\two{\towrdom{n}}}
  }_{L^2(\rbdy\two{\towrdom{n}},\two{h}_n)}
}
{\nm{u}_{L^2(\two{\towrdom{n}},\two{h}_n)}^2},
\end{equation}
with
\begin{align*}
\robinpotential_n
&\vcentcolon=
  \left(\two{\conffact}_n\right)^{-1}\Big|_{
    \rbdy\two{\towrdom{n}}
  } n^{-1}
  =
  (1+e^{-2n})^{-1/2}
  n^{-1},
&
\potential_n
&\vcentcolon=
  \left(\two{\conffact}_n\right)^{-2}
  \abs[\Big]{\two{A}_n}_{\two{g}_n}^2.
\end{align*}
we have not only a uniform upper bound on
$\robinpotential_n$, but also,
by inspecting
\eqref{towrCutoffsConfFactsConfMets}
and bearing in mind the
convergence described in Lemma \ref{towConv}
as well as the
boundedness (with decay)
of the second fundamental form of
$\two{\tow}$,
\begin{equation*}
\sup_n \
 \nm{\potential_n}_{C^0(\two{\towrdom{n}})}
<
\infty.
\end{equation*}
In addition, the convergence
in Lemma \ref{towConv}
further ensures that the constants
appearing in \eqref{interpolatingTraceInequality},
with
$(\lipdom,g)=(\two{\towrdom{n}},\two{h}_n)$
and $\rbdy\lipdom$ in place of $\partial\lipdom$
can be chosen uniformly in $n$: thus, employing such a trace inequality and exploiting the foregoing uniform bounds
we secure the promised uniform
lower bound on $\lambda^{(n)}$.

In turn, from the definitions of eigenvalues and eigenfunctions
and the normalization of $v^{(n)}$ we have
\begin{align*}
\nm[\Big]{
\nabla_{\two{h}_n}v^{(n)}
}_{L^2(\two{\towrdom{n}},\two{h}_n)}^2
=\lambda^{(n)}
  &+\sk[\Big]{
    v^{(n)}, \potential_n v^{(n)}
  }_{L^2(\two{\towrdom{n}},\two{h}_n)}
\\&+\sk[\Big]{
    v^{(n)}|_{\rbdy\two{\towrdom{n}}},
    \robinpotential_n
      v^{(n)}|_{\rbdy\two{\towrdom{n}}}
  }_{L^2(\rbdy\two{\towrdom{n}},\two{h}_n)}.
\end{align*}
The uniform bound on
$\nm{v^{(n)}}_{\sob(\two{\towrdom{n}},\two{h_n})}$
now follows, in view of the above equality, from
the upper bound on $\lambda^{(n)}$
as well as again the above uniform bounds
on $\potential_n$ and $\robinpotential_n$.

It remains to establish the uniform $C^0$ bound.
To start,
by Lemma \ref{dom_red_ext}
and standard elliptic regularity
$v^{(n)}$ is smooth up to the boundary: indeed,
it satisfies
\begin{equation}
\label{hMetricEigenvalueEqn}
\left\{\begin{aligned}
\Bigl(
  \Delta_{\two{h_n}}
  +\left(\two{\rho_n}\right)^{-2}
    \abs[\Big]{\two{A_n}}_{\two{g}_n}^2
  +\lambda^{(n)}
\Bigr)
  v^{(n)}
&=0 &&\text{ in } \two{\towrdom{n}},
\\
\two{h}_n
  \Bigl(
    \two{\eta}_n,
    \nabla_{\two{h}_n}v^{(n)}
  \Bigr)
  &=
  (1+e^{-2n})^{-1/2}n^{-1}v^{(n)}
  &&\text{ on } \rbdy \two{\towrdom{n}},
\\
\two{h}_n
  \Bigl(
    \two{\eta}_n,
    \nabla_{\two{h}_n}v^{(n)}
  \Bigr)
  &=0 &&\text{ on } \nbdy \two{\towrdom{n}},
\end{aligned}\right.
\end{equation}
with $\two{\eta}_n$
the outward
$\two{h}_n$ unit conormal
to $\two{\towrdom{n}}$.
As established above,
we have bounds independent of $n$
on $\abs{\lambda^{(n)}}$
and the $\potential_n$ and $\robinpotential_n$
functions.
By Lemma \ref{towConv}
(and the uniform geometry of $\two{\tow}$)
we also have uniform control over
the geometry
of $(\two{\towrdom{n}}(s),\two{h}_n)$
for each $s>0$ and all $n$ sufficiently large
in terms of $s$.

Standard elliptic regularity
therefore
ensures
that for every $s>0$
there exist $n_s>0$ and $\gamma(s)>0$
so that
\begin{equation}
\label{BoundOnCore}
\nm[\Big]{
  v^{(n)}\Big|_{\two{\towrdom{n}}(s)}
}_{C^0(\two{\towrdom{n}}(s),\two{h}_n)}
\leq
\gamma(s)
\text{ for every integer }
n>n_s.
\end{equation}

Since we do not have uniform control
on the geometry of
$(\two{\towrdom{n}}=\two{\towrdom{n}}(\sqrt{n}),\two{h}_n)$,
we do not obtain a global bound independent of $n$
in the same fashion.
Instead
the proof will be completed
by securing a
$C^0$ bound for $v^{(n)}$, independent of $n$,
on $\two{\transr_n}(s)$
for some $s>0$
to be determined. In the remainder of the proof $\gamma(s)$ will continue
to denote the above constant, depending on $s$,
while $C$ will denote a strictly positive constant
whose value may change from instance to instance
but can always be selected independently of $s$ and $n$.

To proceed
we multiply both sides of the PDE in
\eqref{hMetricEigenvalueEqn}
by $\left(\two{\rho}_n\right)^2$
to get
\begin{equation}
\label{h_eval_eqn_in_g_met}
\left(
  \Delta_{\two{g}_n}
  +\abs[\Big]{\two{A_n}}_{\two{g_n}}^2
  +\lambda^{(n)} \left(\two{\rho_n}\right)^2
\right)
v^{(n)}
=
0,
\end{equation}
and we aim to bound $v^{(n)}$
on $\two{\transr_n}(s)$
on the basis of this equation,
with unknown but controlled
(as we explain momentarily)
Dirichlet data on
 the portion of $\partial\two{\transr_n}(s)$
contained in the interior of $\two{\towrdom{n}}$
and with
homogeneous Neumann data on the rest of the boundary.
By the symmetries
it suffices to establish the estimate
on just
the component of $\two{\transr_n}(s)$
that is a graph over a subset of $n\K_0$.
(For $\three{\transr_m}(s)$
one must also consider
the component which is a graph
over a subset of $(m+1)\B^2$,
but this case does not differ
in substance from the one
we treat now.)

Recall the map
\begin{equation*}
\varphi^{\two{\transr_n}(s),\K}
  \colon
  \IntervaL[\Big]{s,\sqrt{n}}
    \times \IntervaL[\Big]{-\frac{\pi}{2},\frac{\pi}{2}}
  \to 
  \transr_n(s)
\end{equation*}
introduced above Lemma \ref{transitionEstimates}
and continue to write $(t,\vartheta)$
for the standard coordinates on its domain.
For the remainder of this proof
we abbreviate
$\varphi^{\two{\transr_n}(s),\K}$
to
$\varphi_{n,s}$
and its domain to
$R_{n,s}$.
Setting
$
 w^{(n)}
 \vcentcolon=
 \varphi_{n,s}^*
 v^{(n)}
$,
we pull back
\eqref{h_eval_eqn_in_g_met}
to get
\begin{equation*}
\Delta_{\varphi_{n,s}^*\two{g}_n}w^{(n)}
=
-w^{(n)}\varphi_{n,s}^*
  \left(
    \abs[\Big]{\two{A_n}}_{\two{g_n}}^2
    +\lambda^{(n)} \left(\two{\rho_n}\right)^2
  \right).
\end{equation*}
From the uniform bound on $\lambda^{(n)}$,
the expression for the conformal factor
in \eqref{towrCutoffsConfFactsConfMets},
and 
item \ref{transr_cat_A} of
Lemma \ref{transitionEstimates}
we in turn obtain
\begin{equation}
\label{PoissonOnRect}
\Delta_{\varphi_{n,s}^*\two{g}_n}w^{(n)}
=
\bigl(c_{n,s} e^{-t/4}+d_{n,s} n^{-2}\bigr)w^{(n)}
\end{equation}
for some smooth functions
$c_{n,s}, d_{n,s}$
having $C^{0,\alpha}(dt^2+d\vartheta^2)$
norms
uniformly bounded in $n$ and $s$,
with $\alpha \in \interval{0,1}$
now fixed for the rest of the proof.
(Here and below when referring to items of
Lemma \ref{transitionEstimates}
we have in mind of course the corresponding
statements for $\two{\transr_n}(s)$
in place of $\three{\transr_m}(s)$.)

Noting that we have
\eqref{PoissonOnRect}
for all sufficiently large $s$,
it now follows from
the $C^0$ bound
\eqref{BoundOnCore}
and
standard
interior Schauder estimates
(using also 
item \ref{transr_cat_met}
of Lemma \ref{transitionEstimates})
that
\begin{equation}
\label{DirEst}
\nm[\Big]{
  w^{(n)}(s,\cdot)
}_{C^{2,\alpha}(d\vartheta^2)}
\leq
C\gamma(s+1)
\mbox{ for every integer }
n>n_{s+1}.
\end{equation}

Since $v^{(n)}$
satisfies the homogeneous Neumann condition
along
$\partial \two{\towrdom{n}}$,
with the aid of
item \ref{transr_cat_met}
of Lemma \ref{transitionEstimates}
we have
\begin{align}
\label{NeumEstInhom}
(\partial_t w^{(n)})(\sqrt{n},\vartheta)
  &=
  e_{n,s}\,e^{-\sqrt{n}/4}
    (\partial_\vartheta w^{(n)})(\sqrt{n},\vartheta),
\\
\label{NeumEstHom}
(\partial_\vartheta w^{(n)})(\cdot, \pm \pi/2)
  &=
  0
\end{align}
for some smooth function
$e_{n,s}$
having $C^{1,\alpha}(dt^2+d\vartheta^2)$
norm bounded independently of $n$ and $s$.
(For \eqref{NeumEstHom}
we simply use the fact
that $\varphi_{n,s}$
has been constructed
by composing and restricting
maps which commute
with the symmetries of the construction,
including the reflections through planes
corresponding to
$\vartheta=\pm \pi/2$.)

Appealing again to standard Schauder estimates,
now also up to the boundary,
we can conclude from
\eqref{PoissonOnRect},
\eqref{DirEst},
\eqref{NeumEstInhom},
and \eqref{NeumEstHom}
that
\begin{equation}
\label{C2alpha_by_C0}
\nm{w^{(n)}}_{C^{2,\alpha}(dt^2+d\vartheta^2)}
\leq
C\bigl(\gamma(s+1)+\nm{w^{(n)}}_{C^0}\bigr)
\end{equation}
for 
$n$ and $s$ 
sufficiently large
in terms of the bounds assumed
on the functions
$c_{n,s}$, $d_{n,s}$, and $e_{n,s}$,
as well as constants,
which can be chosen uniformly,
that appear in local Schauder estimates
on $R_{n,s}$.
If we exploit \eqref{C2alpha_by_C0}
in \eqref{NeumEstInhom}
we get
\begin{equation}
\label{NeumEstInhomImproved}
\nm{
  (\partial_t w^{(n)})(\sqrt{n},\cdot)
}_{C^{1,\alpha}(d\vartheta^2)}
\leq
Ce^{-\sqrt{n}/4}
  \bigl(\gamma(s+1)+\nm{w^{(n)}}_{C^0}\bigr),
\end{equation}
once again for $n$ and $s$
assumed large enough in terms
of absolute constants.

We next decompose $w^{(n)}$ into
\begin{equation*}
w^{(n)}_0
  \vcentcolon=
  \frac{1}{\pi}
    \int_{-\pi/2}^{\pi/2}
      w^{(n)}(\cdot,\vartheta)
      \, d\vartheta,
\qquad
w^{(n)}_\perp
  \vcentcolon=
  w^{(n)}-w^{(n)}_0.
\end{equation*}
From \eqref{PoissonOnRect},
\eqref{C2alpha_by_C0},
and
item \ref{transr_cat_lap}
of Lemma \ref{transitionEstimates}
we obtain
\begin{equation}
\label{invariant_source_est}
\begin{aligned}
&\partial_t^2 w^{(n)}_0
  =
  a_{n,s}^0 e^{-t/4}
  +b_{n,s}^0 n^{-2}
  +c_{n,s}^0 n^{-1}\partial_t w^{(n)}_0,
\\
&\mbox{with}
\quad
\frac{
  \nm{a_{n,s}^0}_{C^0}
  +\nm{b_{n,s}^0}_{C^0}
}
{\gamma(s+1)+\nm{w^{(n)}}_{C^0}}
+\nm{c_{n,s}^0}_{C^0}
\leq
C
\end{aligned}
\end{equation}
and
\begin{equation}
\label{higher_source_est}
\nm{
  \Delta_{dt^2+d\vartheta^2}w^{(n)}_\perp
}_{C^0}
\leq
C\Bigl(e^{-s/4}+n^{-1/2}\Bigl)
  \bigl(\gamma(s+1)+\nm{w^{(n)}}_{C^0}\bigr).
\end{equation}
For \eqref{invariant_source_est}
we have in particular integrated
\eqref{PoissonOnRect}
in $\vartheta$,
making use
of the $\vartheta$-invariance
(see item (v) of Lemma \ref{transitionEstimates})
of the coefficients of the $n^{-1}\partial_t$
and $n^{-1/2}\partial_\vartheta^2$ terms
and observing that
the $n^{-1/2}\partial_\vartheta^2$
term integrates to zero
because of
\eqref{NeumEstHom};
for \eqref{higher_source_est}
we have made use of the fact that
$
 \nm{
   \Delta_{dt^2+d\vartheta^2}
     w^{(n)}_\perp
 }_{C^0}
 \leq
 2\nm{
  \Delta_{dt^2+d\vartheta^2}
  w^{(n)}
 }_{C^0}
$
and then appealed to \eqref{PoissonOnRect}.

To complete the analysis
we will need some 
basic estimates
for $\Delta_{dt^2+d\vartheta^2}=\partial_t^2+\partial_\vartheta^2$
on $R_{n,s}$.
For any bounded (real-valued) function $f$
on $R_{n,s}$
and for each non-negative
integer $\kappa$
let us define
on
$\IntervaL{s,\sqrt{n}}$
the Fourier coefficients
$f_\kappa$ by
\begin{equation*}
f_\kappa(t)
\vcentcolon=
\begin{cases}
\frac{1}{\pi}
  \int_{-\pi/2}^{\pi/2}
  f(t,\vartheta) \, d\vartheta
  &\mbox{for $\kappa=0$}
\\[1ex]
\frac{2}{\pi}
  \int_{-\pi/2}^{\pi/2}
  f(t,\vartheta)
  \cos \kappa(\vartheta-\pi/2)
  \, d\vartheta
  &\mbox{for $\kappa>0$}.
\end{cases}
\end{equation*}
Then the Fourier coefficients of any
$u \in C^2(R_{n,s},dt^2+d\vartheta^2)$ satisfying $(\partial_\vartheta u)=0$ at $\vartheta=\pm \pi/2$
admit the representations
\begin{align}
\notag
u_0(t)
& =
u_0(s)
  +
  (\partial_t u_0)(\sqrt{n}) \cdot (t-s)
  +\int_s^t \int_{\sqrt{n}}^\tau
    \partial_{t}^2 u_0(\sigma)
    \, d\sigma
    \, d\tau, 
\\[.5ex]\label{zero_coeff}
& =u_0(s)
  +
  (\partial_t u_0)(\sqrt{n}) \cdot (t-s)
  +\int_s^t \int_{\sqrt{n}}^\tau
    (\Delta_{dt^2+d\vartheta^2}u)_0 (\sigma)
    \, d\sigma
    \, d\tau,
\\[2ex]
\notag
u_{\kappa \neq 0}(t)
  &=
  \frac{u_\kappa(s)}{\cosh \kappa(\sqrt{n}-s)}
    \cosh \kappa(t-\sqrt{n})
  +\frac{(\partial_t u_\kappa)(\sqrt{n})}
    {\kappa \cosh \kappa(\sqrt{n}-s)}
    \sinh \kappa(t-s)
  \\[.5ex]\label{nonzero_coeffs}
  &\hphantom{{}={}}
  -\frac{\cosh \kappa(t-\sqrt{n})}
    {\kappa \cosh \kappa(\sqrt{n}-s)}
  \int_s^t (\Delta_{dt^2+d\vartheta^2}u)_\kappa(\tau)
    \sinh \kappa(\tau-s) \, d\tau
\\[.5ex]\notag
  &\hphantom{{}={}}
  -\frac{\sinh \kappa(t-s)}
    {\kappa \cosh \kappa(\sqrt{n}-s)}
  \int_t^{\sqrt{n}}
    (\Delta_{dt^2+d\vartheta^2}u)_\kappa(\tau)
    \cosh \kappa(\tau-\sqrt{n}) \, d\tau.
\end{align}
In particular
\eqref{nonzero_coeffs} implies, for any $\kappa \geq 1$ the inequality
\begin{equation}
\label{nonzero_coeffs_est}
\abs{u_\kappa(t)}
\leq
\abs{u_\kappa(s)}
  +\frac{1}{\kappa}
    \abs[\big]{(\partial_t u_\kappa)(\sqrt{n})}
  +\frac{1}{\kappa^2}
    \nm[\big]{(\Delta_{dt^2+d\vartheta^2}u)_\kappa}_{C^0}. 
\end{equation}
Since $u$ is $C^2$
the Fourier series
$
 \sum_{\kappa=0}^\infty
   u_\kappa(t)
   \cos \kappa(\vartheta-\pi/2)
$
converges (at least) pointwise to $u(t,\vartheta)$; furthermore 
(again appealing to the $C^2$ assumption
in order to control the first two terms
of \eqref{nonzero_coeffs_est})
we obtain the implication \begin{equation}
\begin{gathered}
\int_{-\pi/2}^{\pi/2} u(\cdot,\vartheta) \, d\vartheta
  =
  0
\\[-.5ex]
\Downarrow
\\
\nm{u}_{C^0}
  \leq
  C
  \Bigl(
    \nm{u(s,\cdot)}_{C^2(d\vartheta^2)}
    +\nm{(\partial_t u)(\sqrt{n},\cdot)}_{C^1(d\vartheta^2)}
    +\nm{\Delta_{dt^2+d\vartheta^2}u}_{C^0}
  \Bigr).
\end{gathered}
\end{equation}
This last estimate in conjunction with
\eqref{higher_source_est},
\eqref{DirEst}, 
and \eqref{NeumEstInhomImproved} yields
\begin{equation}
\label{wnperp_est}
\nm{w^{(n)}_\perp}_{C^0}
\leq
C\left(\gamma(s+1)
  +(e^{-\sqrt{n}/4}
    +e^{-s/4}
    +n^{-1/2})
  \nm{w^{(n)}}_{C^0}\right).
\end{equation}
On the other hand,
differentiating \eqref{zero_coeff} with respect to $t$
and applying
\eqref{invariant_source_est}
and \eqref{NeumEstInhomImproved}
we find
\[
\nm{\partial_t w^{(n)}_0}_{C^0}
\leq C\bigl(\gamma(s+1)+\nm{w^{(n)}}_{C^0}\bigr)
\Bigl(e^{-\sqrt{n}/4}+e^{-s/4}+n^{-3/2}\Bigr)
+Cn^{-1/2}\nm{\partial_t w^{(n)}_0}_{C^0}
\]
and therefore,
by absorption,
\begin{equation}
\label{first_derivative_est}
\nm{\partial_t w^{(n)}_0}_{C^0}
\leq C\bigl(\gamma(s+1)+\nm{w^{(n)}}_{C^0}\bigr)\Bigl(e^{-s/4}+n^{-3/2}\Bigr)
\end{equation}
for $n$ sufficiently large
in terms of $s$ and
the constants appearing in the above estimate.
Feeding \eqref{first_derivative_est}
into
\eqref{invariant_source_est}
and applying the result,
along with
\eqref{DirEst}
and \eqref{NeumEstInhomImproved},
in
\eqref{zero_coeff},
we get
\begin{equation}
\label{wn0_est}
\nm{w^{(n)}_0}_{C^0}
\leq
C\left(\gamma(s+1)+
  (\sqrt{n}e^{-\sqrt{n}/4}+e^{-s/4}+n^{-1})
  \nm{w^{(n)}}_{C^0}\right).
\end{equation}
Finally, since
$
 \nm{w^{(n)}}_{C^0}
 \leq
 \nm{w^{(n)}_0}_{C^0}
   +\nm{w^{(n)}_\perp}_{C^0}
$,
estimates
\eqref{wn0_est}
and \eqref{wnperp_est}
jointly imply the desired bound on the $C^0$ norm of $w^{(n)}$
provided we first choose $s$ and then, in turn, $n$ sufficiently large,
in terms of the absolute constants
appearing in the two estimates,
to be able to absorb the
$\nm{w^{(n)}}_{C^0}$
terms appearing on their right-hand sides. This ends the proof.
\end{proof}

\begin{lemma}[Eigenvalue lower bounds on
$\two{\liptowrdom{n}}$ and $\three{\liptowrdom{m}}$]
\label{towrdomLowerBounds}
For each integer $i \geq 1$
\begin{align*}
\liminf_{n \to \infty}
  \eigenvalsym
    {\indexform{\two{\liptowrdom{n}}}_\neum}
    {i}{\cycgrp{\{z=0\}}}{\pm}
&\geq
\eigenvalsym
  {\indexform{\two{\liptowdom}}}
  {i}{\cycgrp{\{z=0\}}}{\pm},
\\[1ex]
\liminf_{m \to \infty}
  \eigenvalsym
    {\indexform{\three{\liptowrdom{m}}}_\neum}
    {i}{\cycgrp{\{y=z=0\}}}{\pm}
&\geq
\eigenvalsym{\indexform{\three{\liptowdom}}}{i}{\cycgrp{\{y=z=0\}}}{\pm}
\end{align*}
for each common choice of sign $\pm$ on both sides of each equation.
\end{lemma}

\begin{proof}
We give the proof for the $+$ choice
on both sides of the top equation,
the argument for the remaining three cases
being identical in structure to this one.
Fix $i \geq 1$,
and for each $n$ let
$\{v_j^{(n)}\}_{j=1}^i$
be an $L^2(\two{\towrdom{n}},\two{h}_n)$
orthonormal set
such that
each
$v_j^{(n)}$
is a $j$\textsuperscript{th}
$(\cycgrp{\{z=0\}},+)$-invariant
eigenfunction
of $\indexform{\two{\liptowrdom{n}}}_\neum$.
Fix $C>0$, as afforded by
Lemma \ref{UniformBoundsOnTower},
such that
\begin{equation*}
\sup_n \sup_{1 \leq j \leq i}
\left(
\nm{v_j^{(n)}}_{C^0}
  +\eigenvalsym
    {\indexform{\two{\liptowrdom{n}}}}{j}
    {\cycgrp{\{z=0\}}}{+}
\right)
\leq
C.
\end{equation*}
Given any $\epsilon>0$ (fixed from now on)
and taking $s>0$ and correspondingly $n_s>0$ large enough,
as afforded by
Lemma \ref{transitionEstimates}
and Lemma \ref{towrdomCptSpecConv}, we have
\begin{align}
\label{epsilonSmall}
\hausmeas{2}{\two{h}_n}(\transr_n(s))
&<\epsilon,
&
\eigenvalsym{\indexform{\two{\liptowdom}}}{i}{\cycgrp{{\{z=0\}}}}{+}
&<\eigenvalsym{\indexform{\two{\liptowrdom{n}}(s)}_\neum}
    {i}{\cycgrp{\{z=0\}}}{+}
+\epsilon.
\end{align}

Now, for $n>n_s$
and any $v$ in the span of $\{v_j^{(n)}\}_{j=1}^i$
we estimate
\begin{equation*}
\begin{aligned}
\nm[\Big]{v}_{L^2(\two{\towrdom{n}},\two{h}_n)}^2
 -\nm[\Big]{v|_{\two{\towrdom{n}}(s)}}_{L^2(\two{\towrdom{n}}(s),\two{h}_n)}^2
&\leq
C^2i\epsilon\nm{v}_{L^2(\two{\towrdom{n}},\two{h}_n)}^2,
\\
\nm[\Big]{
  \nabla_{\two{h}_n}v|_{\two{\towrdom{n}}(s)}
  }_{L^2(\two{\towrdom{n}}(s),\two{h}_n)}
&\leq
\nm[\Big]{
\nabla_{\two{h}_n}v
  }_{L^2(\two{\towrdom{n}},\two{h}_n)},
\\
\sk[\Big]{
    v|_{\two{\towrdom{n}}(s)},
    \Bigl(\two{\conffact}_n\Bigr)^{-2}
      \abs[\Big]{\two{A}_n}_{\two{g}_n}^2
      v|_{\two{\towrdom{n}}(s)}
  }_{L^2(\two{\towrdom{n}}(s),\two{h}_n)}
&\geq
\sk[\Big]{
    v,
    \Bigl(\two{\conffact}_n\Bigr)^{-2}
      \abs[\Big]{\two{A}_n}_{\two{g}_n}^2v
  }_{L^2(\two{\towrdom{n}},\two{h}_n)}
  \\
  &\hphantom{{}\geq{}}
  -2C^2i\epsilon\nm{v}_{L^2(\two{\towrdom{n}},\two{h}_n)}^2,
\end{aligned}
\end{equation*}
where for the last inequality
we have used the 
fact that on $\transr_n(s)$
the potential function appearing here
is bounded above by $2$,
as is obvious from 
inspection of
\eqref{towrCutoffsConfFactsConfMets}.

We conclude
that for all $n>n_s$
the set
$\{v^{(n)}_j|_{\two{\towrdom{n}}(s)}\}_{j=1}^i$
is linearly independent,
and for all $v$ as above we have
\begin{equation*}
\frac{
  \indexform{\two{\liptowrdom{n}}(s)}_\neum
    \left(
    v|_{\two{\towrdom{n}}(s)},
    v|_{\two{\towrdom{n}}(s)}
    \right)
}
{\nm[\big]{v|_{\two{\towrdom{n}}(s)}^2}_{L^2(\two{\towrdom{n}}(s),\two{h}_n)}^2}
\leq
\frac{
  \eigenvalsym
    {\indexform{\two{\liptowrdom{n}}}_\neum}{j}
    {\cycgrp{\{z=0\}}}{+}
  +2C^2i\epsilon
}
{1-C^2i\epsilon}
\end{equation*}
and so
by virtue of the min-max characterization it follows that
\eqref{symminmax}
of the eigenvalues
\begin{equation*}
\eigenvalsym
    {\indexform{\two{\liptowrdom{n}}(s)}_\neum}{j}
    {\cycgrp{\{z=0\}}}{+}
\leq
\frac{
  \eigenvalsym
    {\indexform{\two{\liptowrdom{n}}}_\neum}{j}
    {\cycgrp{\{z=0\}}}{+}
  +2C^2i\epsilon
}
{1-C^2i\epsilon}
\end{equation*}
for all $n>n_s$
and $1 \leq j \leq i$.
Thus, using the second inequality
in \eqref{epsilonSmall},
we get in particular
\begin{equation*}
  \eigenvalsym
    {\indexform{\two{\liptowdom}}}
    {i}{\cycgrp{{\{z=0\}}}}{+}
\leq
\frac{
  \eigenvalsym
    {\indexform{\two{\liptowrdom{n}}}_\neum}{i}
    {\cycgrp{\{z=0\}}}{+}
  +2C^2i\epsilon
}
{1-C^2i\epsilon}
+\epsilon
\end{equation*}
for all $n>n_s$.
The claim now follows,
since this inequality holds
for all $\epsilon>0$,
with $C$ independent of $\epsilon$ and $n$. 
\end{proof}

By combining 
Lemma \ref{towrdomUpperBounds}
with Lemma \ref{towrdomLowerBounds} we immediately derive the following conclusion.

\begin{corollary}[Eigenvalues on
$\two{\liptowrdom{n}}$ and $\three{\liptowrdom{m}}$]
\label{towrdomEvals}
For each integer $i \geq 1$
\begin{align*}
\lim_{n \to \infty}
\eigenvalsym
  {\indexform{\two{\liptowrdom{n}}}_\neum}{i}
  {\cycgrp{\{z=0\}}}{\pm}
&=
\eigenvalsym{\indexform{\two{\liptowdom}}}{i}{\cycgrp{\{z=0\}}}{\pm},
\\[1ex]
\lim_{m \to \infty}
\eigenvalsym
  {\indexform{\three{\liptowrdom{m}}}_\neum}
  {i}{\cycgrp{\{y=z=0\}}}{\pm}
&=
\eigenvalsym{\indexform{\three{\liptowdom}}}{i}{\cycgrp{\{y=z=0\}}}{\pm},
\end{align*}
for each common choice of sign $\pm$ on both sides of each equation.
\end{corollary}

\begin{corollary}
[Equivariant index and nullity on $\two{\towr_n}$ and $\three{\towr_m}$]
\label{towrIndexAndNullity}
There exist $n_0, m_0>0$
such that
we have the following
indices and nullities
for all integers $n>n_0$
and $m>m_0$.
\[
\begin{array}{c|c||c|c|}
S & \grp
  & \equivind{\grp}(\indexform{S}_\neum)
    & \equivnul{\grp}(\indexform{S}_\neum) \\
\hline\hline
\two{\towr_n}  & \pri_{n} & 1 & 0 \\
\hline
\three{\towr_m} & \apr_{m+1} & 1 & 0 \\
\hline
\end{array}
\]
Additionally,
still assuming $m>m_0$
we have the upper bound
\begin{equation*}
\equivind{\pyr_{m+1}}\Bigl(\indexform{\three{\towr_m}}_\neum\Bigr)+\equivnul{\pyr_{m+1}}\Bigl(\indexform{\three{\towr_m}}_\neum\Bigr)
\leq
3.
\end{equation*}
\end{corollary}

\begin{proof}
All claims follow from the conjunction of
Lemma \ref{dom_red_ext}
(to reduce to the appropriately even
and odd indices and nullities on 
$n^{-1}\two{\towrdom{n}}$ and $(m+1)^{-1}\three{\towrdom{m}}$
with Neumann boundary data),
Proposition \ref{conformalInvariance}
(to dispense with the above scale factors
$n, m+1$ and, more substantially,
to pass from the natural metric
to $\two{h}_n$ or $\three{h}_m$),
Lemma \ref{towrdomEvals}
(to reduce to the appropriate indices
and nullities of
$\two{\liptowdom}$
and
$\three{\liptowdom}$
),
and finally Lemma \ref{towIndexAndNullity}
(which provides these last quantities).
\end{proof}

\subsection{Proofs of Theorem \ref{cswEquivariantIndexAndNullity} and \ref{cswIndexBounds}}

The following statement collects, from the broader analysis conducted in the previous section, those conclusions we shall need to prove the two main results stated in the introduction.

\begin{corollary}[Equivariant index and nullity 
upper bounds for $\threebc_m$ and $\zerogen_n$]\label{cor:EquivUpperBound}
There exists $m_0, n_0>0$ such that for all integers $m>m_0$ and $n>n_0$
we have the bounds
\begin{alignat*}{3}
&\equivind{\apr_{m+1}}(\threebc_{m})
&&+\equivnul{\apr_{m+1}}(\threebc_{m})
&&\leq2,
\\[.5ex]
&\equivind{\pyr_{m+1}}(\threebc_{m})
&&+\equivnul{\pyr_{m+1}}(\threebc_{m})
&&\leq 6,
\\[.5ex]
&\equivind{\pri_n}(\zerogen_n)
&&+\equivnul{\pri_n}(\zerogen_n)
&&\leq2.
\end{alignat*}
\end{corollary}

\begin{proof}
We apply item \ref{mrUpper}
of Proposition \ref{mr}, for the partition ``into building blocks'' defined in Section \ref{subs:deconstruction} (cf. Figure \ref{fig:decomposition}),
in conjunction with
Lemma \ref{catrAndDiscrIndexAndNullity}
and Corollary \ref{towrIndexAndNullity}
for the ancillary estimates for the index and nullity of the various blocks. We find that the three index-plus-nullity sums appearing in the statement
are respectively bounded above by
\begin{equation*}
\begin{gathered}
  \equivind{\pyr_{m+1}}\Bigl(\indexform{\three{\catr_m}}_\neum\Bigr)
    +\equivind{\apr_{m+1}}\Bigl(\indexform{\three{\towr_m}}_\neum\Bigr)
    +\left[
      \equivind{\apr_{m+1}}\Bigl(\indexform{\three{\discr_m}}_\neum\Bigr)
      +\equivnul{\apr_{m+1}}\Bigl(\indexform{\three{\discr_m}}_\neum\Bigr)
    \right]
    \leq
    1+1+0
    =
    2,
  \\
  2\equivind{\pyr_{m+1}}\Bigl(\indexform{\three{\catr_m}}_\neum\Bigr)
    +\equivind{\pyr_{m+1}}\Bigl(\indexform{\three{\towr_m}}_\neum\Bigr)
    +\left[
      \equivind{\pyr_{m+1}}\Bigl(\indexform{\three{\discr_m}}_\neum\Bigr)
      +\equivnul{\pyr_{m+1}}\Bigl(\indexform{\three{\discr_m}}_\neum\Bigr)
    \right]
    \leq
    2+3+1
    =
    6,
  \\
   \equivind{\pyr_n}\Bigl(\indexform{\two{\catr_m}}_\neum\Bigr)
    +\left[
      \equivind{\pri_n}\Bigl(\indexform{\two{\towr_m}}_\neum\Bigr)
      +\equivnul{\pri_n}\Bigl(\indexform{\two{\towr_m}}_\neum\Bigr)
    \right]
    \leq
    1+1
    =
    2.
\end{gathered}
\end{equation*}
The first term in the first line
arises as an upper bound for the
$\apr_{m+1}$-equivariant index
of $\three{\catr_m} \cup \refl_{\{y=z=0\}}\three{\catr_m}$
subject to the natural (free boundary) Robin condition
on the portion of its boundary in $\Sp^2$
and subject to the homogeneous Neumann boundary condition
on the remainder of the boundary.
To obtain this upper bound
we have used the fact that
a function on $\three{\catr_m} \cup \refl_{\{y=z=0\}}\three{\catr_m}$
(a disjoint union, with each annulus disjoint from $\{z=0\}$)
is $\apr_{m+1}$-equivariant
if and only if
its restriction to $\three{\catr_m}$
is $\pyr_{m+1}$-equivariant
and it is odd with respect to any one (so all) of the $m+1$
reflections through horizontal lines in $\apr_{m+1}$.
The first term of the final line
is obtained in similar fashion.
\end{proof}

So, we are in position to fully determine the (maximally) equivariant index and nullity for the two families of free boundary minimal surfaces we constructed in \cite{CSWnonuniqueness}.

\begin{proof}[Proof of Theorem \ref{cswEquivariantIndexAndNullity}]
We combine the upper bounds
of the preceding corollary
with the lower bounds
from our earlier paper \cite{CSWnonuniqueness}, specifically with the content of Proposition 7.1 (cf. Remark 7.5) therein for what pertains to the index. At that stage, the fact that both nullities are zero then follows from the first and third inequality in Corollary \ref{cor:EquivUpperBound}.
\end{proof}

Finally, we can obtain the absolute estimates on the Morse index of the same families.

\begin{proof}[Proof of Theorem \ref{cswIndexBounds}]
The lower bounds have already been established: specifically, for $\threebc_m$ this is just part of Proposition \ref{pro:CheapLowerBods}, while for $\zerogen_n$ it follows from just combining Proposition \ref{pro:CheapLowerBods} with Proposition~\ref{pro:OddLowerBods}.
For the upper bound
we can apply the Montiel--Ros argument
making use of the equivariant upper bounds above, as we are about to explain.
In the case of $\zerogen_n$, the $\pri_n$-equivariant upper bound on the Morse index (and nullity) is equivalent to an upper bound on the index and nullity on each domain $\Omega^n_i=\zerogen_n\cap W_i$ where $W_1,\ldots, W_{4n}$ are the open domains defined, in $\B^3$, by the horizontal plane $\left\{z=0\right\}$ together with the $n$ vertical planes passing through the origin and having equations $\theta=\pi/(2n)+i\pi/n$, $i=0,1,\ldots, n-1$ (in the cylindrical coordinates defined at the beginning of Section \ref{sec:FBMSgen}), subject to Neumann conditions in the interior boundary as prescribed by Lemma \ref{dom_red_ext}. Thus the conclusion comes straight by appealing to Corollary \ref{cor:MontielRosHuman} given the third displayed equation of Corollary \ref{cor:EquivUpperBound}. Similarly, for $\threebc_m$ we can interpret the second inequality in the statement of Corollary \ref{cor:EquivUpperBound} as a statement on the index and nullity of the portions of surfaces that are contained in any of the $2(m+1)$ sets obtained by cutting with the $m+1$ vertical planes passing through the origin and having equations $\theta=\pi/(2(m+1))+i\pi/(m+1)$, $i=0,1,\ldots, m$, again subject to Neumann conditions. This completes the proof.
\end{proof}

\setlength{\parskip}{1ex plus 1pt minus 1pt}

\printaddress

\end{document}